\documentclass[12pt,reqno]{amsart}

\usepackage{cite}
\usepackage[T1]{fontenc}
\usepackage[utf8]{inputenc}
\usepackage{lmodern,amsmath,amssymb,amsthm,mathtools,mathrsfs}
\mathtoolsset{showonlyrefs=true}

\usepackage{array,booktabs,longtable,caption}
\usepackage[a4paper,margin=2.1cm]{geometry}
\usepackage{microtype}
\usepackage[colorlinks=true,linkcolor=blue,citecolor=blue,urlcolor=blue]{hyperref}
\hypersetup{pdftitle={Corrections and additions to the odd-prime hit problem},
 pdfauthor={Dang Vo Phuc},
 pdfsubject={Steenrod operations and the third algebraic transfer},
 pdfkeywords={Steenrod algebra, hit problem, odd prime, algebraic transfer}}
\numberwithin{equation}{section}
\theoremstyle{plain}
\newtheorem{thm}{Theorem}[section]
\newtheorem{prop}[thm]{Proposition}
\newtheorem{lem}[thm]{Lemma}
\newtheorem{corl}[thm]{Corollary}
\theoremstyle{definition}

\theoremstyle{remark}
\newtheorem{nx}[thm]{Remark}
\newcommand{\Fp}{\mathbb F_p}
\newcommand{\NN}{\mathbb N}
\newcommand{\A}{\mathscr A_p}
\newcommand{\Pp}{\mathscr P}
\newcommand{\Hit}{\mathcal H}
\newcommand{\QP}{QP}
\newcommand{\QH}{QH}
\newcommand{\digit}{\preceq_p}
\newcommand{\ds}{\operatorname{s}_p}
\DeclareMathOperator{\GL}{GL}
\DeclareMathOperator{\Ext}{Ext}
\DeclareMathOperator{\Ann}{Ann}
\DeclareMathOperator{\im}{im}
\DeclareMathOperator{\rank}{rank}
\DeclareMathOperator{\Span}{span}
\DeclareMathOperator{\diag}{diag}
\DeclareMathOperator{\Tr}{Tr}
\newcommand{\doi}[1]{\href{https://doi.org/#1}{\nolinkurl{doi:#1}}}
\title[Corrections to the hit problem at odd primes]{Corrections and additions to ``A note on the hit problem for the polynomial algebra in the case of odd primes and its application'' [Rev. R. Acad. Cienc. Exactas F\'is. Nat. Ser. A Mat. \textbf{118} (2024), article 22]}
\author{\DJ\d{\u a}ng V\~o Ph\'uc}
\address{Department of Mathematics, FPT University, An Phu Thinh New Urban Area, Quy Nhon, Vietnam}
\email{dangphuc150488@gmail.com}
\thanks{ORCID: \url{https://orcid.org/0000-0002-6885-3996}.}
\keywords{Steenrod algebra, hit problem, odd prime, algebraic transfer, general linear group invariants}
\subjclass[2020]{55S10, 55T15, 55R40, 13A50}
\date{}
\begin{document}
\begin{abstract}
In this work, we correct results in the published paper [Rev. R. Acad. Cienc. Exactas F\'is. Nat. Ser. A Mat. \textbf{118} (2024), article 22]. We determine the non-hit kernel of the first reduced power in polynomial degree $p+1$, identify the published Frobenius construction as zero on indecomposables, and replace the multiplier assertion by a theorem with explicit degree hypotheses. We also classify the monomial multipliers of Frobenius that preserve hits in one variable. For the third algebraic transfer, we correct the target dimensions and establish surjectivity throughout the four principal degree families. Explicit invariant calculations and Minami's degree reduction give isomorphisms in seven infinite families at $p=3$ and $p=5$, including a family with an initial change from dimension one to dimension two. For every odd prime, compatibility with the Frobenius operation on $\Ext$ gives an exact sequence comparing successive transfer kernels and forces a nonzero invariant extraction kernel at every second-family dimension jump. We compute the first cohomology of a nonzero extraction kernel and determine the exact obstruction to injectivity from rank-two products in the first and fourth families. Coefficient extraction and changes of coordinates give an additional infinite isomorphism family and a spanning result for annihilated polynomial homology. At each extraction isomorphism, the quotient by all rank-two products is preserved as a general linear group module. We distinguish these results from the unresolved invariant dimensions at arbitrary normalized index differences, and give the precise limitations of composition-factor counts and Dickson multiplication. An appendix supplies explicit representatives for all the nonzero rank-two invariant lines in Crossley's dimension classification.
\end{abstract}
\maketitle
\enlargethispage{6pt}
\raggedbottom
\setcounter{tocdepth}{1}
\tableofcontents

\section{Introduction}\label{sec:intro}

Let $p$ be an odd prime. The hit problem for a polynomial algebra
over $\Fp$ asks for a minimal set of generators as a module over
the mod~$p$ Steenrod algebra. Equivalently, it asks for a basis of
the quotient by the images of positive-degree Steenrod operations.
The problem combines the arithmetic of binomial coefficients with
the relations imposed by the Cartan formula. Its connection with
the algebraic transfer also requires the action of the general
linear group on this quotient. At an odd prime, the cohomology of
an elementary abelian group contains exterior generators as well
as polynomial generators. The Bockstein relations must therefore
be included when passing from the polynomial hit problem to the
coinvariant space that forms the domain of the transfer.

The action used here is the unstable action described by Steenrod
and Epstein~\cite{SteenrodEpstein}; Milnor's description of the
Steenrod algebra and its dual~\cite{Milnor} provides its algebraic
setting. Chen and Shen~\cite{ChenShen1992} studied the action of
reduced powers on polynomial algebras, and
Meyer~\cite{Meyer2001} related hit polynomials to excess.
For two polynomial variables at odd primes,
Crossley determined the annihilated homology~\cite{Crossley1996}
and monomial bases for the indecomposables~\cite{Crossley1999Bases}.
His bounded-type results~\cite{Crossley1998} describe a further
structural property of these modules. These works supply the
odd-primary background for the polynomial calculations below;
the rank-one criterion is recalled from
\cite[Theorems~1 and~$1'$, p.~327]{ChenShen1992}, with
the grading and the degree-zero endpoint made explicit.

Tanay and \"Oner~\cite[Theorems~2.1--2.3]{TanayOner2013}
recorded the Frobenius compatibility of reduced powers and
the corresponding product formula. Their bounded multiplication
formulas and their applications to specified hit expressions
\cite[Theorem~2.12 and Propositions~3.2--3.3]{TanayOner2013}
are directly relevant to the constructions considered here.
We use these correct identities with attribution and without
repeating their proofs. Two later summation formulas in that
article omit necessary zero indices; we give the complete
ranges and verify them independently by the digit formula.
The numerical conclusions of its examples remain valid.
The distinction between a prescribed hit expression and all
hits in a homogeneous component is essential to the descent
arguments below.

The monographs of Walker and Wood~\cite{WalkerWood1,WalkerWood2}
develop the polynomial methods and the representation theory of
the mod~$2$ hit problem. Their discussion of Frobenius constructions,
in particular~\cite[\S6.5]{WalkerWood1}, emphasizes the need to check
that a proposed map preserves the subspace of hit polynomials.
At odd primes, Minami's downward operation and its induced maps
on indecomposables~\cite[Proposition~1.7, Theorem~1.8 and
Corollary~1.9]{Minami} provide degree-dependent reduction results.
The hypotheses of those results are essential when a Frobenius
operation is used to compare different degrees. Another related
problem is the determination of Steenrod generators for the
algebra of symmetric polynomials $H^*(BU(2);\Fp)$, studied by Pengelley and
Williams~\cite{PengelleyWilliams2013}; its relations belong to that
Steenrod module and require a separate comparison with the full
polynomial algebra.

Singer's algebraic transfer~\cite{Singer} relates annihilated
homology to the cohomology of the Steenrod algebra. Crossley proved
that the odd-primary transfer is an isomorphism in ranks one and
two, and that the transfer is multiplicative
\cite[Theorem~1 and Proposition~2]{Crossley1999Transfer}.
Together with Aikawa's computation in homological degree three
\cite[Theorem~8.1 and Tables~8.1--8.2]{Aikawa}, these results determine the image on the
decomposable classes considered in this paper. They do not supply
an upper bound for the dimension of the domain of the third transfer. Such
a bound requires an invariant-space calculation in the full
cohomology quotient. Crossley's representation-theoretic calculation
in rank two distinguishes simple factors occurring as submodules
from those occurring only as quotients
\cite[\S2, Table~2 and Comment~2]{Crossley1999Transfer}.
This distinction remains essential in rank three: fixed vectors
are determined by the trivial summands of the socle, while
composition multiplicities give an upper bound
\cite[\S6.3]{Webb2016}.

Our article~\cite{PhucOriginal}, its published
correction~\cite{PhucCorrection}, and the comments of
Sum~\cite{Sum2025} give the context for the present corrections.
Those comments raise questions of attribution as well as
questions about the validity of the polynomial and transfer
arguments. The classical rank-one criterion is retained here
as a prerequisite with a complete proof. Corollary~2.2 of~\cite[pp.~4--5]{PhucOriginal}
already cited the full-cohomology calculation
of Crossley~\cite[Theorem~2]{Crossley1999Transfer}; the
reason for supplying its detailed verification and the
necessary degree-one correction are explained in
Subsection~\ref{subsec:transfer-preliminaries}.
The attribution of the separate rank-two example is also specified
at its calculation below. These attribution questions are
distinct from whether the subsequent constructions preserve
the hit subspace.

The implication in Theorem~2.4 of~\cite{PhucOriginal}
from annihilation by the first reduced power to membership
in its image is false. Theorem~2.5 of~\cite[p.~5]{PhucOriginal} follows from
instability and the cited Frobenius--Cartan identity, and the
precise extension below includes the constant-term condition.
The construction in~\cite[Theorem~2.6, p.~6]{PhucOriginal}
has its whole image in the
hit subspace, so its induced map is well defined and zero.
The multiplier in~\cite[Theorem~2.7, p.~6]{PhucOriginal}
does not preserve hits, even
when the input is an image of $\Pp^1$ satisfying the
restricted premise printed in that theorem.

For Theorem~2.8 of~\cite[pp.~6--7]{PhucOriginal},
the assertion concerning the invariant-space dimension fails
in an infinite subfamily. The corresponding target
degrees contain two independent entries of Aikawa's basis
\cite[Table~8.1]{Aikawa}, and multiplicativity of the
transfer places both in its image. This disproves the
assertion that the domain of the transfer is one-dimensional, but does not
by itself disprove that the transfer is an isomorphism.
The argument also applies a degree reduction without
verifying its numerical hypothesis and omits the
Bockstein relations in passing to full cohomology.
The surjectivity conclusion of
Corollary~2.11 of~\cite[pp.~7--8]{PhucOriginal} follows independently from the low-rank
transfer theorem and the complete target-basis calculation.
We retain that conclusion with its proof separated from
the disputed calculation of the invariant-space dimension. In this paper, we give corrected statements and proofs in the
order of these polynomial and transfer constructions. The
kernel of the first reduced power is determined in the degree
containing the counterexample to the kernel--image assertion.
The Frobenius constructions are reformulated with consistent
polynomial degrees, and the general multiplier assertion is
replaced by a result with explicit restrictions on the residual
multiplier and the input degree.
The two counterexamples also distinguish the roles of those
restrictions. A separate subsection gives further polynomial consequences,
including the complete one-variable multiplier criterion.
For the transfer, we determine the target dimensions in all four
principal families and prove surjectivity in every one of
them. We verify the six further surjectivity families
and retain the degree-zero isomorphism.
The initial invariant calculations and Minami's degree reduction,
together with the rank-two coefficient argument below, establish
isomorphisms for every integer $r\geq0$ in the families
\[
\begin{split}
 p=3:\quad &n=16\cdot3^r-3,\ 24\cdot3^r-3,\ 32\cdot3^r-3,\\
           &\hspace{16mm}40\cdot3^r-3,\ 136\cdot3^r-3,\
             364\cdot3^r-3,\\
 p=5:\quad &n=48\cdot5^r-3,
\end{split}
\]
where $n$ is total cohomological degree. Both transfer spaces
have dimension two when $p=3$, $n=40\cdot3^r-3$ and $r\geq1$,
and dimension one in all other cases in this display.
The isomorphisms are established in
Theorem~\ref{thm:infinite-isomorphisms},
Corollary~\ref{cor:extended-dimension-jump-family}, and
Theorem~\ref{thm:additional-infinite-transfer-family}.
Explicit invariant representatives are given at the corresponding
calculations and in Corollary~\ref{cor:infinite-invariant-bases}.

More generally, Theorem~\ref{thm:uniform-common-shift} verifies the
degree-reduction hypothesis from the second simultaneous shift
onwards in all four families, for every odd prime and every
admissible choice of the indices. In the first family it applies
from the first shift. The normalized degrees and, in the last
three families, their first shifts therefore remain to be
controlled for each choice of relative index differences.
Proposition~\ref{prop:transfer-forces-connecting-zero} gives an
exact sequence comparing successive transfer kernels. It bounds
the connecting homomorphism for extraction by the lower-degree
transfer kernel and bounds the invariant extraction kernel below
by the increase in the target dimension. In particular, every
second-family dimension jump forces a nonzero invariant kernel,
for every odd prime. In an induction starting from an injective
transfer, the connecting homomorphism is consequently zero. The exact
condition for the next transfer to be an isomorphism is then
that the invariant extraction kernel have the same dimension
as the increase in the target. At the step from degree $37$
to degree $117$ for $p=3$, this invariant kernel has dimension
one; Proposition~\ref{prop:boundary-kernel-first-cohomology}
also computes its first group cohomology as zero.

Rank-two products give another description of the remaining
obstruction. We form their full $G_3$-span in the annihilated
polynomial homology. Proposition~\ref{prop:block-product-transfer}
identifies the transfer kernel in the first and fourth
families with the coinvariants of the quotient by this span,
after the required determinant twist. Thus injectivity in
these families is equivalent to a specific invariant
vanishing condition. In polynomial degree $65$ at $p=3$,
thirteen explicitly displayed translated functionals span
the entire dual quotient. Proposition~\ref{prop:block-span-extraction} proves more generally
that extraction isomorphisms preserve the quotient by all rank-two
products as a $G_3$-module. The spanning assertion therefore
persists through the associated common-shift family. Five rank-two
coefficient equations in polynomial degree $9$ prove the
additional family $n=24\cdot3^r-3$. These arguments use
functionals that annihilate all rank-two hit relations,
not arbitrary coefficients of non-hit monomials.
We also prove that coordinate extraction alone can miss
the invariance equations, and that a degree-matching
Dickson power can fail to preserve hits between admissible
degrees of the first family. The uniform invariant
vanishing condition for unrestricted index differences
and all odd primes is not established here.
The appendix supplements
Crossley's invariant-dimension theorem
\cite[Corollary~1]{Crossley1999Transfer} with explicit bases of
the rank-two invariant spaces.

The proofs use the complete Cartan formula, Lucas's congruence,
and the compatibility of reduced powers with Frobenius. A
one-variable specialization detects the failure of the
unrestricted multiplier. To compute the invariant space dual to the domain of the transfer, the
incoming
Bockstein images are included before decomposing by exterior
degree. A central-character argument then reduces the relevant
rank-three invariant spaces to the top exterior summand.
Diagonal weights restrict the possible representatives, and
the actions of permutations and transvections impose the
remaining invariance conditions. The comparison of the targets
uses all entries of the basis in
\cite[Table~8.1]{Aikawa} in a given internal degree,
including distinct entries with the same degree. To pass from
the initial invariant spaces to infinite families, we use the
equivariance of Minami's extraction and verify its numerical
hypothesis by explicit base-$p$ digit sums. The dual determinant
character is retained in every degree. Four first steps fail
the sufficient hypothesis and are computed separately; every
later step is an isomorphism of the relevant $G_3$-modules. For the general families, explicit base-$p$ expansions give the same
uniform starting bound. The invariant-kernel calculation retains
the connecting map in the group-cohomology exact sequence
\cite[Chapter~III]{Brown1982}. Minami's compatibility theorem
\cite[Theorem~4.4]{Minami} then compares its dual with the
Frobenius operation on the complete target basis. In the
explicit first-cohomology calculation, diagonal weights and
two transvections determine every cocycle after restriction
to a Borel subgroup. For the rank-two construction, the
Cartan formula makes coefficient extraction descend to
indecomposables; changes of coordinates give the required
functionals on the full dual. Multiplicativity of the
low-rank transfer and the degree classification in
\cite{Crossley1999Transfer} identify the contribution of
their products to rank-three coinvariants.
For the Dickson comparison, we use
the classical invariant algebra~\cite{Dickson1911} and detect the
failure of hit preservation by a coefficient functional that
annihilates every positive reduced-power image.
In rank two, the quotient bases and dimension theorem in
\cite[Theorem~2, Table~1 and Corollary~1]{Crossley1999Transfer}
reduce the
determination of explicit invariant generators to precise
coefficient identities.

{\bf Organization of the paper.}
Section~\ref{sec:main} treats the polynomial results.
Subsection~\ref{subsec:polynomial-preliminaries} fixes the grading,
the hit subspace, the digit notation, and the identities used
throughout the polynomial arguments, distinguishing the known
Frobenius identities from the corrected summation ranges. Subsection~\ref{subsec:first-power}
gives the counterexample to the kernel--image assertion and
determines the kernel in its degree.
Subsection~\ref{subsec:frobenius-maps} reformulates the two
Frobenius constructions, and
Subsection~\ref{subsec:multiplier-correction} gives the
counterexample and replacement for the multiplier theorem.
The additional polynomial results are collected in
Subsection~\ref{subsec:further-polynomial}.
Section~\ref{sec:transfer} is devoted to the third algebraic
transfer. Its first subsection introduces the full cohomology
quotient, the group actions, the duality between annihilated homology
coinvariants and cohomology invariants, and the invariant
equations used in all subsequent calculations. The next
subsection locates the grading and hypothesis failures in the
published reduction, corrects the target dimensions in the four
degree families, and proves their surjectivity statement.
Subsection~\ref{subsec:additional-surjectivity} verifies all
basis entries in the six further surjectivity families
and proves the independent degree-zero
isomorphism. Subsection~\ref{subsec:rankthree-calculations}
gives explicit rank-three representatives, proves the smallest
case directly, and determines the remaining invariant spaces
by displaying and solving the coefficient equations. The resulting invariant dimensions, together with surjectivity and the
duality established in the first subsection, give the six finite
isomorphisms. Subsection~\ref{subsec:infinite-transfer} verifies
the digit criterion uniformly in all four principal degree families,
computes the four additional starting spaces, and proves the
six infinite isomorphism statements. It also supplies explicit
invariant representatives in every degree of those families
and distinguishes this reduction on indecomposables from
stripping inside the Steenrod algebra.
Subsection~\ref{subsec:comparison-obstructions} constructs the
exact sequence comparing successive transfer kernels,
determines the invariant extraction kernel and its first
cohomology in a boundary case, and gives the exact dimension
change from degree $37$ to $117$. It also separates the
limitations of rank-one coefficient functionals from the
failure of Dickson multiplication to preserve hits.
Subsection~\ref{subsec:ranktwo-products} constructs rank-two
annihilating functionals, proves the necessity of changes of
coordinates in degree $65$, and gives the additional infinite
isomorphism family beginning in degree $21$. It then proves that extraction isomorphisms preserve the quotient
by rank-two products and identifies the remaining transfer kernel
in the first and fourth families.
Subsection~\ref{subsec:composition-factors} states the correct
composition-multiplicity criterion, distinguishes module classes
from fixed-subspace dimensions and extension data, and explains
why the usual
action of the algebraic general linear group does not
automatically descend to the hit quotient.
Remark~\ref{rem:comments-context} then distinguishes the valid
objections in~Sum \cite{Sum2025} from its inaccurate transcription,
its false proposed replacement, and the contextual qualifications
resolved by the preceding results. Appendix~\ref{app:ranktwo} determines explicit invariant bases
in rank two, distinguishing the even and odd cohomological
degrees and accounting for the Bockstein in the latter.
\section{Main theorems and their proofs}\label{sec:main}

The polynomial constructions are considered in the order in which
they enter the argument of~\cite{PhucOriginal}. We first fix the
grading and recall the elementary facts needed to test whether a
polynomial is hit. The kernel--image assertion is then replaced
by an exact calculation in its first relevant multivariable
degree. The two Frobenius constructions and the multiplier
construction are treated separately, so that the hypotheses
needed for a map on indecomposables are explicit. Additional
polynomial results are collected at the end of the section.

\subsection{Notation and preliminary results}
\label{subsec:polynomial-preliminaries}

Put $\NN=\{0,1,2,\ldots\}$. Let $p$ be an odd prime, let $h\ge1$, and let $P_h=\Fp[x_1,\ldots,x_h]$.
Each $x_i$ has polynomial degree $1$ and cohomological degree $2$.
Throughout the polynomial arguments, $(P_h)_n$ and $\deg f$
refer to polynomial degree. Put
\[
 \Hit_h=\A^+P_h,\qquad
 \QP_h=P_h/\Hit_h=\Fp\otimes_{\A}P_h,
\]
where $\A^+$ is the augmentation ideal of $\A$.
A polynomial is \emph{hit} if it belongs to $\Hit_h$; its class
in the quotient is denoted by $[f]$. The Bockstein is zero on
$P_h$. The Cartan formula and the one-variable binomial formula
give the following classical identity, also recorded in
\cite[Corollary~1.8 and Theorem~2.4, pp.~17--19]{TanayOner2013}:
\begin{equation}\label{eq:cartan}
 \Pp^s(x_1^{b_1}\cdots x_h^{b_h})
 =\sum_{c_1+\cdots+c_h=s}
 \left(\prod_{i=1}^h\binom{b_i}{c_i}\right)
 x_1^{b_1+(p-1)c_1}\cdots x_h^{b_h+(p-1)c_h}.
\end{equation}
All $b_i,c_i$ are nonnegative integers. Binomial coefficients
are taken in $\Fp$ and are zero when the lower argument exceeds
the upper argument. Thus $\Pp^s$ raises polynomial degree by
$(p-1)s$.

\begin{lem}\label{lem:hit-images}
Let $n\ge0$. Then
\begin{align*}
 (\Hit_h)_n
 &=\sum_{1\le s\le\lfloor n/p\rfloor}
   \Pp^s\bigl((P_h)_{n-(p-1)s}\bigr)\\
 &=\sum_{\substack{j\ge0\\p^{j+1}\le n}}
   \Pp^{p^j}\bigl((P_h)_{n-(p-1)p^j}\bigr).
\end{align*}
\end{lem}
\begin{proof}
The Steenrod algebra is generated by $\beta$ and the operations
$\Pp^{p^j}$ for $j\ge0$~\cite{SteenrodEpstein}.
A word containing a Bockstein acts as zero on $P_h$: every
factor to the right of the rightmost Bockstein preserves $P_h$,
and that Bockstein annihilates its value. A nonidentity word
containing only reduced powers has a leftmost positive reduced
power after identity factors are removed. Its image is contained
in the image of that operation. This proves both image-span
equalities before restricting the indices.

If a monomial of degree $d$ has a nonzero term under $\Pp^s$,
then~\eqref{eq:cartan} has $c_i\le b_i$ for every $i$ in that
term, so $s\le d$. A term in degree $n$ has
$d=n-(p-1)s$, and hence $ps\le n$. The omitted operations
therefore vanish. For $s=p^j$, the same inequality is
$p^{j+1}\le n$.
\end{proof}

For $a=\sum_i a_ip^i$ and $b=\sum_i b_ip^i$, with digits
$0\le a_i,b_i<p$, write $a\digit b$ when $a_i\le b_i$ for
every $i$. Lucas's congruence, in the form discussed
by Fine~\cite{Fine1947}, states that
\begin{equation}\label{eq:binomial-digits}
 \binom ba=\prod_i\binom{b_i}{a_i}\quad\text{in }\Fp,
 \qquad \binom ba\ne0\quad\Longleftrightarrow\quad a\digit b.
\end{equation}
We use this classical formula without a separate proof.

The rank-one classification is already contained in the digit criterion
of Chen and Shen~\cite[Theorems~1 and~$1'$, p.~327]{ChenShen1992}.
It is also recovered from the pure-variable classes in Crossley's
rank-two monomial calculation
\cite[Theorem~1.1 and Statement~1.2, p.~172]{Crossley1999Bases},
as observed in~\cite[\S2, item~1]{Sum2025}.
Here the passage from rank two to rank one uses a split inclusion:
the inclusion $\Fp[x]\hookrightarrow\Fp[x,y]$ and the retraction
defined by $y\mapsto0$ commute with Steenrod operations, and their
composite is the identity. They therefore induce a split inclusion
on indecomposables. A pure power is hit in rank two precisely when
it is hit in rank one.

Theorem~2.1 of~\cite[p.~4]{PhucOriginal} was used as a preliminary
calculation for Corollary~2.2 of~\cite[pp.~4--5]{PhucOriginal}. The polynomial exponent
list in positive degree agrees with the classical classification, but its first
equivalence must also include the constant monomial. The proof
below is included for the subsequent counterexamples and for the
full-cohomology calculation in Corollary~\ref{cor:full-rank-one};
the classification itself is not a new result. We normalize the
indices so that each allowed exponent has a unique expression.

\begin{thm}\label{thm:rank-one}
Let $d$ be a nonnegative integer.  Then the class of $x^d$ in $QP_1$
is nonzero if and only if
\[
 d+1=c p^r\qquad\text{for some }1\le c\le p-1\text{ and }r\ge0.
\]
\end{thm}
\begin{proof}
Each homogeneous component of $P_1$ is one-dimensional.  Lemma~\ref{lem:hit-images} therefore shows that $x^d$ is hit exactly when there exist
$s>0$ and $b\ge0$ such that
\[
 b+(p-1)s=d,\qquad \binom bs\not\equiv0\pmod p.
\]

First suppose $d=c p^r-1$ with $1\le c\le p-1$.
Consider any $s>0$ for which $b=d-(p-1)s\ge0$.
Let $j$ be the least position with a nonzero digit $s_j$ in $s$.
If $j\ge r$, then $s\ge p^r$ and
\[
 b\le(c-(p-1))p^r-1\le-1,
\]
contrary to the choice of $b$.  Hence $j<r$.
Modulo $p^{j+1}$ we have
\[
 b=d-(p-1)s\equiv-1+s_jp^j
 =(s_j-1)p^j+(p^j-1).
\]
Thus the $j$th digit of $b$ is $s_j-1$, which is strictly smaller
than the $j$th digit of $s$.  The digit formula for binomial coefficients
gives $\binom bs=0$ in $\Fp$.  No positive reduced power can therefore
hit $x^d$.

Conversely, suppose $d+1$ does not have exactly one nonzero base-$p$
digit.  Let $L$ be the highest nonzero digit position of $d$.
There is a position $j<L$ with $d_j\le p-2$.  Indeed, if every digit
below $L$ were $p-1$, then $d+1=(d_L+1)p^L$ would have one nonzero
digit after the possible carry when $d_L=p-1$.
Write
\[
 d=A p^{j+1}+b p^j+c,
 \quad A\ge1,\quad0\le b\le p-2,\quad0\le c<p^j.
\]
The integer
\[
 d-(p-1)p^j=(A-1)p^{j+1}+(b+1)p^j+c
\]
is nonnegative and has $j$th digit $b+1\ne0$ in $\Fp$.
Consequently
\[
 x^d=(b+1)^{-1}\Pp^{p^j}
       \bigl(x^{d-(p-1)p^j}\bigr),
\]
which is a hit expression.
\end{proof}

The Frobenius identity in the following lemma is
\cite[Theorem~2.2, p.~18]{TanayOner2013}.
The product identity is the specialization of
\cite[Theorem~2.3, p.~18]{TanayOner2013} to a multiplier
of polynomial degree less than $p$: instability annihilates
every positive operation on that multiplier whose index is
divisible by $p$. We recall these identities without repeating
their proofs.

\begin{lem}\label{lem:frobenius}
Let $f\in P_h$ and let $r,s\ge0$. Then
\[
 \Pp^s(f^{p^r})=
 \begin{cases}
 (\Pp^{s/p^r}f)^{p^r},&p^r\mid s,\\
 0,&p^r\nmid s.
 \end{cases}
\]
If $a\in P_h$ has degree less than $p$, then
\begin{equation}\label{eq:small-multiplier}
 \Pp^{ps}(af^p)=a(\Pp^s f)^p.
\end{equation}
\end{lem}

More generally, for $f,g\in P_h$ and $k,r\ge0$,
\cite[Theorem~2.3, p.~18]{TanayOner2013} gives
\begin{equation}\label{eq:frobenius-product}
 \Pp^k(gf^{p^r})
 =\sum_{\substack{i,j\ge0\\i+p^rj=k}}
       \Pp^i(g)(\Pp^jf)^{p^r}.
\end{equation}
All operation indices in this formula are nonnegative.
The condition $\deg a<p$ in our polynomial grading corresponds
to cohomological degree less than $2p$ in that paper \cite{TanayOner2013}.

Some later displayed summation bounds
in that article require zero indices and must be distinguished
from the valid identities just recalled. In
\cite[Corollary~2.7, pp.~19--20]{TanayOner2013}, the condition
$1\le s_j\le r_j$ must be replaced by $0\le s_j\le r_j$.
For example, over $\mathbb F_3$,
\[
 \Pp^3(x^3y^3)=x^9y^3+x^3y^9,
\]
whereas positive values of both summation indices cannot have sum~$1$.
Similarly, the first summation index in
\cite[Theorem~2.9, p.~21]{TanayOner2013} must be allowed to vanish:
with $p=3$, $t=q_1=r_1=1$, its input is $x^{12}$ and
\[
 \Pp^3(x^{12})=\binom{12}{3}x^{18}=x^{18}.
\]
The printed requirements $s_1\ge1$ and $9s_1+3s_2=3$
exclude this term. The following formulation records the corrected
bounds in both cases.

\begin{prop}\label{prop:tanay-indices}
Let $t,k\ge0$, let $q_1,\ldots,q_h\ge0$, and let
$0\le r_i<p$ for $1\le i\le h$. Put
$Q=\prod_i x_i^{q_i}$ and $R=\prod_i x_i^{r_i}$.
Then
\[
 \Pp^k(Q^{p^{t+1}}R^{p^t})
 =\sum_{\substack{a,b\ge0\\p^{t+1}a+p^tb=k}}
       (\Pp^aQ)^{p^{t+1}}(\Pp^bR)^{p^t}.
\]
The summation may be restricted to
$a\le\sum_iq_i$ and $b\le\sum_ir_i$.
\end{prop}
\begin{proof}
Set $A_i=p^t(pq_i+r_i)$. In the Cartan formula
\eqref{eq:cartan}, a term corresponding to indices
$c_1,\ldots,c_h$ has coefficient
$\prod_i\binom{A_i}{c_i}$. The lowest $t$ base-$p$ digits
of $A_i$ vanish. By~\eqref{eq:binomial-digits}, the coefficient
is therefore zero unless $p^t\mid c_i$ for every~$i$.
For a remaining term, write uniquely
\[
 c_i=p^t(pu_i+v_i),\qquad u_i\ge0,\quad 0\le v_i<p.
\]
A second application of~\eqref{eq:binomial-digits} gives
\[
 \binom{A_i}{c_i}
 =\binom{q_i}{u_i}\binom{r_i}{v_i}\quad\text{in }\Fp.
\]
Hence $0\le u_i\le q_i$ and $0\le v_i\le r_i$
contain every nonzero term. Its exponent on $x_i$ is
\[
 A_i+(p-1)c_i
 =p^{t+1}\bigl(q_i+(p-1)u_i\bigr)
  +p^t\bigl(r_i+(p-1)v_i\bigr).
\]
Group these terms according to $a=\sum_i u_i$ and
$b=\sum_i v_i$. The original index condition
$\sum_i c_i=k$ becomes $p^{t+1}a+p^tb=k$.
For each such pair, the Cartan coefficients in the two
factors on the right are precisely
$\prod_i\binom{q_i}{u_i}$ and
$\prod_i\binom{r_i}{v_i}$. These coefficients are unchanged
by raising them to powers of~$p$. The displayed exponents
therefore identify the complete grouped sum with the asserted
product. The bounds on $a$ and $b$ follow from the bounds
on the individual $u_i$ and $v_i$.
\end{proof}

When $q_i=0$ for every $i$, Proposition~\ref{prop:tanay-indices}
gives the corrected form of
\cite[Corollary~2.7]{TanayOner2013}; when the $q_i$ and
nonzero $r_i$ are positive, it gives the corrected form of
\cite[Theorem~2.9]{TanayOner2013}. The zero-index cases in
\cite[Lemma~2.5, Corollary~2.6 and Theorem~2.8]{TanayOner2013}
are obtained by retaining $\Pp^0$ as the identity.
The subsequent formulas
\cite[Corollaries~2.10--2.11, Theorem~2.12 and Corollary~2.13]{TanayOner2013}
are correct; the numerical conclusions in the examples on
pp.~24--25 are also correct and will not be rederived here.

A compatibility with changes of variables will be used repeatedly.
It allows one-variable non-hit classes to detect multivariable
counterexamples without making any assertion that the hit subspace
is an ideal.

\begin{lem}\label{lem:linear-substitution}
Let $L:P_h\longrightarrow P_k$ be an $\Fp$-algebra homomorphism
such that each $L(x_i)$ is a linear form. Then
$L\Pp^s=\Pp^sL$ for every $s\ge0$, and
$L(\Hit_h)\subseteq\Hit_k$.
\end{lem}
\begin{proof}
Let $y_1,\ldots,y_k$ be the target variables and write
$L(x_i)=\sum_j a_{ij}y_j$, with $a_{ij}\in\Fp$.
The total reduced power of this linear form is
\[
 \Pp_z(L(x_i))
 =\sum_j a_{ij}(y_j+y_j^pz)
 =L(x_i)+L(x_i)^pz
 =L(\Pp_z(x_i)).
\]
The middle equality uses $a_{ij}^p=a_{ij}$ and Frobenius
additivity. Both composites are algebra homomorphisms into
$P_k[z]$, so agreement on all generators proves agreement
on $P_h$. Comparing coefficients of $z^s$ proves the first
assertion. Apply this equality to the image expressions of
Lemma~\ref{lem:hit-images} to obtain the second assertion.
\end{proof}

\subsection{A correction concerning the first reduced power}
\label{subsec:first-power}

Theorem~2.4 of~\cite[p.~5]{PhucOriginal} asserted that
$\Pp^1(f)=0$ implies $f\in\im\Pp^1$.
The implication fails even for homogeneous polynomials of
positive degree. For two variables, put $f=x^py-xy^p$.
Formula~\eqref{eq:cartan} gives $\Pp^1(f)=0$, whereas
Lemma~\ref{lem:hit-images} identifies the entire hit subspace
in degree $p+1$ with
\[
 \Span_{\Fp}\{x^{p+1},\ x^py+xy^p,\ y^{p+1}\}.
\]
Since $p$ is odd, $f$ does not belong to this subspace.
The induction on p.~5 of~\cite{PhucOriginal} has the two
defects identified in~\cite[\S2, item~3]{Sum2025}.
If $f=x_hf_1+f_2$ is homogeneous of polynomial degree $d$
and $f_2$ is independent of $x_h$, then $f_1$ has degree
$d-1$, or cohomological degree $2d-2$.
It does not have the degree of an antecedent of $f$ under
$\Pp^1$. Moreover, the Cartan formula and evaluation at
$x_h=0$ give
\[
 \Pp^1(f_2)=0,\qquad
 x_h\Pp^1(f_1)+x_h^pf_1=0.
\]
Since $P_h$ is an integral domain, the second equality gives
$\Pp^1(f_1)=-x_h^{p-1}f_1$.
It cannot be replaced by $\Pp^1(f_1)=x_hf_1$, as required
by the proposed final line of that induction. The explicit
non-hit polynomial above shows that correcting the induction
alone cannot preserve the asserted conclusion.

The following theorem replaces the kernel--image assertion in
the degree of this counterexample and determines the full
obstruction in any number of variables.

\begin{thm}\label{thm:first-power-kernel}
Let $h\ge1$ and let $p$ be odd. Then
\[
 \ker\bigl(\Pp^1:(P_h)_{p+1}\longrightarrow(P_h)_{2p}\bigr)
 =\Span_{\Fp}\{x_i^px_j-x_ix_j^p:1\le i<j\le h\}.
\]
The displayed polynomials are linearly independent, and this
kernel has zero intersection with $(\Hit_h)_{p+1}$.
\end{thm}
\begin{proof}
On $P_h$, the operation $\Pp^1$ is the derivation
$\sum_i x_i^p\partial/\partial x_i$, by~\eqref{eq:cartan}.
Partition the degree-$p+1$ monomials into the powers
$x_i^{p+1}$, the monomials $x_i^px_j$ with $i\ne j$, and
the monomials whose exponents are all less than $p$.
The first two types have images
\[
 \Pp^1(x_i^{p+1})=x_i^{2p},\qquad
 \Pp^1(x_i^px_j)=x_i^px_j^p.
\]
Each nonzero term in the image of a monomial of the third
type has exactly one exponent at least $p$. That exponent
is at most $2p-2$, and all other exponents are less than $p$.
Such a monomial uniquely determines the monomial to which $\Pp^1$ was
applied: subtract $p-1$ from its unique exponent at least $p$. Consequently,
the supports of the $\Pp^1$-images of distinct monomials of the third type
are disjoint. Every monomial of the third type has a nonzero image,
because at least one of its exponents lies between $1$ and $p-1$.
The supports of these images are also disjoint from those of the
images of monomials of the first two types.

It follows that a kernel element has zero coefficients on
the first and third types. Within the second type, the only
coincidences of images are those of $x_i^px_j$ and $x_ix_j^p$.
Their coefficients must be opposite. This proves the kernel
formula, and disjoint supports prove the stated independence.

Lemma~\ref{lem:hit-images} gives
$(\Hit_h)_{p+1}=\Pp^1((P_h)_2)$. Its basis consists of
the nonzero polynomials $2x_i^{p+1}$ and
$x_i^px_j+x_ix_j^p$ for $i<j$; their supports are disjoint.
For a fixed pair $i<j$, a vector proportional both to
$(1,-1)$ and to $(1,1)$ must be zero, since $2\ne0$ in
$\Fp$. Comparing the coefficients on every pair proves
that the kernel and the hit subspace have zero intersection.
\end{proof}

There is a separate attribution issue concerning the converse
implication discussed on p.~5 of~\cite{PhucOriginal}.
The example used there is the example of
Crossley~\cite[p.~175]{Crossley1999Bases}; its omitted citation
was pointed out in~\cite[\S2, item~2]{Sum2025}.
With this attribution made explicit, its calculation is
\[
\begin{aligned}
 \Pp^1(x^4y^5+x^8y)
 &=4x^8y^5+5x^4y^9+8x^{12}y+x^8y^5\\
 &=3x^{12}y
 \qquad\text{in }\mathbb F_5[x,y],\\
 \Pp^1(3x^{12}y)
 &=36x^{16}y+3x^{12}y^5
 =x^{16}y+3x^{12}y^5\ne0.
\end{aligned}
\]
The two terms in the last polynomial are distinct nonzero
monomials. Thus this example shows that membership in
$\im\Pp^1$ need not imply annihilation by $\Pp^1$.
It has a different role from the non-hit kernel elements
of Theorem~\ref{thm:first-power-kernel}.

\subsection{The Frobenius constructions}
\label{subsec:frobenius-maps}

Theorem~2.5 of~\cite[p.~5]{PhucOriginal} assumes positive
cohomological degree $2p^j$. The typographical repairs
in~\cite{PhucCorrection} leave that degree hypothesis
unchanged. Its conclusions follow from the instability identity
\cite[Proposition~1.4(ii), p.~16]{TanayOner2013}
and~\eqref{eq:small-multiplier}. As observed in~\cite[\S2, item~4]{Sum2025}, the restriction
to a power of $p$ is unnecessary for its first conclusion.
The following statement includes nonhomogeneous polynomials
and records the condition at degree zero.

\begin{thm}\label{thm:frobenius-hit}
Let $g\in P_h$ and let $\ell\in(P_h)_1$. Then
$g^p\in\Hit_h$ if and only if $g(0)=0$.
If $g\in\Hit_h$, then $\ell^{p-1}g^p\in\Hit_h$.
\end{thm}
\begin{proof}
Every positive reduced power has zero constant coefficient,
so every hit polynomial has zero constant coefficient.
The constant coefficient of $g^p$ is $g(0)^p=g(0)$, which
proves necessity.

Suppose $g(0)=0$ and write $g=\sum_{d>0}g_d$ with $g_d$
homogeneous of degree $d$. The instability identity just cited
gives $\Pp^d(g_d)=g_d^p$. Frobenius additivity gives
$g^p=\sum_{d>0}\Pp^d(g_d)\in\Hit_h$.

If $g$ is hit, Lemma~\ref{lem:hit-images} gives a finite
expression $g=\sum_{s>0}\Pp^s(g_s)$, after taking homogeneous
components if necessary. Since $\deg\ell^{p-1}<p$,
\eqref{eq:small-multiplier} gives
\[
 \ell^{p-1}g^p
 =\sum_{s>0}\Pp^{ps}(\ell^{p-1}g_s^p)\in\Hit_h.
\]
\end{proof}

Theorems~2.6 and~2.7 of~\cite[p.~6]{PhucOriginal} were
introduced there as consequences of Tanay and \"Oner's work.
The relevant results
\cite[Theorem~2.12 and Propositions~3.2--3.3]{TanayOner2013}
concern multiplication of a specified bounded hit expression
by a high Frobenius monomial. In each resulting antecedent,
that same monomial factor is prescribed. They do not identify
arbitrary hit expressions of the product, nor do they by
themselves establish a map on every homogeneous quotient.

The limited hit expressions in
\cite[Theorems~2.6--2.7, p.~6]{PhucOriginal} use
the generators $\Pp^{p^k}$; their $k=0$ term is $\Pp^1$.
A map on indecomposables must preserve the entire hit
subspace in the specified degree. For the first construction,
every value is a $p$th power with zero constant coefficient,
so the induced map exists and is zero, as also observed in
\cite[\S2, item~5]{Sum2025}. The following formulation records
this conclusion and its polynomial degree shift.

\begin{thm}\label{thm:zero-frobenius-map}
Let $t\ge0$, let $q_1,\ldots,q_h\ge1$, and define
\[
 \varphi(f)=\left(\prod_{j=1}^h x_j^{q_jp^{t+1}}\right)f^p
 \qquad(f\in P_h).
\]
Then $\varphi(P_h)\subseteq\Hit_h$. For every $n\ge0$, the
induced linear map
\[
 (\QP_h)_n\longrightarrow
 (\QP_h)_{pn+p^{t+1}\sum_jq_j}
\]
is zero.
\end{thm}
\begin{proof}
For $b=\prod_jx_j^{q_jp^t}$ one has $\varphi(f)=(bf)^p$.
The polynomial $bf$ has zero constant coefficient, so
Theorem~\ref{thm:frobenius-hit} proves that $\varphi(f)$ is
hit. Frobenius is $\Fp$-linear, and a homogeneous input of
degree $n$ has output degree $pn+p^{t+1}\sum_jq_j$.
The induced map exists because the whole image is contained
in $\Hit_h$, and it has value zero on every quotient class.
\end{proof}

\subsection{A correction to the multiplier theorem}
\label{subsec:multiplier-correction}

The multiplier in Theorem~2.7 of~\cite[p.~6]{PhucOriginal} has
the form
\begin{equation}\label{eq:psi-definition}
 \psi(g)=
 \left(\prod_{j=1}^{i}x_j^{q_jp^{t+1}+r_j}\right)
 \left(\prod_{j=i+1}^{h}x_j^{q_jp^{t+1}}\right)g^p,
\end{equation}
where $t\ge2$, $q_j\ge1$, and $1\le r_j\le p-1$.
The printed premise is
\[
 g=\sum_{k=0}^{t-2}\Pp^{p^k}(g_k),
\]
whereas the conclusion asserts a map on the whole cohit
quotient in every degree. A verification for a limited
collection of operations would not by itself establish this
conclusion. In fact, the following example disproves even
preservation of hits satisfying the printed premise.

\begin{prop}\label{prop:limited-counterexample}
Let $p$ be odd, let $t\ge2$, and work in $P_3$.
Set
\[
 A=x_1^{(p-2)p^{t+1}+p-1}
       x_2^{p^{t+1}+p-1}x_3^{p^{t+1}+1},
 \qquad g=x_1^{p^{t+1}-2}.
\]
Then $g\in\im\Pp^1$ and $Ag^p\notin\Hit_3$.
The multiplier $A$ satisfies the parameters of
\eqref{eq:psi-definition} with $i=3$, $(q_1,q_2,q_3)=(p-2,1,1)$, and
$(r_1,r_2,r_3)=(p-1,p-1,1)$.
\end{prop}
\begin{proof}
Since $p\ge3$, the three $q_j$ are positive, their sum
is $p$, and the specified residual exponents all lie between
$1$ and $p-1$. The integer $b=p^{t+1}-p-1$ is nonnegative
and is congruent to $-1$ modulo $p$. Formula~\eqref{eq:cartan}
therefore gives
\[
 \Pp^1(x_1^b)=b x_1^{b+p-1}=-x_1^{p^{t+1}-2}=-g.
\]
Thus $g=-\Pp^1(x_1^b)$ has a hit expression with only
the $k=0$ term in the printed range. Its degree also
satisfies $\deg g=p^{t+1}-2<p^{t+1}$.

Let $L:P_3\to P_1$ send every variable to $x$.
By Lemma~\ref{lem:linear-substitution}, $L$ sends hits to hits.
The degree of $A$ is $p^{t+2}+2p-1$, so
\[
 L(Ag^p)
 =x^{p^{t+2}+2p-1+p(p^{t+1}-2)}
 =x^{2p^{t+2}-1}.
\]
The exponent plus one has the form $2p^{t+2}$, with
$1\le2\le p-1$. Theorem~\ref{thm:rank-one} says that
this monomial is non-hit. Consequently $Ag^p$ cannot
be hit, proving the assertion.
\end{proof}

For $p=3$ and $t=2$, the example is
$g=x_1^{25}=2\Pp^1(x_1^{23})$ and
$A=x_1^{29}x_2^{29}x_3^{28}$; the specialization is $x^{161}$.

A second example will distinguish the two hypotheses in
the replacement theorem below. Its residual multiplier has
degree less than $p$, but its input lies above the allowed
degree range.

\begin{prop}\label{prop:uniform-counterexample}
Let $p$ be odd and let $t\ge2$. Define
\[
 \psi(g)=x_1^{p^{t+1}+p-1}x_2^{p^{t+1}}g^p
 \quad(g\in P_2),\qquad
 g_t=x_1^{2(p-1)p^t-1}.
\]
Then $g_t\in\Hit_2$ and $\psi(g_t)\notin\Hit_2$.
\end{prop}
\begin{proof}
The digit in position $t$ of $(p-1)p^t-1$ is $p-2$.
Equations~\eqref{eq:cartan} and~\eqref{eq:binomial-digits}
therefore give
\[
 g_t=(p-2)^{-1}\Pp^{p^t}
       \bigl(x_1^{(p-1)p^t-1}\bigr).
\]
On the other hand,
\[
 \psi(g_t)=x_1^{(2p-1)p^{t+1}-1}x_2^{p^{t+1}}.
\]
The substitution $\rho:P_2\to P_1$ with
$\rho(x_1)=\rho(x_2)=x$ sends hits to hits by
Lemma~\ref{lem:linear-substitution}. However,
$\rho\psi(g_t)=x^{2p^{t+2}-1}$ is non-hit by
Theorem~\ref{thm:rank-one}, because $2\le p-1$.
Thus $\psi(g_t)$ cannot be hit.
\end{proof}

The obstruction arises from positive Cartan terms on the
multiplier. The identities of Tanay and \"Oner recalled in
Lemma~\ref{lem:frobenius} and~\eqref{eq:frobenius-product}
eliminate those terms under two degree bounds. Their
bounded multiplication result
\cite[Theorem~2.12(ii) and Proposition~3.2]{TanayOner2013}
explains the high Frobenius factor; the residual-degree bound
is separately required for the additional factor in
\eqref{eq:psi-definition}. The next theorem proves that every
hit in the stated homogeneous component has an expression
within the permitted range, and therefore gives the required
replacement of~\cite[Theorem~2.7]{PhucOriginal} on that quotient.

\begin{thm}\label{thm:bounded-multiplier}
Let $t\ge0$, $0\le i\le h$, $q_1,\ldots,q_h\ge0$, and
$0\le r_j\le p-1$ for $1\le j\le i$.
Set $R=\sum_{j=1}^i r_j$ and define $\psi$ by
\eqref{eq:psi-definition}. If $R<p$ and $0\le n<p^{t+1}$,
then
\[
 \psi((\Hit_h)_n)\subseteq
 (\Hit_h)_{pn+p^{t+1}\sum_jq_j+R}.
\]
Consequently $\psi$ induces a linear map between the cohit
quotients in these two polynomial degrees.
\end{thm}
\begin{proof}
Write $a=\prod_{j=1}^i x_j^{r_j}$,
$b=\prod_{j=1}^h x_j^{q_j}$, and $A=ab^{p^{t+1}}$.
Thus $\deg a=R<p$ and $\psi(g)=Ag^p$.
For $1\le v<p^t$, Lemma~\ref{lem:frobenius} and the Cartan
formula give $\Pp^{pv}(A)=0$. Indeed, a positive index on
$b^{p^{t+1}}$ would have to be at least $p^{t+1}>pv$,
and index $pv$ on $a$ is greater than its degree.

Let $g\in(\Hit_h)_n$. By Lemma~\ref{lem:hit-images},
\[
 g=\sum_{s=1}^{\lfloor n/p\rfloor}\Pp^s(g_s),
 \qquad \deg g_s=n-(p-1)s.
\]
Every index satisfies $s<p^t$. The cited product formula~\eqref{eq:frobenius-product},
together with Lemma~\ref{lem:frobenius}, yields
\[
 \Pp^{ps}(Ag_s^p)
 =\sum_{v=0}^s\Pp^{pv}(A)(\Pp^{s-v}g_s)^p
 =A(\Pp^sg_s)^p.
\]
Summation proves that $\psi(g)$ is hit. If the index set
is empty, the hit subspace is zero and the same conclusion
holds. The degree formula follows by adding the degree of
$A$ to $pn$, and $\Fp$-linearity follows from Frobenius
additivity and $c^p=c$ for $c\in\Fp$.
\end{proof}

Both restrictions have a role in a uniform statement of this form.
In Proposition~\ref{prop:uniform-counterexample}, the residual
degree is $p-1<p$, while
\[
 \deg g_t-p^{t+1}=(p-2)p^t-1>0.
\]
It therefore rules out removal of the input-degree bound.
In Proposition~\ref{prop:limited-counterexample}, the input
has degree below $p^{t+1}$, while the residual degree is
$2p-1\ge p$. It rules out removal of the residual-degree
restriction. These examples do not assert that the bounds
are necessary for every individual multiplier; they establish
that each hypothesis is needed for a general statement with
the other hypotheses unchanged.

\subsection{Further results on hit polynomials}
\label{subsec:further-polynomial}

The following results are additional consequences of the
polynomial analysis. They are separate from the replacements
for the statements of~\cite{PhucOriginal}. The first gives
a complete criterion for monomial multipliers in rank one;
the remaining statements explain two constraints on
multivariable arguments involving hit relations.

\begin{thm}\label{thm:rankone-multiplier}
Let $a\ge0$ be an integer, and define a linear map
$M_a:P_1\longrightarrow P_1$ by $M_a(f)=x^a f^p$.
Then $M_a(\Hit_1)\subseteq \Hit_1$ if and only if
\[
 a\not\equiv p-1\pmod p\qquad\text{or}\qquad a=p-1.
\]
\end{thm}
\begin{proof}
The rank-one homogeneous components are one-dimensional, so
$\Hit_1$ is spanned by the hit monomials.  It suffices to test these
monomials separately.

Suppose first that $a\equiv r\pmod p$ with $0\le r\le p-2$.
A hit monomial $x^d$ has $d\ge p$, since all monomials of degree
less than $p$ are non-hit by Theorem~\ref{thm:rank-one}.
The integer $a+pd+1$ is larger than $p$, and its units digit is
$r+1\ne0$.  It therefore has a nonzero digit above the units
position as well.  The rank-one theorem shows that
$x^{a+pd}$ is hit.
If $a=p-1$, then $a+pd+1=p(d+1)$, so the number of nonzero
base-$p$ digits is unchanged.  Hit monomials are again sent to
hit monomials.

It remains to consider $a=pb+p-1$ with $b\ge1$.
Let $r=v_p(b)$, and choose an integer $K>r$ so large that
$p^K-b>(p-1)p^r$.  Put $d=p^K-b-1$.
The positive integer $d+1$ has $p$-adic valuation exactly $r$,
and its quotient by $p^r$ is greater than $p-1$ and not divisible
by $p$.  Hence $d+1$ has at least two nonzero base-$p$ digits.
The monomial $x^d$ is hit, whereas
\[
 M_a(x^d)=x^{p(b+d+1)-1}=x^{p^{K+1}-1}
\]
is non-hit.  Thus $M_a$ does not preserve $\Hit_1$ in this remaining
case.
\end{proof}

The hit subspace is a vector subspace, and multiplication by an
arbitrary polynomial need not preserve it. In particular, replacing
it by the ideal it generates loses precisely the information
needed to define a map on indecomposables.

\begin{prop}\label{prop:hit-not-ideal}
Let $m\ge1$, and let $J_m$ be the ideal of $P_h$ generated by
$\operatorname{im}\Pp^s$ for $1\le s\le m$.  Then
\[
 J_m=(x_1^p,\ldots,x_h^p).
\]
The hit subspace $\Hit_h$ is not an ideal when $h\ge1$.
\end{prop}
\begin{proof}
Since $\Pp^1x_i=x_i^p$, the ideal on the right is contained
in $J_m$.  Consider any nonzero monomial term in~\eqref{eq:cartan} with $s>0$.
Some $c_i$ is positive.  Its coefficient can be nonzero only if
$b_i\ge c_i\ge1$, and its $i$th exponent is
$b_i+(p-1)c_i\ge p$.  The term is divisible by $x_i^p$.
Hence every positive reduced-power image belongs to
$(x_1^p,\ldots,x_h^p)$, proving the equality.

The monomial $x_1^p=\Pp^1x_1$ is hit, whereas
$x_1^{2p-1}$ is non-hit by Theorem~\ref{thm:rank-one}.
For $h>1$, apply Lemma~\ref{lem:linear-substitution} to
$x_2=\cdots=x_h=0$. A hypothetical hit expression would
make the same rank-one monomial hit, which is impossible.  Multiplication of $x_1^p$ by $x_1^{p-1}$ therefore
does not preserve $\Hit_h$.
\end{proof}

Although separate rank-one conditions do not determine all
multivariable hit relations, they give a useful family of
independent cohit classes. The coefficient functionals in the
next proof will also be used in Appendix~\ref{app:ranktwo}.

\begin{prop}\label{prop:spike-independence}
Let $\mathcal S_n$ be the set of vectors
$A=(a_1,\ldots,a_h)\in\mathbb Z_{\ge0}^h$ with $\sum_i a_i=n$
such that each $a_i+1$ has exactly one nonzero base-$p$ digit.
Then the classes $[x^A]$, $A\in\mathcal S_n$, are linearly
independent in $(QP_h)_n$.
\end{prop}
\begin{proof}
For $A\in\mathcal S_n$, let $\lambda_A$ extract the coefficient
of $x^A$ in a homogeneous polynomial.  Consider a contribution to
this coefficient in $\Pp^s(x^B)$ with $s>0$.
It would require $a_i=b_i+(p-1)c_i$ for integers
$c_i\ge0$ with $\sum_i c_i=s$.
Some $c_i$ is positive.  Theorem~\ref{thm:rank-one} says that
$x_i^{a_i}$ is non-hit, so
$\binom{b_i}{c_i}=0$ in $\Fp$.  Formula~\eqref{eq:cartan} makes the entire
contribution zero.  Thus $\lambda_A$ vanishes on $(\Hit_h)_n$.
Since $\lambda_A(x^{A'})$ equals $1$ when $A=A'$ and $0$
otherwise, these functionals prove linear independence in the
quotient.
\end{proof}
\section{The third algebraic transfer}\label{sec:transfer}

The polynomial corrections do not by themselves determine the
domain of Singer's transfer. This coinvariant space is dual to the invariant
subspace of the indecomposables of the full cohomology algebra.
We first establish the notation, the exterior-degree decomposition,
and the invariant equations needed for this calculation. We then
correct the dimension argument in Theorem~2.8
of~\cite[pp.~6--7]{PhucOriginal}, prove the transfer statement that holds
throughout its four degree families, and determine the corresponding
invariant spaces explicitly in the required starting degrees.
Equivariant degree reduction and the rank-two coefficient
construction prove isomorphisms in seven infinite families at
$p=3$ and $p=5$. The final subsections identify the exact
invariant and coinvariant obstructions that remain for
arbitrary normalized index differences.

\subsection{Cohomology, group actions, and the domain of the transfer}
\label{subsec:transfer-preliminaries}

Let $V_h=(\mathbb Z/p)^h$ and $G_h=\GL(h,\Fp)$. Write
\[
 H_h=H^*(BV_h;\Fp)
 =P_h\otimes\Lambda(u_1,\ldots,u_h),
 \qquad |x_i|=2,\quad |u_i|=1,
\]
and put $\QH^n(h)=(\Fp\otimes_{\A}H_h)^n$, where the
superscript denotes total cohomological degree. The action of
$G_h$ is contragredient on both sets of generators. Equivalently,
on a linear form $\ell$ one has $(g\ell)(v)=\ell(g^{-1}v)$,
and the corresponding exterior generator transforms by the
same linear substitution. This action commutes with Steenrod
operations and therefore acts on the quotients. In rank two
we use $x,y,u,v$ for the generators; in rank three we use
$x,y,z,u,v,w$ and write $U=uvw$.

The positive reduced powers annihilate each $u_i$, so
$\Pp^s(f\omega)=(\Pp^sf)\omega$ for a polynomial $f$ and an
exterior monomial $\omega$. The Bockstein satisfies
\begin{equation}\label{eq:beta}
 \beta(x_i)=0,\qquad \beta(u_i)=x_i,\qquad
 \beta(u_{i_1}\cdots u_{i_e})
 =\sum_{a=1}^e(-1)^{a-1}x_{i_a}
     u_{i_1}\cdots\widehat{u_{i_a}}\cdots u_{i_e}.
\end{equation}
For $0\le e\le h$ and $m\ge0$, set
$W_{e,m}=(P_h)_m\otimes\Lambda^e(V_h^*)$;
outside these ranges set $W_{e,m}=0$.

\begin{prop}\label{prop:exterior-slices}
Let $n\ge0$. Then there is a $G_h$-equivariant decomposition
\begin{equation}\label{eq:exterior-slices}
 \QH^n(h)=
 \bigoplus_{\substack{0\le e\le h\\m=(n-e)/2\in\NN}}
 \frac{W_{e,m}}
 {\displaystyle\sum_{s>0}\Pp^sW_{e,m-(p-1)s}
                  +\beta W_{e+1,m-1}}.
\end{equation}
The summand with $e=h$ is
$(\QP_h)_m\otimes\Lambda^h(V_h^*)$.
\end{prop}
\begin{proof}
Reduced powers preserve exterior degree and the Bockstein
lowers it by one, by~\eqref{eq:beta}. Every nonidentity word
in Steenrod generators has a leftmost positive reduced power
or Bockstein, after identity factors are omitted. Its image
therefore lies in the image of one of these operations.
Conversely, every such image is hit. The component of the
hit subspace in $W_{e,m}$ is consequently the denominator
in~\eqref{eq:exterior-slices}. Each image is homogeneous in
exterior degree, so the direct sum decomposition descends to
the quotient. All the maps commute with $G_h$, giving
equivariance. In exterior degree $h$, the incoming Bockstein
space is zero, and the remaining relations are exactly the
polynomial hit relations tensored with $\Lambda^h(V_h^*)$.
\end{proof}

The rank-one full-cohomology calculation has a specific role in
the attribution of the preliminary results. Corollary~2.2 of
\cite[pp.~4--5]{PhucOriginal} already referred to
Crossley~\cite[Theorem~2]{Crossley1999Transfer}.
In the author version of the latter paper, the discussion following
Table~1 leaves the rank-one calculation as an exercise.
The explanation following Corollary~2.2
of~\cite[p.~5]{PhucOriginal} concerns this verification. A detailed elementary proof therefore serves
to make the prerequisites explicit; it does not constitute a new
rank-one classification. The exercise in question concerns
the rank-one assertion of that transfer paper, whereas the
rank-two monomial-basis theorem is the separate result
in~\cite[Theorem~1.1]{Crossley1999Bases}.

The support list in Corollary~2.2
of~\cite[pp.~4--5]{PhucOriginal} nevertheless requires an endpoint correction: it omits $\QH^1(1)$.
The next statement includes its nonzero class $[u]$ and records
the Bockstein calculation needed in Appendix~\ref{app:ranktwo}.

\begin{corl}\label{cor:full-rank-one}
Let $h=1$. Then $\QH^0(1)=\Fp[1]$.
If $n>0$, then $\QH^n(1)$ is one-dimensional with basis
$[x^{cp^r-1}u]$ when $n=2cp^r-1$, $1\le c\le p-1$,
$r\ge0$, and it is zero in every other degree.
\end{corl}
\begin{proof}
In degree zero the cohomology is spanned by $1$. Every positive Steenrod
operation increases total degree, and cohomology is concentrated in
nonnegative degrees. Hence no nonzero image of a positive operation lies in
degree zero, so the class of $1$ survives. In even degree $2d>0$ the cohomology is spanned by
$x^d$. Formula~\eqref{eq:beta} gives
\[
 x^d=\beta(x^{d-1}u),
\]
so this entire component is hit.

In odd degree $2d+1$, with $d\geq0$, the cohomology is spanned
by $x^du$. There is no exterior-degree-two component in rank
one, and therefore no Bockstein image enters this component.
By Proposition~\ref{prop:exterior-slices}, its only relations
come from positive reduced powers. The identity
$\Pp^s(fu)=(\Pp^sf)u$ identifies these relations with the
polynomial hit relations in degree $d$. Consequently the class
of $x^du$ survives if and only if the class of $x^d$ survives
in $\QP_1$. Theorem~\ref{thm:rank-one} gives precisely
$d=cp^r-1$, where $1\leq c\leq p-1$ and $r\geq0$.
The total degree is then $2cp^r-1$, and the surviving component
is one-dimensional because its original component was
one-dimensional. In particular, $c=1$ and $r=0$ give the
nonzero class $[u]$ in degree one.
\end{proof}

\begin{lem}\label{lem:central-character}
Let $0\le e\le h$ and $m\ge0$. If $p-1$ does not divide
$m+e$, then every $G_h$-stable subquotient of $W_{e,m}$ has
zero $G_h$-fixed subspace.
\end{lem}
\begin{proof}
The central element $\lambda I$ acts on $W_{e,m}$ by
$\lambda^{-(m+e)}$ and hence acts by the same scalar on
every subquotient. Choose $\lambda$ generating $\Fp^\times$.
Under the stated hypothesis this scalar is different from
$1$, so a vector fixed by $\lambda I$ must be zero.
\end{proof}

\begin{corl}\label{cor:central-top}
Let $A$ be an integer such that $n=2(p-1)A-3\ge0$.
Then
\[
 (\QH^n(3))^{G_3}
 \cong\bigl((\QP_3)_{(p-1)A-3}\otimes\langle U\rangle\bigr)^{G_3},
\]
with a negative polynomial degree interpreted as zero.
\end{corl}
\begin{proof}
Only exterior degrees $1$ and $3$ occur in odd total degree
in rank three. In the exterior-degree-one summand,
$m=(p-1)A-2$, so $m+1=(p-1)A-1$ is not divisible by
$p-1$. Lemma~\ref{lem:central-character} makes its invariant
subspace zero, after all relations have been imposed.
Proposition~\ref{prop:exterior-slices} identifies the remaining
summand with the displayed top exterior quotient.
\end{proof}

This reduction uses the central character; parity alone does
not remove the exterior-degree-one summand. The top exterior
representation is the character $\det^{-1}$, which must be
included in the remaining invariant calculation.

Let $T_h$ be the diagonal subgroup of $G_h$.
For a polynomial exponent vector $B=(b_1,\ldots,b_h)$, write
$x^B=x_1^{b_1}\cdots x_h^{b_h}$ and $|B|=\sum_i b_i$.

\begin{prop}\label{prop:torus-block}
Let $m\ge0$. Then the polynomial hit quotient decomposes
into $T_h$-weight spaces indexed by
$(b_1,\ldots,b_h)\bmod(p-1)$.
In $(\QP_h)_m\otimes\Lambda^h(V_h^*)$, the $T_h$-fixed
subspace is the weight space represented by the monomials
with $b_i\equiv p-2\pmod{p-1}$ for every $i$.
\end{prop}
\begin{proof}
Every nonzero term in~\eqref{eq:cartan} has the same
coordinate residues modulo $p-1$ as the monomial to which the operation is
applied. Thus the
polynomial space and the hit subspace decompose by those
residues, and so does the quotient. The diagonal element
$\diag(\lambda_1,\ldots,\lambda_h)$ acts on $x^B$ by
$\prod_i\lambda_i^{-b_i}$. Tensoring with the top exterior
representation changes this scalar to
$\prod_i\lambda_i^{-(b_i+1)}$. It equals $1$ for all
independent choices of the $\lambda_i$ precisely when
$p-1$ divides every $b_i+1$.
\end{proof}

For the finite calculations, the hit relations and group
actions can be written as exact linear maps. Let $M_{h,p,m}$
have rows indexed by exponent vectors $A$ with $|A|=m$ and
columns indexed by pairs $(j,B)$ with
$j\ge0$ and $|B|=m-(p-1)p^j\ge0$. Order the rows
lexicographically by increasing exponent vectors. Define
\begin{equation}\label{eq:hit-matrix-entry}
 (M_{h,p,m})_{A,(j,B)}=
 \begin{cases}
 \displaystyle\prod_{i=1}^h\binom{b_i}{c_i},
 &A-B=(p-1)C,\ C\in\NN^h,\ |C|=p^j,\\
 0,&\text{otherwise}.
 \end{cases}
\end{equation}
Columns on which instability forces zero are allowed in this
definition. The column indexed by $(j,B)$ is the coefficient vector of $\Pp^{p^j}(x^B)$,
with all Cartan terms retained.

\begin{prop}\label{prop:hit-matrix}
Let $m\ge0$ and $N=\binom{m+h-1}{h-1}$. Then
\[
 (\QP_h)_m\cong\Fp^N/\im M_{h,p,m},\qquad
 \dim(\QP_h)_m=N-\rank M_{h,p,m}.
\]
If $J$ is the set of nonpivot coordinate indices in the
reduced row echelon form of $M_{h,p,m}^{\mathsf T}$, then
the classes of the monomials indexed by $J$ form a basis
of $(\QP_h)_m$.
\end{prop}
\begin{proof}
Equation~\eqref{eq:cartan} identifies the columns with
$\Pp^{p^j}(x^B)$. Lemma~\ref{lem:hit-images} shows that
their span is $(\Hit_h)_m$, proving the quotient and
dimension formulas. The row space of $M_{h,p,m}^{\mathsf T}$
is this same subspace in target coordinates. In reduced row
echelon form, a nonzero row has one pivot coordinate equal
to $1$, all other pivot coordinates zero, and possible
entries in the nonpivot coordinates. Subtracting multiples
of these rows leaves a representative supported on $J$.
If a vector supported on $J$ belongs to the row space,
its pivot coefficients show that every coefficient in its
expression in reduced rows is zero. Thus the nonpivot
classes are both spanning and independent.
\end{proof}

Write $d=\dim(\QP_h)_m$. Relative to this quotient basis,
let $R:(P_h)_m\to\Fp^d$ be reduction modulo the hit subspace
and let $C:\Fp^d\to(P_h)_m$ choose the corresponding
monomial representatives. Thus $\ker R=(\Hit_h)_m$ and
$RC=I_d$. If $s\in G_h$, let $\Gamma_s$ be its matrix on
$(P_h)_m$.

\begin{lem}\label{lem:gl-generators}
Let $\zeta$ generate $\Fp^\times$. Then the substitutions
which scale one variable by $\zeta$, interchange adjacent
variables, or replace $x_j$ by $x_j-x_i$ for $i\ne j$
generate the contragredient action of $G_h$.
\end{lem}
\begin{proof}
Powers of a scaling give every nonzero scalar on that
coordinate. Powers of a stated transvection give every
parameter in $\Fp$, since its nilpotent part has square
zero. Adjacent interchanges generate all coordinate
permutations. Given an invertible matrix, interchange rows
to place a nonzero entry in the first pivot position,
scale it to $1$, and use row additions to clear the other
entries in its column. Repeat on the remaining submatrix.
This reduces the matrix to the identity by the listed
types of elementary transformations. Taking inverses
expresses it as their product. The contragredient
representation permutes the same set of invertible linear
substitutions, proving the assertion.
\end{proof}

Let $S_h$ consist of the elements of $G_h$ whose contragredient
substitutions are listed in Lemma~\ref{lem:gl-generators}. Thus
$\det(s)^{-1}$ is the determinant of the substitution on the
linear generators. In the finite calculations take
$\zeta=2$ for $p=3$ and $\zeta=3$ for $p=5$.
For a one-dimensional
character $\chi$ of $G_h$, define
\begin{equation}\label{eq:fixed-matrix}
 K_{h,p,m}^{\chi}=
 \begin{pmatrix}
 \chi(s_1)R\Gamma_{s_1}C-I_d\\
 \vdots\\
 \chi(s_a)R\Gamma_{s_a}C-I_d
 \end{pmatrix},\qquad S_h=\{s_1,\ldots,s_a\}.
\end{equation}

\begin{thm}\label{thm:gl-invariants}
Let $m\ge0$ and let $\chi$ be a one-dimensional character
of $G_h$. Then
\[
 \bigl((\QP_h)_m\otimes\chi\bigr)^{G_h}
 \cong\ker K_{h,p,m}^{\chi},\qquad
 \dim\bigl((\QP_h)_m\otimes\chi\bigr)^{G_h}
 =d-\rank K_{h,p,m}^{\chi}.
\]
\end{thm}
\begin{proof}
Since the action commutes with Steenrod operations, it
preserves the hit subspace. Thus $R\Gamma_sC$ is its
matrix on the quotient, and tensoring with $\chi$ multiplies
this matrix by $\chi(s)$. A vector is fixed by the group
if and only if it is fixed by all the generators in $S_h$:
fixedness is preserved under their products and inverses.
The simultaneous fixed-vector equations are exactly the
kernel equations for~\eqref{eq:fixed-matrix}. The dimension
formula follows from rank-nullity.
\end{proof}

The domain of the odd-primary algebraic transfer is the coinvariant space
\[
 D_{h,n}=\bigl(\Ann_{\A^+}H_n(BV_h;\Fp)\bigr)_{G_h},
\]
where the subscript denotes coinvariants. The transfer is
\begin{equation}\label{eq:transfer}
 \Tr_{h,n}:D_{h,n}\longrightarrow
 \Ext_{\A}^{h,h+n}(\Fp,\Fp).
\end{equation}
The duality between this coinvariant space and the cohomology invariants can
be described without identifying homology with cohomology. In a fixed degree, the evaluation pairing
$H_n(BV_h;\Fp)\otimes H^n(BV_h;\Fp)\to\Fp$ is perfect.
An element $v$ belongs to $\Ann_{\A^+}H_n$ exactly when
its pairing with $\theta f$ is zero for every positive operation
$\theta$ and every $f$ of the required degree. Thus the
annihilated homology is the annihilator of the hit subspace,
and restriction of the pairing identifies it with
$(\QH^n(h))^*$.

For any finite-dimensional $G_h$-module $E$, a functional on
$E$ factors through $E_{G_h}$ if and only if it vanishes on
all $gv-v$. This condition says that the functional belongs
to $(E^*)^{G_h}$. Applying it to the annihilated homology,
and then using finite-dimensional double duality, gives
\begin{equation}\label{eq:coinvariant-duality}
 D_{h,n}^*\cong(\QH^n(h))^{G_h},\qquad
 \dim D_{h,n}=\dim(\QH^n(h))^{G_h}.
\end{equation}
This is the full cohomology quotient of
Proposition~\ref{prop:exterior-slices}; its incoming
Bockstein relations are part of the pairing argument.

To compare invariant spaces in different degrees, we use Minami's
extraction map on polynomial indecomposables. We record its
compatibility with the contragredient action. The numerical
hypothesis for its being an isomorphism will be verified in
Subsection~\ref{subsec:infinite-transfer}.

For $d=\sum_i d_i p^i\in\NN$, with $0\le d_i<p$, put
$\ds(d)=\sum_i d_i$. Define
\begin{equation}\label{eq:minami-mu}
 \mu_p(d)=\min\left\{b\in\NN:
 d(p-1)=\sum_{j=1}^{b}(p^{a_j}-1),\quad a_j\in\NN\right\}.
\end{equation}
The empty sum is zero. The minimum exists because $d$ copies of
$p^1-1$ give a representation. We use $\mu_p$ for the function
written $\beta_p$ in~\cite{Minami}. The numerical criterion
\cite[Lemma~1.1(ii)]{Minami} is
\begin{equation}\label{eq:minami-digit-criterion}
 \mu_p(d)\le b
 \quad\Longleftrightarrow\quad
 \ds\bigl(d(p-1)+b\bigr)\le b,
 \qquad d,b\in\NN.
\end{equation}

Define an $\Fp$-linear map $\kappa_h:P_h\longrightarrow P_h$
on monomials by
\begin{equation}\label{eq:extraction-map-general}
 \kappa_h(x_1^{a_1}\cdots x_h^{a_h})=
 \begin{cases}
 \displaystyle\prod_{i=1}^h x_i^{(a_i-p+1)/p},
       &a_i\equiv p-1\pmod p\text{ for every }i,\\
 0,&\text{otherwise}.
 \end{cases}
\end{equation}
In the first case every $a_i\ge p-1$, so all the resulting
exponents are nonnegative. By
\cite[Proposition~1.7 and Theorem~1.8]{Minami}, this map satisfies
\begin{equation}\label{eq:extraction-powers-general}
 \kappa_h\Pp^j=
 \begin{cases}
 \Pp^{j/p}\kappa_h,&p\mid j,\\
 0,&p\nmid j,
 \end{cases}
 \qquad j\ge0,
\end{equation}
and induces a surjective map
\begin{equation}\label{eq:extraction-on-cohits}
 \overline\kappa_h:
 (\QP_h)_{pd+h(p-1)}\longrightarrow(\QP_h)_d,
 \qquad d\ge0.
\end{equation}
This map is an isomorphism whenever
$\mu_p(pd+h(p-1))=h(p-1)$. We use these results with their
stated degree condition. The equivariance needed for the
invariant spaces is verified next.

\begin{lem}\label{lem:extraction-equivariant}
Let $p$ be an odd prime and let $h\ge1$. Then the map
$\kappa_h$ in~\eqref{eq:extraction-map-general} is
$G_h$-equivariant. If $d\ge0$ and
$\mu_p(pd+h(p-1))=h(p-1)$, then
\[
 \overline\kappa_h\otimes1:
 (\QP_h)_{pd+h(p-1)}\otimes\Lambda^h(V_h^*)
 \longrightarrow
 (\QP_h)_d\otimes\Lambda^h(V_h^*)
\]
is a $G_h$-equivariant isomorphism.
\end{lem}
\begin{proof}
By Lemma~\ref{lem:gl-generators}, it suffices to check
coordinate interchanges, nonzero scalar substitutions on one
coordinate, and the substitutions $x_j\mapsto x_j-x_i$ for
$i\ne j$. Coordinate interchanges preserve the simultaneous
congruence condition in~\eqref{eq:extraction-map-general} and
permute the resulting exponents, so they commute with
$\kappa_h$.

Consider a scalar substitution $x_i\mapsto\lambda x_i$,
where $\lambda\in\Fp^\times$. If a monomial fails the
congruence condition, both compositions with $\kappa_h$ are
zero. Otherwise write $a_i=pA_i+p-1$. The scalar before
extraction is
\[
 \lambda^{a_i}=
 \lambda^{pA_i+p-1}=\lambda^{A_i},
\]
since $\lambda^p=\lambda$ and $\lambda^{p-1}=1$. This is
also the scalar obtained by applying the substitution after
extraction. Thus every such scalar substitution commutes
with $\kappa_h$.

Suppose $h\ge2$ and consider
$\tau(x_1)=x_1$, $\tau(x_2)=x_2-x_1$, with the other
coordinates fixed. If an exponent on one of the fixed
coordinates is not congruent to $p-1$ modulo $p$, both
compositions vanish. If all those congruences hold, their
extracted factors are the same in both compositions. It
therefore remains to compare the two compositions on
$x_1^a x_2^b$. Write
\[
 a=pA+a_0,\qquad b=pB+b_0,
 \qquad 0\le a_0,b_0<p.
\]
The substitution gives
\[
 \tau(x_1^a x_2^b)
 =\sum_{j=0}^{b}(-1)^j\binom bj
      x_1^{a+j}x_2^{b-j}.
\]
Let $j_0$ be the least base-$p$ digit of $j$. If a summand
has nonzero coefficient, Lucas's congruence
\eqref{eq:binomial-digits} implies $j_0\le b_0$.
If that summand is retained by $\kappa_2$, then
$b_0-j_0\equiv p-1\pmod p$. The inequalities
$0\le b_0-j_0\le p-1$ force
$b_0-j_0=p-1$, hence $b_0=p-1$ and $j_0=0$.
The other retained exponent then forces $a_0=p-1$.
Consequently, if either $a_0$ or $b_0$ differs from
$p-1$, every retained coefficient is zero. In this case
$\kappa_2\tau(x_1^a x_2^b)=0=
\tau\kappa_2(x_1^a x_2^b)$.

Suppose now that $a_0=b_0=p-1$. A retained index has the
form $j=pJ$, with $0\le J\le B$. Lucas's congruence and
the oddness of $p$ give
\[
 \binom{pB+p-1}{pJ}=\binom BJ,
 \qquad (-1)^{pJ}=(-1)^J.
\]
It follows that
\begin{align*}
 \kappa_2\tau(x_1^{pA+p-1}x_2^{pB+p-1})
 &=\sum_{J=0}^{B}(-1)^J\binom BJ
       x_1^{A+J}x_2^{B-J}\\
 &=x_1^A(x_2-x_1)^B\\
 &=\tau\kappa_2(x_1^{pA+p-1}x_2^{pB+p-1}).
\end{align*}
Restoring the unchanged coordinates proves commutation with
$\tau$ in rank $h$. Coordinate interchanges conjugate this
substitution to each substitution $x_j\mapsto x_j-x_i$.
All the generators in Lemma~\ref{lem:gl-generators} therefore
commute with $\kappa_h$, proving equivariance for the
contragredient action.

The map induced on the cohit quotient is consequently
$G_h$-equivariant. Under the numerical hypothesis, it is a
linear isomorphism by~\cite[Theorem~1.8(ii)]{Minami}.
Tensoring it with the identity of $\Lambda^h(V_h^*)$
preserves both properties: for $g\in G_h$, $v$ in the
cohit quotient, and $\omega\in\Lambda^h(V_h^*)$,
\[
 (\overline\kappa_h\otimes1)\bigl(gv\otimes g\omega\bigr)
 =g\overline\kappa_h(v)\otimes g\omega
 =g(\overline\kappa_h\otimes1)(v\otimes\omega).
\]
The inverse is $\overline\kappa_h^{-1}\otimes1$.
\end{proof}

Crossley's rank-one and rank-two isomorphism theorem and
multiplicativity~\cite[Theorem~1 and Proposition~2]{Crossley1999Transfer}
give the following consequence.

\begin{prop}\label{prop:decomposable-image}
Let $p$ be an odd prime and $t\ge3$ an integer. If
$z\in\Ext_{\A}^{3,t}(\Fp,\Fp)$ belongs to the linear span of products
$ab$ with $a\in\Ext_{\A}^{1,*}(\Fp,\Fp)$ and
$b\in\Ext_{\A}^{2,*}(\Fp,\Fp)$, then $z$ belongs to
$\im\Tr_{3,t-3}$.
\end{prop}
\begin{proof}
Write $z=\sum_{\ell=1}^N a_\ell b_\ell$, with each product of
internal degree $t$. Surjectivity in ranks one and two provides
classes $\alpha_\ell$ and $\gamma_\ell$ in the coinvariant spaces that form
the domains of $\Tr_1$ and $\Tr_2$, respectively, such that $\Tr_1(\alpha_\ell)=a_\ell$ and
$\Tr_2(\gamma_\ell)=b_\ell$. Their external products define classes
in rank three: the Cartan formula shows that a tensor product of
annihilated homology classes is annihilated, and passing first to
block-diagonal coinvariants and then to $G_3$-coinvariants defines the
external product on the coinvariant spaces. Multiplicativity gives
$\Tr_3(\alpha_\ell\gamma_\ell)=a_\ell b_\ell$. Linearity then
gives $\Tr_3(\sum_\ell\alpha_\ell\gamma_\ell)=z$.
\end{proof}

We fix the target notation at this point. Put $q=2(p-1)$ and
write $\Ext^{s,t}_{\A}$ for $\Ext^{s,t}_{\A}(\Fp,\Fp)$ when
the coefficient modules are understood. The required classes are
\[
\begin{gathered}
 a_0\in\Ext^{1,1}_{\A},\qquad
 h_i\in\Ext^{1,qp^i}_{\A},\qquad
 \rho\in\Ext^{2,2q+1}_{\A},\\
 \widetilde\lambda_i\in\Ext^{2,qp^{i+1}}_{\A},\qquad
 \mu_i\in\Ext^{2,q(p+2)p^i}_{\A},\qquad
 \nu_i\in\Ext^{2,q(2p+1)p^i}_{\A}
 \qquad(i\ge0).
\end{gathered}
\]
We use Aikawa's degree-three basis
\cite[Theorem~8.1 and Table~8.1, pp.~107, 110]{Aikawa}.
Its degree-two factors agree, up to nonzero scalars, with
Crossley's classes of the indicated names
\cite[Theorem~3]{Crossley1999Transfer}.
Indeed, the corresponding degree-two components in the
cited basis theorem are one-dimensional, so two nonzero
choices differ by a nonzero scalar. These choices do not
affect linear independence or membership in a transfer image.
The exceptional relations at $p=3$ will always be interpreted
with the basis index ranges in
\cite[Tables~8.1 and~8.2, pp.~110--111]{Aikawa}.

\subsection{A correction to the four degree families}
\label{subsec:transfer-correction}

The proof of Theorem~2.8 of~\cite[pp.~6--7]{PhucOriginal}
begins by asserting a one-dimensional invariant space in each
of four degree families. It then invokes the polynomial
constructions in~\cite[Theorems~2.4--2.7, pp.~5--6]{PhucOriginal}
and a reduction in degree before passing to the full-cohomology quotient.
The objections to that argument in~\cite[pp.~1, 4]{Sum2025}
require a replacement of the invariant-space calculation. The polynomial
statements corrected in Section~\ref{sec:main} do not yield
the asserted reduction. Moreover, the passage from polynomial
indecomposables to full cohomology must include the
Bockstein images and the exterior representation, as in
Proposition~\ref{prop:exterior-slices}.

We use the following total cohomological degrees:
\begin{align}
 n_1&=q(p^k+p^j+p^i)-3,
       &&0\le k\le j-2\le i-4,\label{eq:family1}\\
 n_2&=q(p^{i+1}+p^j)-3,
       &&i,j\ge0,\quad j\ne i+2,\label{eq:family2}\\
 n_3&=q(p^{i+1}+2p^i+p^j)-3,
       &&i,j\ge0,\quad j\notin\{i-1,i,i+2\},\label{eq:family3}\\
 n_4&=q(2p^{i+1}+p^i+p^j)-3,
       &&i,j\ge0,\quad j\notin\{i-1,i,i+1,i+2\}.
       \label{eq:family4}
\end{align}
By Corollary~\ref{cor:central-top}, the associated polynomial
degree is $m_a=(n_a-3)/2$.

The displayed invocation of Minami's theorem on p.~6 of
\cite{PhucOriginal} also mixes the two gradings.
The top exterior polynomial factor has cohomological
degree $n-3$ and polynomial degree $m=(n-3)/2$.
Extraction of $(x_1x_2x_3)^{p-1}$ followed by division
of all exponents by $p$ has polynomial output degree
\begin{equation}\label{eq:corrected-downward-degree}
 \frac{m-3(p-1)}p
 =\frac{n-3-6(p-1)}{2p}.
\end{equation}
With polynomial subscripts, the left index of the published
display must therefore be $(n-3)/2$. With cohomological
subscripts, its right denominator must be $p$, rather
than $2p$. Neither convention gives the printed pair
of indices.

There is also a failed numerical hypothesis. On p.~6,
\cite{PhucOriginal} uses the function $\mu_p$ defined
in~\eqref{eq:minami-mu} and invokes the value
$\mu_p(m)=3(p-1)$ in its reduction.
For $p=3$, $i=j=0$ in~\eqref{eq:family2}, one has
$n=13$ and $m=5$. The equality
\[
 5(3-1)=(3^2-1)+(3-1)=8+2
\]
shows that $\mu_3(5)\le2$. The value is not $0$, since
$10\ne0$, and is not $1$, since $3^a=11$ has no integer
solution. Hence $\mu_3(5)=2$, whereas the required value
would be $6$. Formula~\eqref{eq:corrected-downward-degree}
also gives $-1/3$, which is not a polynomial degree.
Thus the degree-reduction argument does not cover this
member of the claimed family. This is a failure of its
application in~\cite{PhucOriginal}, not of the degree-dependent
result in~\cite[Theorem~1.8 and Corollary~1.9]{Minami}.
Proposition~\ref{prop:degree-thirteen} gives a direct
calculation in this degree.

The assertion that the domain of the transfer is one-dimensional has a
second, independent obstruction.
The target display on p.~7 of~\cite{PhucOriginal} selects
only one product in family~$2$, although two distinct
entries of Aikawa's basis can have that internal degree
\cite[Table~8.1]{Aikawa}. The following calculation makes
the resulting contradiction explicit.

\begin{prop}\label{prop:transfer-counterexample}
Let $p$ be odd, $q=2(p-1)$, and $n=q(p+p^3)-3$. Then
$\widetilde\lambda_0h_3$ and
$\widetilde\lambda_2h_1$ are linearly independent elements
of $\Ext_{\A}^{3,n+3}(\Fp,\Fp)$, and
$\dim(\QH^n(3))^{G_3}\ge2$.
\end{prop}
\begin{proof}
The two products occur as distinct entries of Aikawa's
basis~\cite[Theorem~8.1 and Table~8.1]{Aikawa}. Their
internal degrees are both $q(p+p^3)$, and their index pairs
$(0,3)$ and $(2,1)$ satisfy the restrictions on the
$\widetilde\lambda_i h_j$ row. Since they are distinct
basis elements, they are independent, also at $p=3$.
Proposition~\ref{prop:decomposable-image} puts both in the
image of the transfer. Their preimages in $D_{3,n}$ are
independent: any relation among the preimages would map
to a relation among the two basis elements. Thus
$\dim D_{3,n}\ge2$, and~\eqref{eq:coinvariant-duality} proves
the invariant-space inequality.
\end{proof}

In particular, the degree in this proposition is $117$
when $p=3$. The error concerns the dimension argument:
a transfer between two-dimensional spaces can still be an isomorphism. The next theorem gives the exact target
dimensions in every degree under consideration.

\begin{thm}\label{thm:target-dimensions}
Let $p$ be an odd prime, and let the indices satisfy the restrictions
in \eqref{eq:family1}--\eqref{eq:family4}. Then the target spaces in
families $1$, $3$, and $4$ have respective bases
\[
 \{h_i h_j h_k\},\qquad \{\mu_i h_j\},\qquad\{\nu_i h_j\}.
\]
If $n=n_2$, then
\begin{equation}\label{eq:target-two}
 \dim\Ext_{\A}^{3,n+3}(\Fp,\Fp)=
 \begin{cases}
  2,&j\ge1\text{ and }|j-i-1|\ge2,\\
  1,&\text{otherwise}.
 \end{cases}
\end{equation}
In the first case a basis is
$\{\widetilde\lambda_i h_j,
\widetilde\lambda_{j-1}h_{i+1}\}$; in the second case a basis is
$\{\widetilde\lambda_i h_j\}$.
\end{thm}
\begin{proof}
We apply Aikawa's basis theorem
\cite[Theorem~8.1 and Table~8.1]{Aikawa} and account for every
entry with the same internal degree. Divide a degree divisible by $q$ by $q$.
The four relevant rows of that table have normalized degrees
\[
 p^k+p^j+p^i,\quad p^{i+1}+p^j,\quad
 p^{i+1}+2p^i+p^j,\quad 2p^{i+1}+p^i+p^j,
\]
with exactly the restrictions in
\eqref{eq:family1}--\eqref{eq:family4}. Their digit sums are,
respectively, $3$, $2$, $4$, and $4$. In the third row the allowed
coincidence $j=i+1$ gives two adjacent digits equal to $2$, so the
digit sum remains $4$, including when $p=3$.

For completeness, the other normalized degrees divisible by $q$ in
Aikawa's table are
\begin{gather*}
 3p^{r+2}+2p^{r+1}+p^r,\quad
 p^{r+2}+3p^{r+1}+p^r,\quad
 p^{r+2}+2p^{r+1}+3p^r\qquad(p\ne3),\\
 2p^{r+3}+2p^{r+1}+p^r\quad(p=3),\qquad
 2p^{r+2}+p^{r+1}+2p^r,\\
 p^{r+1}+2p^r,\quad 2p^{r+1}+p^r\qquad(r\ge1).
\end{gather*}
The restriction $r\ge1$ applies only to the last two expressions.
The first three have digit sum $6$, the next two have digit sum
$5$, and the last two have digit sum $3$. The last two contain
adjacent nonzero digits $1$ and $2$, whereas the first family has
three separated digits $1$. Thus none of these additional rows
meets any of the four specified families. Every remaining entry in
the table has internal degree congruent to $1$, $2$, or $3$ modulo
$q$. Since $q\ge4$, these remaining entries cannot occur here.

Within the first family, the three nonzero digit positions uniquely
determine $i,j,k$. In the third family with $j\ne i+1$, the unique
digit $2$ determines $i$, and the digit $1$ outside position $i+1$
determines $j$. If $j=i+1$, the two adjacent digits $2$ determine
the pair uniquely. In the fourth family, the unique digit $2$
determines $i+1$, and the remaining digit $1$ outside position $i$
determines $j$. A degree from the third family with only one
digit $2$ could equal a degree from the fourth family only if the
fourth-family parameter were $i-1$. Its required digit $1$ at
position $i-1$ would force $j=i-1$ in the third family, which is
excluded. A third-family degree with two digits $2$ cannot equal
a fourth-family degree. This proves the assertions for families
$1$, $3$, and $4$.

For the second family put $a=i+1\ge1$ and $b=j\ge0$.
The expansion of $p^a+p^b$ determines the unordered multiset
$\{a,b\}$, since $2<p$ and no carry occurs.
If $a=b$, there is one table entry. If $a\ne b$, the entry
$\widetilde\lambda_{a-1}h_b$ is allowed by the assumed restriction
$b\ne a+1$. The only possible additional entry is
$\widetilde\lambda_{b-1}h_a$. It exists exactly when $b\ge1$ and
$a\ne b+1$. Combining these conditions with $a\ne b$ and
$b\ne a+1$ gives $b\ge1$ and $|a-b|\ge2$.
This is precisely the first condition in \eqref{eq:target-two}.
All entries counted are entries of Aikawa's \emph{basis}; distinct
ones are independent. In particular, the exceptional relations at
$p=3$ do not identify these two allowed entries. The relations in~\cite[Table~8.2]{Aikawa} express other,
excluded products in terms of the chosen basis entries. The stated
bases and dimensions follow.
\end{proof}

The conclusion that follows throughout the four degree families
from these target calculations is surjectivity. The following
theorem gives this replacement for the unrestricted isomorphism
assertion in~\cite[Theorem~2.8]{PhucOriginal}. Injectivity in
the six infinite families of Theorem~\ref{thm:infinite-isomorphisms}
will also use the invariant calculations and the equivariant
degree reduction proved below.

\begin{thm}\label{thm:surjectivity}
Let $p$ be an odd prime and let $n$ belong to one of the four
families \eqref{eq:family1}--\eqref{eq:family4}. Then
$\Tr_{3,n}$ is surjective, with target dimension given by
Theorem~\ref{thm:target-dimensions}.
\end{thm}
\begin{proof}
Each of the displayed basis elements in
Theorem~\ref{thm:target-dimensions} is a product of a degree-one
class with a degree-two class; the triple product in the first
family is of this form after grouping two factors. Each such
element is in the image by
Proposition~\ref{prop:decomposable-image}. Since these elements
form a basis of the entire target in the specified degree,
the transfer is surjective.
\end{proof}

Theorem~\ref{thm:surjectivity} follows from
\cite[Theorem~8.1 and Table~8.1]{Aikawa} and
\cite[Theorem~1 and Proposition~2]{Crossley1999Transfer}. To obtain an
isomorphism in all degrees of the four families, one
would additionally need an upper bound for $\dim D_{3,n}$ equal to the value
in
Theorem~\ref{thm:target-dimensions}. The calculations in
Subsection~\ref{subsec:rankthree-calculations} establish that
equality in the specified finite cases, and
Subsection~\ref{subsec:infinite-transfer} extends it to six
infinite families after verifying the degree-reduction hypotheses.
\subsection{Six further surjectivity families}
\label{subsec:additional-surjectivity}

Corollary~2.11 of~\cite[pp.~7--8]{PhucOriginal} asserts that
the target equals the image of the transfer in six additional
degree families. This is a surjectivity assertion. Its proof
uses the decomposable target classes and the rank-one and
rank-two transfer theorem
\cite[Theorem~1, Proposition~2 and Lemmata~1--3]{Crossley1999Transfer};
it does not use a bound for $\dim D_{3,n}$.
We verify the complete target bases below and recover that
surjectivity conclusion with its independent product argument.

The degrees introduced in Remark~2.10 of~\cite[p.~7]{PhucOriginal}
are
\begin{align}
 n_5&=q(p^i+p^j)-2,
       &&1\le j\le i-2,\label{eq:family5}\\
 n_6&=qp^i-1,
       &&i\ge1,\label{eq:family6}\\
 n_7&=qp^{i+1}-2,
       &&i\ge0,\label{eq:family7}\\
 n_8&=q(p^{i+1}+2p^i)-2,
       &&i\ge1,\label{eq:family8}\\
 n_9&=q(2p^{i+1}+p^i)-2,
       &&i\ge1,\label{eq:family9}\\
 n_{10}&=q(p^j+2)-2,
       &&j\ge2.\label{eq:family10}
\end{align}

\begin{thm}\label{thm:remaining-surjectivity}
Let $p$ be an odd prime, and let the indices satisfy the
restrictions in~\eqref{eq:family5}--\eqref{eq:family10}.
Then the six target spaces
$\Ext^{3,n_a+3}_{\A}$, for $5\le a\le10$, have respective bases
\[
 \{a_0h_ih_j\},\quad
 \{a_0^2h_i\},\quad
 \{a_0\widetilde\lambda_i\},\quad
 \{a_0\mu_i\},\quad
 \{a_0\nu_i\},\quad
 \{h_j\rho\}.
\]
For each of these degrees, $\Tr_{3,n_a}$ is surjective.
\end{thm}
\begin{proof}
We first establish that the displayed products are the entire
target bases, using~\cite[Theorem~8.1 and Table~8.1]{Aikawa}.
In each case their internal degrees follow by adding the
degrees fixed in Subsection~\ref{subsec:transfer-preliminaries}.
For instance, $a_0h_ih_j$ has internal degree
$1+q(p^i+p^j)=n_5+3$, while $h_j\rho$ has internal degree
$qp^j+2q+1=n_{10}+3$.
The index ranges in the theorem are exactly the ranges
of the corresponding six rows of the cited table.

All but the second of these degrees are congruent to $1$
modulo $q$. Write such an internal degree as $qN+1$.
The five relevant normalized values of $N$ are
\[
 p^i+p^j,\qquad p^{i+1},\qquad
 p^{i+1}+2p^i,\qquad 2p^{i+1}+p^i,\qquad p^j+2,
\]
with the restrictions in the corresponding families.
There is no carry in these expressions, since $2<p$.
The first has two digits $1$ at positive positions
separated by at least two places. The second has a single
digit $1$ at a positive position. The third and fourth
have adjacent nonzero digits $1$ and $2$ at positions
at least $1$, in opposite orders. The fifth has digit
$2$ in the units position and digit $1$ in a position
at least $2$. These patterns are pairwise distinct, and
each pattern determines its indices uniquely.

The only other rows of Aikawa's table with internal degree
congruent to $1$ modulo $q$ have normalized values
\[
 N=3p+2\quad\text{or}\quad N=p+3\qquad(p\ge5),
 \qquad N=2p^2+2\qquad(p=3).
\]
For $p\ge5$, their digit sums are $5$ and $4$;
for $p=3$, the last has digit sum $4$.
The five families above have digit sums $2$, $1$, $3$,
$3$ and $3$, respectively. Hence none of the additional
rows has one of the required degrees.

The second target has internal degree $qp^i+2$, with
$i\ge1$. The row of $a_0^2h_i$ is a basis row in this
degree. The only competing degree-$2$ residue entries
in~\cite[Table~8.1]{Aikawa} have internal degree
$3q+2$ when $p\ge5$, or internal degree $26$ when $p=3$.
Equality with the former would require $p^i=3$, which
is impossible for $p\ge5$. Equality with the latter
would require $3^i=6$, which is also impossible.
All other rows have residue $0$, $1$ or $3$ modulo $q$;
these residues differ from $2$, since $q\ge4$.
Thus this target also has exactly the stated basis.
The exceptional product relations at $p=3$ in
\cite[Table~8.2]{Aikawa} do not change these counts,
because the argument counts the allowed basis rows
of Table~8.1, rather than all possible product labels.

Each basis element is a product of a homological-degree-one
class and a homological-degree-two class. For the first
two products, group the factors as $a_0(h_ih_j)$ and
$h_i(a_0^2)$, respectively. The last four already have
the required form. Proposition~\ref{prop:decomposable-image}
therefore puts every basis element in the corresponding
transfer image. Their spans are the whole targets, so
all six maps are surjective.
\end{proof}

The degree-zero isomorphism on p.~7 of~\cite{PhucOriginal}
also has an independent proof. It does not use the invariant-space
calculations in the positive-degree families.

\begin{prop}\label{prop:transfer-degree-zero}
Let $p$ be an odd prime. Then
$\Tr_{3,0}:D_{3,0}\longrightarrow\Ext^{3,3}_{\A}$
is an isomorphism.
\end{prop}
\begin{proof}
The degree-zero homology of $BV_3$ is spanned by its unit
class. A positive Steenrod operation lowers homological
degree, so it annihilates this class, and $G_3$ fixes it.
It follows that $D_{3,0}$ is one-dimensional.
By~\cite[Theorem~8.1 and Table~8.1]{Aikawa},
$\Ext^{3,3}_{\A}$ is the one-dimensional space generated
by $a_0^3$. Proposition~\ref{prop:decomposable-image},
applied to $a_0\cdot a_0^2$, puts this nonzero class
in the image of $\Tr_{3,0}$.
Thus the transfer is a surjection between two
one-dimensional spaces and has zero kernel.
\end{proof}
\subsection{Explicit calculations in rank three}
\label{subsec:rankthree-calculations}

The invariant calculations in this subsection determine the dimensions of
the coinvariant spaces $D_{3,n}$ needed for the six finite isomorphism
statements.
A previously published quotient basis is not required for a fixed
degree: Proposition~\ref{prop:hit-matrix} constructs one from the
complete Cartan image relations. The fixed-vector equations of
Theorem~\ref{thm:gl-invariants}, together with~\eqref{eq:coinvariant-duality},
then determine $\dim D_{3,n}$.
These finite calculations provide the initial values for the
degree reduction in Subsection~\ref{subsec:infinite-transfer},
where its hypotheses are verified and the additional first
steps are computed.

We first compute the smallest case directly. For the remaining
cases we give explicit quotient monomials, all the equations
for two adjacent interchanges and one transvection, and their
solutions. In particular, the calculation at $p=3$, $n=117$
gives $\dim D_{3,117}=2$, in agreement with the
target calculation of Theorem~\ref{thm:target-dimensions}.

\begin{prop}\label{prop:degree-thirteen}
Let $p=3$. Then $\dim(\QP_3)_5=14$, and
\[
 (\QH^{13}(3))^{G_3}
 =\mathbb F_3\bigl[(xy^3z-x^3yz)U\bigr].
\]
\end{prop}
\begin{proof}
Only $\Pp^1$ contributes to $(\Hit_3)_5$, by
Lemma~\ref{lem:hit-images}. The relevant map is
$\Pp^1:(P_3)_3\longrightarrow(P_3)_5$, and $(P_3)_3$ has a basis consisting
of its ten degree-$3$ monomials. The cubes $x^3,y^3,z^3$ map to zero. The six
monomials involving exactly two variables give
$2x^4y+x^2y^3$ and $x^3y^2+2xy^4$, together with the
corresponding two expressions for each of the pairs
$(x,z)$ and $(y,z)$. The image of $xyz$ is
\[
 A+B+C=x^3yz+xy^3z+xyz^3.
\]
These seven nonzero images have pairwise disjoint supports.
They are therefore independent, and
$\dim(\QP_3)_5=21-7=14$.

Proposition~\ref{prop:torus-block} requires all three polynomial
exponents to be odd in a torus-fixed top exterior class.
The only monomials of degree $5$ with that property are
$A,B,C$. Their only hit relation is $A+B+C=0$, so this
torus-fixed space has dimension two. The interchange of
$x$ and $y$ interchanges $A$ and $B$ and multiplies $U$
by $-1$. Its fixed vectors are the multiples of $(B-A)U$.

This class is fixed by all diagonal substitutions.
The interchange of $y$ and $z$ sends $(B-A)U$ to $(A-C)U$.
The difference is $(2A-B-C)U$, which is zero modulo
$A+B+C$ in characteristic $3$. Thus both adjacent
interchanges fix the class. The transvection $y\mapsto y-x$
fixes $U$ and satisfies
\[
 x(y-x)^3z-x^3(y-x)z=xy^3z-x^3yz.
\]
Conjugation by permutations gives the other coordinate
transvections. Lemma~\ref{lem:gl-generators} therefore
proves invariance under $G_3$. The class is nonzero because
$B-A$ is not a multiple of $A+B+C$. Finally,
Corollary~\ref{cor:central-top} identifies this top exterior
invariant line with $(\QH^{13}(3))^{G_3}$.
\end{proof}

For the other finite cases at $p=3$, define
\begin{align}
 f_{3,13}={}&2xy^7z^5+2x^3y^5z^5+x^5y^3z^5
               +2x^5y^7z+x^7yz^5+x^7y^5z,
               \label{eq:rankthree-f13}\\
 F_{3,57}={}&2x^5y^{35}z^{17}+x^{17}y^{23}z^{17}
      +2x^{23}y^{17}z^{17}+x^{35}y^5z^{17}\notag\\
     &\quad+2x^{35}y^{17}z^5+x^{47}y^5z^5,
               \label{eq:rankthree-f57}\\
 G_{3,57}={}&x^3yz^{53}+2x^3y^{49}z^5+2x^3y^{53}z
      +x^{41}y^{11}z^5\notag\\
     &\quad+x^{51}yz^5+2x^{51}y^5z+x^{53}y^3z
      +2x^{47}y^5z^5,
               \label{eq:rankthree-g57}\\
 f_{3,65}={}&x^7y^5z^{53}+2x^7y^{41}z^{17}+2x^7y^{53}z^5
                +x^{25}y^{23}z^{17}\notag\\
     &\quad+2x^{43}y^5z^{17}+x^{43}y^{17}z^5+x^{53}y^7z^5.
               \label{eq:rankthree-f65}
\end{align}
The representative in polynomial degree $179$ is
\begin{align}
 f_{3,179}={}&x(2y^{17}z^{161}+y^{125}z^{53}+y^{161}z^{17})
                 \notag\\
 &+x^{13}(y^5z^{161}+2y^{149}z^{17}+2y^{161}z^5)
                 \notag\\
 &+x^{17}(yz^{161}+2y^{157}z^5+2y^{161}z)
                 \notag\\
 &+x^{121}(2y^5z^{53}+y^{41}z^{17}+y^{53}z^5)
                 \notag\\
 &+x^{125}(2yz^{53}+y^{49}z^5+y^{53}z)
                 \notag\\
 &+x^{151}(y^{11}z^{17}+2y^{23}z^5)
   +x^{157}(y^5z^{17}+y^{17}z^5)
                 \notag\\
 &+x^{161}(2yz^{17}+y^{13}z^5+y^{17}z).
                 \label{eq:rankthree-f179}
\end{align}
At $p=5$, put
\begin{equation}\label{eq:rankthree-f21}
 f_{5,21}=4x^3y^{15}z^3+3x^7y^{11}z^3
              +2x^{11}y^7z^3+x^{15}y^3z^3.
\end{equation}
Each second subscript is the polynomial degree of its
representative. All coefficients in these expressions are
interpreted in the indicated field.

\begin{prop}\label{prop:finite-invariant-bases}
Let $(p,n)$ be one of the six pairs in
Table~\ref{tab:rankthree-invariants}. Then the displayed
classes form a basis of $(\QH^n(3))^{G_3}$.
\begin{center}
\begin{tabular}{@{}ccl@{}}
\toprule
$p$ & $n$ & Basis\\
\midrule
$3$ & $13$ & $\{[(xy^3z-x^3yz)U]\}$\\
$3$ & $29$ & $\{[f_{3,13}U]\}$\\
$3$ & $117$ & $\{[F_{3,57}U],[G_{3,57}U]\}$\\
$3$ & $133$ & $\{[f_{3,65}U]\}$\\
$3$ & $361$ & $\{[f_{3,179}U]\}$\\
$5$ & $45$ & $\{[f_{5,21}U]\}$\\
\bottomrule
\end{tabular}
\captionof{table}{Explicit rank-three invariant bases.}
\label{tab:rankthree-invariants}
\end{center}
\end{prop}
\begin{proof}
Corollary~\ref{cor:central-top} identifies the invariant space in each
listed degree with the invariants of
$(\QP_3)_m\otimes\det^{-1}$, where $m=(n-3)/2$.
We determine the latter space by first taking invariants under the
diagonal subgroup $T$ and then imposing the remaining generators of
$G_3$.  This separates the dimension of the polynomial quotient from
the equations determining its general linear invariants.

Put $d=\dim(\QP_3)_m$, and let $R$ be the reduction to the nonpivot monomials of
Proposition~\ref{prop:hit-matrix}, with the monomials ordered
lexicographically by their exponent triples.  Since each Cartan
summand preserves every exponent modulo $p-1$, the hit matrix splits
into residue blocks.  Proposition~\ref{prop:torus-block} therefore
identifies the $T$-fixed subspace with the quotient of the block
\[
 a\equiv b\equiv c\equiv p-2\pmod{p-1}.
\]
If $m-3(p-2)=(p-1)\ell$ with $\ell\geq0$, this block has
$\binom{\ell+2}{2}$ ambient monomials: its exponent triples are
$(p-2+(p-1)i,p-2+(p-1)j,p-2+(p-1)k)$ with $i+j+k=\ell$.
Write $r_T$ for the hit rank in this block and $t$ for its quotient
dimension.  Substitution in \eqref{eq:hit-matrix-entry}, followed by
the reduction defining $R$, gives
\begin{center}
\begin{tabular}{@{}rrrrrrrr@{}}
\toprule
$p$ & $m$ & $\binom{m+2}{2}$ & $\rank M_{3,p,m}$ &
$\dim(\QP_3)_m$ & $\binom{\ell+2}{2}$ & $r_T$ & $t$\\
\midrule
$3$ & $5$ & $21$ & $7$ & $14$ & $3$ & $1$ & $2$\\
$3$ & $13$ & $105$ & $81$ & $24$ & $21$ & $15$ & $6$\\
$3$ & $57$ & $1711$ & $1668$ & $43$ & $406$ & $393$ & $13$\\
$3$ & $65$ & $2211$ & $2198$ & $13$ & $528$ & $521$ & $7$\\
$3$ & $179$ & $16290$ & $16238$ & $52$ & $4005$ & $3983$ & $22$\\
$5$ & $21$ & $253$ & $156$ & $97$ & $10$ & $6$ & $4$\\
\bottomrule
\end{tabular}
\captionof{table}{Polynomial quotient dimensions and diagonal fixed
blocks in the six degrees.}
\label{tab:rankthree-ranks}
\end{center}
The nonpivot exponent triples in these blocks are the following
ordered tuples, denoted by $\mathcal B_{p,m}$.  For an entry
$(a_j,b_j,d_j)$ of such a tuple, put
$v_j=x^{a_j}y^{b_j}z^{d_j}$.  The classes
$[v_jU]$, $1\leq j\leq t$, form a basis of the diagonal
fixed subspace.
\begin{align*}
\mathcal B_{3,5}={}&\bigl((1,3,1),(3,1,1)\bigr),\\
\mathcal B_{3,13}={}&\bigl((1,7,5),(3,5,5),(5,3,5),
                         (5,7,1),(7,1,5),(7,5,1)\bigr),\\
\mathcal B_{3,57}={}&\bigl((3,1,53),(3,49,5),(3,53,1),(5,35,17),\\
 &\qquad (17,23,17),(23,17,17),(35,5,17),(35,17,5),\\
 &\qquad (41,11,5),(47,5,5),(51,1,5),(51,5,1),(53,3,1)\bigr),\\
\mathcal B_{3,65}={}&\bigl((7,5,53),(7,41,17),(7,53,5),(25,23,17),\\
 &\qquad (43,5,17),(43,17,5),(53,7,5)\bigr),\\
\mathcal B_{3,179}={}&\bigl((1,17,161),(1,125,53),(1,161,17),(13,5,161),\\
 &\qquad (13,149,17),(13,161,5),(17,1,161),(17,157,5),\\
 &\qquad (17,161,1),(121,5,53),(121,41,17),(121,53,5),\\
 &\qquad (125,1,53),(125,49,5),(125,53,1),(151,11,17),\\
 &\qquad (151,23,5),(157,5,17),(157,17,5),(161,1,17),\\
 &\qquad (161,13,5),(161,17,1)\bigr),\\
\mathcal B_{5,21}={}&\bigl((3,15,3),(7,11,3),(11,7,3),(15,3,3)\bigr).
\end{align*}

Write $f=\sum_{j=1}^t c_jv_j$.
Let $\sigma_{12}$ and $\sigma_{23}$ interchange the first two
and the last two variables, respectively, and let
$\tau(x)=x$, $\tau(y)=y-x$, $\tau(z)=z$.
Both interchanges send $U$ to $-U$, while $\tau(U)=U$.
Thus the remaining invariance equations are the coordinates of
\begin{equation}\label{eq:rankthree-three-defects}
\begin{split}
 D_{12}(f)&=R\bigl(-f(y,x,z)-f(x,y,z)\bigr),\\
 D_{23}(f)&=R\bigl(-f(x,z,y)-f(x,y,z)\bigr),\\
 D_{\tau}(f)&=R\bigl(f(x,y-x,z)-f(x,y,z)\bigr).
\end{split}
\end{equation}
All coordinates of these reductions are retained, including the
coordinates outside the diagonal fixed subspace.  In particular,
the last equation detects the components of a transvection image
with nontrivial diagonal character.
The two interchanges generate the permutation subgroup.
Conjugating $\tau$ by these permutations gives the transvections
for every ordered pair of distinct coordinates; powers give every
parameter in the prime field.  Lemma~\ref{lem:gl-generators}
therefore shows that the three equations in
\eqref{eq:rankthree-three-defects}, together with diagonal
invariance, are sufficient for $G_3$-invariance.

For each of the three defects, let $\mathcal E_{12}$,
$\mathcal E_{23}$ and $\mathcal E_\tau$ be the set of its nonzero
coordinate forms, after dividing each form by its first nonzero
coefficient and removing repeated forms.  Hence the corresponding
defect vanishes if and only if every element of its displayed set
vanishes.  These forms are obtained by applying $R$ to each monomial
in \eqref{eq:rankthree-three-defects}; for the transvection one uses
\[
 x^a(y-x)^bz^d
 =\sum_{i=0}^{b}(-1)^{b-i}\binom bi x^{a+b-i}y^iz^d
 \quad\text{in }\Fp[x,y,z].
\]
We give the forms explicitly, so that the dimensions and the
coefficients of the invariant representatives follow from the
resulting small systems.

For $(p,m)=(3,5)$, the three sets are all
$\{c_1+c_2\}$.  Thus the coefficient vector is
$(c_1,c_2)=c_1(1,-1)$, in agreement with
Proposition~\ref{prop:degree-thirteen}.

For $(p,m)=(3,13)$, the sets are
\begin{align*}
 \mathcal E_{12}&=\{c_1+c_5,\ c_2+c_3,\ c_4+c_6\},\\
 \mathcal E_{23}&=\{c_1-c_2,\ c_3+c_4,\ c_4-c_5-c_6,
                          \ c_3+c_5+c_6\},\\
 \mathcal E_\tau&=\{c_1-c_2,\ c_1+c_2+c_3+c_5,\ c_4+c_6\}.
\end{align*}
The first set gives $c_5=-c_1$, $c_3=-c_2$ and $c_6=-c_4$.
The first two forms in the second set then give
$c_2=c_1$ and $c_4=-c_3=c_1$.
Consequently the only possible vectors are
\[
 (c_1,\ldots,c_6)=c_6(-1,-1,1,-1,1,1).
\]
Substituting this expression into every displayed form gives zero.
The coefficient $c_6$ is unrestricted, and the vector with $c_6=1$
is precisely the coefficient vector of $f_{3,13}$.
Its last coordinate is nonzero in the nonpivot basis, so its class
is nonzero and spans the invariant subspace.

For $(p,m)=(3,57)$, the complete sets of forms are
\begin{align*}
 \mathcal E_{12}={}&\{c_2+c_{11},\ c_3+c_{13},\ c_4+c_7,
                         \ c_5+c_6,\ c_8+c_9+c_{10}\},\\
 \mathcal E_{23}={}&\{c_1+c_3,\ c_5+c_6,\ c_7+c_8,
                         \ c_4+c_9+c_{10},\ c_{11}+c_{12}\},\\
 \mathcal E_\tau={}&\left\{
 \begin{gathered}
 c_5+c_6,\quad c_4-c_6,\quad c_4-c_5+c_6,\quad c_4+c_5,\\
 c_2-c_7-c_9+c_{10},\quad
 c_4+c_5+c_6+c_7+c_8+c_9+c_{10},\\
 c_2+c_5+c_7-c_8+c_9,\quad c_3-c_{12}
 \end{gathered}\right\}.
\end{align*}
Put $a=c_5$ and $b=c_1$.
The equations in $\mathcal E_{12}$ first give
\[
 c_{11}=-c_2,\quad c_{13}=-c_3,\quad c_7=-c_4,
 \quad c_6=-a,\quad c_{10}=-c_8-c_9.
\]
The equations $c_1+c_3=0$, $c_7+c_8=0$ and
$c_{11}+c_{12}=0$ give
$c_3=-b$, $c_8=c_4$ and $c_{12}=c_2$.
The transvection equations $c_4-c_6=0$ and
$c_3-c_{12}=0$ now imply
$c_4=-a$ and $c_2=-b$.
It follows that
\[
 c_7=a,\quad c_8=-a,\quad c_{10}=a-c_9,\quad
 c_{11}=b,\quad c_{12}=-b,\quad c_{13}=b.
\]
The remaining equation $c_2-c_7-c_9+c_{10}=0$ becomes
$-b-2c_9=0$; since $-2=1$ in $\mathbb F_3$, it gives
$c_9=b$.  Thus every invariant coefficient vector must be
\begin{equation}\label{eq:rankthree-57-coordinates}
 (c_1,\ldots,c_{13})
 =(b,-b,-b,-a,a,-a,a,-a,b,a-b,b,-b,b).
\end{equation}
Every form in the three displayed sets vanishes on this vector.
For instance, the two longest transvection forms become
\[
 -a+a-a+a-a+b+(a-b)=0,
 \qquad -b+a+a+a+b=3a=0,
\]
and the other forms cancel term by term.  Hence $a$ and $b$ are
independent free parameters.  Taking $(a,b)=(1,0)$ yields
$F_{3,57}$, and taking $(a,b)=(0,1)$ yields $G_{3,57}$.
The coordinates at $x^{17}y^{23}z^{17}$ and $x^3yz^{53}$ are
$a$ and $b$, respectively.  These are distinct nonpivot coordinates,
so the two displayed classes are independent and form a basis.

For $(p,m)=(3,65)$, the systems are
\[
 \mathcal E_{12}=\{c_2-c_5,\ c_3+c_7\},\qquad
 \mathcal E_{23}=\{c_1+c_3,\ c_5+c_6\},\qquad
 \mathcal E_\tau=\{c_2+c_4,\ c_3+c_6\}.
\]
The first four equations give
$c_5=c_2$, $c_7=-c_3$, $c_3=-c_1$ and $c_6=-c_2$.
The last two give $c_4=-c_2$ and $c_6=-c_3=c_1$.
Consequently $c_2=-c_1$, and the general solution is
\[
 (c_1,\ldots,c_7)=c_1(1,-1,-1,1,-1,1,1).
\]
This vector satisfies every displayed equation.  Its first coordinate
is free, and setting $c_1=1$ gives $f_{3,65}$ with a nonzero
nonpivot coordinate.  Therefore its class is a basis of the
one-dimensional invariant subspace.

For $(p,m)=(3,179)$, the permutation equations are
\begin{align*}
 \mathcal E_{12}={}&\left\{
 \begin{gathered}
 c_1+c_7,\quad c_2+c_{13},\quad c_3+c_{20},\quad
 c_5-c_{10}-c_{16}+c_{18},\\
 c_6+c_{21},\quad c_8+c_{19},\quad c_9+c_{22},\quad
 c_{12}-c_{14}
 \end{gathered}\right\},\\
 \mathcal E_{23}={}&\left\{
 \begin{gathered}
 c_1+c_3,\quad c_4+c_6,\quad c_7+c_9,\quad
 c_{10}+c_{12},\quad c_{13}+c_{15},\\
 c_5+c_8+c_{11}+c_{14}-c_{16}-c_{17},\\
 c_5-c_8+c_{11}-c_{14}+c_{17}-c_{18}-c_{19},\\
 c_5+c_{11}+c_{16}+c_{18}+c_{19},\quad c_{20}+c_{22}
 \end{gathered}\right\}.
\end{align*}
The transvection equations are
\[
 \mathcal E_\tau=\left\{
 \begin{gathered}
 c_1+c_4,\quad c_2-c_3+c_5-c_6,\quad c_2+c_{10},\quad
 c_{11}-c_{12},\quad c_5+c_{11},\\
 c_2-c_3-c_6-c_{12},\quad
 c_2-c_3-c_5-c_6-c_{11}-c_{12},\\
 c_2-c_3-c_6-c_8-c_{12}-c_{14}-c_{16}-c_{17},\\
 c_9+c_{15},\quad c_3+c_{16}+c_{18},\\
 c_5+c_6+c_{11}+c_{12}-c_{16}+c_{17}-c_{19},\\
 c_3+c_5+c_{11}+c_{16}+c_{18},\quad
 c_6+c_8+c_{12}+c_{14}-c_{17}-c_{19}
 \end{gathered}\right\}.
\]
We solve these equations in a fixed order.  Put $a=c_{22}$.
The equations $c_9+c_{22}=0$, $c_7+c_9=0$, $c_1+c_7=0$,
$c_1+c_3=0$ and $c_3+c_{20}=0$ give, successively,
\[
 c_9=-a,\qquad c_7=a,\qquad c_1=-a,\qquad
 c_3=a,\qquad c_{20}=-a.
\]
Using $c_1+c_4=0$, $c_4+c_6=0$ and $c_6+c_{21}=0$ gives
$c_4=a$, $c_6=-a$ and $c_{21}=a$.
Next, $c_9+c_{15}=0$, $c_{13}+c_{15}=0$ and
$c_2+c_{13}=0$ give
$c_{15}=a$, $c_{13}=-a$ and $c_2=a$.
The equations $c_2+c_{10}=0$, $c_{10}+c_{12}=0$,
$c_{12}-c_{14}=0$, $c_{11}-c_{12}=0$ and
$c_5+c_{11}=0$ then give
\[
 c_{10}=-a,\qquad c_{12}=a,\qquad c_{14}=a,
 \qquad c_{11}=a,\qquad c_5=-a.
\]
The variables still undetermined are $c_8,c_{16},c_{17},c_{18},c_{19}$.
The equations $c_8+c_{19}=0$ and
$c_3+c_{16}+c_{18}=0$ give
$c_{19}=-c_8$ and $c_{16}+c_{18}=-a$.
After the preceding substitutions, the equation
$c_5-c_{10}-c_{16}+c_{18}=0$ becomes
$-c_{16}+c_{18}=0$.
Consequently $2c_{16}=-a$ and therefore
$c_{16}=c_{18}=a$ in $\mathbb F_3$.
The equation
$c_5+c_8+c_{11}+c_{14}-c_{16}-c_{17}=0$
reduces to $c_8-c_{17}=0$.
Finally,
$c_6+c_8+c_{12}+c_{14}-c_{17}-c_{19}=0$
reduces to $a+c_8=0$, which gives
$c_8=c_{17}=-a$ and $c_{19}=a$.
Thus the complete coefficient vector is
\begin{equation}\label{eq:rankthree-179-coordinates}
\begin{split}
 (c_1,\ldots,c_{22})=a(&-1,1,1,1,-1,-1,1,-1,-1,-1,1,\\
                     &1,-1,1,1,1,-1,1,1,-1,1,1).
\end{split}
\end{equation}
For completeness, consider the equations not used to obtain this
vector.  In $\mathcal E_{12}$ there are none.  In
$\mathcal E_{23}$, the two remaining longer forms evaluate to
$-3a$ and $3a$, and $c_{20}+c_{22}$ evaluates to zero.
In $\mathcal E_\tau$, the six unused forms, in their displayed order,
evaluate to
\[
 0,\quad 0,\quad 0,\quad 0,\quad -3a,\quad 3a.
\]
Every remaining equation therefore vanishes in $\mathbb F_3$.
The parameter $a$ is arbitrary, and the value $a=1$ gives exactly
$f_{3,179}$.  Its $22$nd coordinate is $1$, proving that its class
is nonzero.  This proves that it spans the invariant subspace.

For $(p,m)=(5,21)$, the systems over $\mathbb F_5$ are
\begin{align*}
 \mathcal E_{12}&=\{c_1+c_4,\ c_2+c_3\},\\
 \mathcal E_{23}&=\{c_1+2c_2,\ c_1-2c_2+2c_3-2c_4\},\\
 \mathcal E_\tau&=\{c_1+2c_2,\ c_2+c_3,\ c_1-c_2+2c_3,
                    \ c_3-2c_4,\ c_2+2c_3-2c_4,
                    \ c_1+c_2+c_3+c_4\}.
\end{align*}
The first two equations give $c_4=-c_1$ and $c_3=-c_2$.
The equation $c_1+2c_2=0$ gives $c_2=2c_1$ because
$2^{-1}=3$ in $\mathbb F_5$.  It follows that
\[
 (c_1,c_2,c_3,c_4)=c_4(-1,-2,2,1)
                  =c_4(4,3,2,1).
\]
Substitution into the other permutation equation gives
$-1+4+4-2=5=0$.  The six transvection forms evaluate,
in order, to $-5,0,5,0,0,0$, multiplied by $c_4$.
Thus every displayed equation holds, and $c_4$ is free.
The vector with $c_4=1$ is the coefficient vector of $f_{5,21}$
and has a nonzero last coordinate.

We have therefore obtained a one-dimensional invariant space in
five cases and a two-dimensional invariant space for $(p,m)=(3,57)$,
with the bases asserted in the proposition.  Equivalently, the
ranks of the systems on the diagonal fixed subspaces are
$1,5,11,6,21,3$, in the order of the table.
The three diagonal generators have common kernel of dimension $t$
inside the $d$-dimensional polynomial quotient.  Adding the
remaining invariance equations increases the rank by the rank of
their restrictions to that common kernel.  Consequently the full
fixed-matrix ranks are $13,23,41,12,51,96$, respectively.
Finally, multiplication by $U$ and the central-character
identification at the start of the proof give the stated bases of
$(\QH^n(3))^{G_3}$.
\end{proof}

The finite transfer statements now follow by comparing these
values of $\dim D_{3,n}$ with the dimensions of
$\Ext^{3,n+3}_{\A}(\Fp,\Fp)$. They are
independent of any assertion about an invariant line in
all degrees of an infinite family.

\begin{thm}\label{thm:finite-isomorphisms}
Let $p$ be odd. If
\[
 (p,n)\in\{(3,13),(3,29),(3,117),(3,133),(3,361),(5,45)\},
\]
then $\Tr_{3,n}$ is an isomorphism.
Its domain and codomain have dimension two for $(p,n)=(3,117)$
and dimension one in the other five cases.
\end{thm}
\begin{proof}
The parameter choices in
\eqref{eq:family1}--\eqref{eq:family4} are, respectively,
family~$2$ with $(i,j)=(0,0)$, family~$3$ with $(i,j)=(0,1)$,
family~$2$ with $(i,j)=(0,3)$, family~$4$ with $(i,j)=(0,3)$,
family~$1$ with $(k,j,i)=(0,2,4)$, and family~$2$ with
$(i,j)=(0,0)$. Direct substitution gives the six degrees
and verifies every index restriction.
Theorem~\ref{thm:surjectivity} therefore makes the transfer
surjective in each case. Its target dimension is given by
Theorem~\ref{thm:target-dimensions}, and $\dim D_{3,n}$ has the same value
by Proposition~\ref{prop:finite-invariant-bases}
and~\eqref{eq:coinvariant-duality}. Rank-nullity then gives a zero
kernel in each of the six cases.
\end{proof}

\subsection{Degree reduction and six infinite families}
\label{subsec:infinite-transfer}

The finite calculations in Subsection~\ref{subsec:rankthree-calculations}
extend to infinite families once the degree hypothesis for Minami's
extraction is checked at each step. Two of the six initial degrees
satisfy that hypothesis at every subsequent step. The other four
require one additional invariant calculation before the reduction
applies. We first determine the necessary starting degrees and then
combine the resulting invariant dimensions with the surjectivity
and target dimensions already established.

The following criterion verifies the numerical hypothesis
along a sequence of simultaneous increases of the
base-$p$ exponents. Its last assertion uses the exterior
representation and the duality already established for
$D_{3,n}$.

\begin{prop}\label{prop:digit-stability}
Let $p$ be an odd prime, let $m_0\in\NN$, and put
\[
 R=3(p-1),\qquad A=(p-1)(m_0+3),\qquad
 m_t=p^t(m_0+3)-3\quad(t\ge0).
\]
Let $r\ge1$ be an integer. If
\begin{equation}\label{eq:rankthree-stability-condition}
 \ds(A)\le R,
 \qquad \ds(A-1)+r(p-1)>R-1,
\end{equation}
then $\mu_p(m_r)=R$ and
\[
 \overline\kappa_3\otimes1:
 (\QP_3)_{m_r}\otimes\langle U\rangle
 \longrightarrow
 (\QP_3)_{m_{r-1}}\otimes\langle U\rangle
\]
is a $G_3$-equivariant isomorphism. If, in addition,
$p-1$ divides $m_0+3$ and $n_t=2m_t+3$, then it induces
an isomorphism
\[
 D_{3,n_{r-1}}\cong D_{3,n_r}.
\]
\end{prop}
\begin{proof}
Since $m_0\ge0$, every $m_t$ is a nonnegative integer.
The definitions give
\[
 m_r(p-1)+R=A p^r.
\]
Multiplication by $p^r$ appends $r$ zero digits, so
\[
 \ds\bigl(m_r(p-1)+R\bigr)=\ds(A)\le R.
\]
Applying~\eqref{eq:minami-digit-criterion} with $b=R$
yields $\mu_p(m_r)\le R$.

The integer $A$ is positive. Hence
\[
 m_r(p-1)+R-1=A p^r-1
 =(A-1)p^r+\sum_{i=0}^{r-1}(p-1)p^i.
\]
The two terms on the right occupy disjoint base-$p$ digit
positions. Therefore
\[
 \ds\bigl(m_r(p-1)+R-1\bigr)
 =\ds(A-1)+r(p-1)>R-1.
\]
By~\eqref{eq:minami-digit-criterion} with $b=R-1$, one
has $\mu_p(m_r)>R-1$. Since $\mu_p(m_r)$ is an integer,
this inequality and $\mu_p(m_r)\le R$ imply
$\mu_p(m_r)=R$.

The degree relation is
\[
 m_r=pm_{r-1}+3(p-1).
\]
Lemma~\ref{lem:extraction-equivariant}, with $h=3$ and
$d=m_{r-1}$, now gives the asserted equivariant
isomorphism on the cohit spaces tensored with
$\langle U\rangle=\Lambda^3(V_3^*)$.

Assume also that $p-1$ divides $m_0+3$. For each $t\ge0$,
the integer
\[
 A_t=\frac{p^t(m_0+3)}{p-1}
\]
satisfies $n_t=2(p-1)A_t-3$. Corollary~\ref{cor:central-top}
therefore identifies
\[
 (\QH^{n_t}(3))^{G_3}
 \cong\bigl((\QP_3)_{m_t}\otimes\langle U\rangle\bigr)^{G_3}.
\]
The equivariant isomorphism just proved restricts to an
isomorphism of fixed subspaces. Indeed, both that map and
its inverse commute with $G_3$, so each takes a fixed
vector to a fixed vector. The resulting isomorphism is
\[
 (\QH^{n_r}(3))^{G_3}
 \longrightarrow(\QH^{n_{r-1}}(3))^{G_3}.
\]
By~\eqref{eq:coinvariant-duality}, these spaces are
$D_{3,n_r}^*$ and $D_{3,n_{r-1}}^*$, respectively.
Dualizing the isomorphism and applying finite-dimensional
double duality yields
$D_{3,n_{r-1}}\cong D_{3,n_r}$.
\end{proof}

The same digit criterion gives a reduction valid for every odd prime
and every admissible choice of the indices in the four principal
families. The required starting index is bounded independently of
the gaps between those indices.

\begin{thm}\label{thm:uniform-common-shift}
Let $p$ be an odd prime, let $n_0$ be a degree in one of
\eqref{eq:family1}--\eqref{eq:family4}, and put
\[
 m_0=\frac{n_0-3}{2},\qquad
 m_r=p^r(m_0+3)-3,\qquad n_r=2m_r+3
 \quad(r\geq0).
\]
If $r\geq2$, then
\[
 \overline\kappa_3\otimes1:
 (\QP_3)_{m_r}\otimes\langle U\rangle
 \longrightarrow
 (\QP_3)_{m_{r-1}}\otimes\langle U\rangle
\]
is a $G_3$-equivariant isomorphism and
$D_{3,n_{r-1}}\cong D_{3,n_r}$.
If $n_0$ belongs to~\eqref{eq:family1}, then both conclusions
hold for every $r\geq1$.
\end{thm}
\begin{proof}
Write $n_0=2(p-1)N-3$, and put
\[
 C=(p-1)^2,\qquad A=CN=(p-1)(m_0+3),\qquad R=3(p-1).
\]
We will prove
\begin{equation}\label{eq:uniform-digit-bounds}
 \ds(A)\leq3(p-1),\qquad \ds(A-1)\geq2p-3.
\end{equation}
If $B$ is a positive integer not divisible by $p$ and $v\geq0$,
then
\begin{equation}\label{eq:digit-subtract-shift}
 \ds(p^vB-1)=\ds(B)-1+v(p-1).
\end{equation}
Indeed, subtracting one from the nonzero lowest digit of $B$
decreases its digit sum by one, and the $v$ lower positions
in $p^vB-1$ all have digit $p-1$.
The following expansions will be used:
\begin{align*}
 C&=1+(p-2)p,\\
 C(1+p)&=1+(p-1)p+(p-2)p^2,\\
 C(p+2)&=2+(p-3)p+(p-1)p^2,\\
 C(2p+1)&=1+(p-3)p^2+p^3,\\
 C(2p+2)&=2+(p-2)p+(p-3)p^2+p^3.
\end{align*}
Each coefficient displayed here lies between zero and $p-1$.
Consequently these are base-$p$ expansions, including when
$p=3$ and a coefficient $p-3$ is zero.

In family~$1$, the three translates of the two-digit block $C$
occur at $k,j,i$, with consecutive starting positions separated
by at least two. Their digit positions are disjoint. Thus
\[
 \ds(A)=3(p-1),\qquad
 \ds(A-1)=(k+3)(p-1)-1\geq3p-4.
\]
The second equality follows from~\eqref{eq:digit-subtract-shift}
and the fact that the lowest nonzero digit is at position $k$.

In family~$2$, put $a=i+1$ and $b=j$, so that
$A=C(p^a+p^b)$. Suppose first that $a\ne b$, and let
$v=\min(a,b)$. If $|a-b|\geq2$, the two blocks are disjoint
and have total digit sum $2(p-1)$. If $|a-b|=1$, the expansion
of $C(1+p)$ gives the same digit sum. In either case
\[
 \ds(A)=2(p-1),\qquad
 \ds(A-1)=(v+2)(p-1)-1\geq2p-3.
\]
If $a=b$, then $a\geq1$ and $A=2Cp^a$.
For $p=3$, one has $2C=2+2p$, so
$\ds(A)=4=2(p-1)$ and
$\ds(A-1)=3+2a\geq2p-3$.
For $p\geq5$, the expansion
\[
 2C=2+(p-4)p+p^2
\]
gives $\ds(A)=p-1$ and
$\ds(A-1)=(a+1)(p-1)-1\geq2p-3$.

In family~$3$, one has
$A=C(p+2)p^i+Cp^j$.
The allowed values of $j$ are $j\leq i-2$, $j=i+1$,
or $j\geq i+3$. In the first and third cases, the displayed
expansions of $C(p+2)$ and $C$ occupy disjoint positions.
Their total digit sum is $3(p-1)$. With $v=\min(i,j)$,
\eqref{eq:digit-subtract-shift} gives
$\ds(A-1)=(v+3)(p-1)-1$.
In the remaining case, $j=i+1$, one has
$A=C(2p+2)p^i$. The displayed expansion then gives
\[
 \ds(A)=2(p-1),\qquad
 \ds(A-1)=(i+2)(p-1)-1.
\]
Both alternatives satisfy~\eqref{eq:uniform-digit-bounds}.

In family~$4$, one has
$A=C(2p+1)p^i+Cp^j$, where either $j\leq i-2$
or $j\geq i+3$. The two blocks are disjoint unless
$j=i+3$. In that case their only overlapping digits are
the two coefficients $1$ at position $i+3$; their sum is
$2\leq p-1$, so no carry occurs. Thus in all allowed cases
\[
 \ds(A)=2(p-1),\qquad
 \ds(A-1)=(\min(i,j)+2)(p-1)-1.
\]
This proves~\eqref{eq:uniform-digit-bounds} in the last family.

For $r\geq2$, those bounds imply
\[
 \ds(A-1)+r(p-1)
 \geq 2p-3+2(p-1)=4p-5>3p-4=R-1.
\]
Moreover, $p-1$ divides $m_0+3=(p-1)N$.
Proposition~\ref{prop:digit-stability} therefore gives both
asserted isomorphisms. In family~$1$, the stronger bound
$\ds(A-1)\geq3p-4=R-1$ makes the strict inequality hold
already for $r\geq1$, proving the last assertion.
\end{proof}

\begin{corl}\label{cor:uniform-boundary-reduction}
Let $p$, $n_0$ and $n_r$ satisfy the hypotheses of
Theorem~\ref{thm:uniform-common-shift}. If
$\Tr_{3,n_0}$ and $\Tr_{3,n_1}$ are isomorphisms, then
$\Tr_{3,n_r}$ is an isomorphism for every $r\geq0$.
If $n_0$ belongs to~\eqref{eq:family1} and
$\Tr_{3,n_0}$ is an isomorphism, then
$\Tr_{3,n_r}$ is an isomorphism for every $r\geq0$.
\end{corl}
\begin{proof}
Simultaneous increases of all the indices preserve the restrictions
in~\eqref{eq:family1}--\eqref{eq:family4}.
Theorem~\ref{thm:uniform-common-shift} identifies the dimensions
of $D_{3,n_r}$ for all $r\geq1$.
The target dimensions in Theorem~\ref{thm:target-dimensions}
are also constant for $r\geq1$. In family~$2$, this follows
because the difference $j-i-1$ is unchanged and $j+r\geq1$;
in the other three families the target dimension is one.
Surjectivity in Theorem~\ref{thm:surjectivity}
and equality of the domain and target dimensions at $r=1$
therefore give an isomorphism for every $r\geq1$.
The case $r=0$ is part of the hypothesis.
In family~$1$, both dimensions are constant starting at $r=0$,
so the same argument uses only the initial isomorphism.
\end{proof}

Subtracting the least index from every index normalizes a tuple
in~\eqref{eq:family1} to $k=0$, and a tuple in the other three
families to $\min(i,j)=0$.
Corollary~\ref{cor:uniform-boundary-reduction} reduces the
isomorphism problem to the normalized tuple in family~$1$,
and to the normalized tuple and its first simultaneous shift
in the other three families. The relative gaps and the odd prime
are still arbitrary. These remaining invariant dimensions require
additional bounds; the reduction does not replace them by a
finite collection of degrees.

For each initial pair $(p,n_0)$ in
Theorem~\ref{thm:finite-isomorphisms}, set
\begin{equation}\label{eq:transfer-shift-degrees}
 m_0=\frac{n_0-3}{2},\qquad
 m_r=p^r(m_0+3)-3,\qquad
 n_r=2m_r+3=p^r(n_0+3)-3.
\end{equation}
The digit sums in the reduction criterion are as follows. Here
$A=(p-1)(m_0+3)$, and the final column records the least positive
$r$ from which that criterion applies at every step.
\begin{center}
\begin{tabular}{@{}rrrrrrr@{}}
\toprule
$p$ & $n_0$ & $m_0$ & $A$ & $\ds(A)$ & $\ds(A-1)$ & $r\geq$\\
\midrule
$3$ & $13$ & $5$ & $16$ & $4$ & $3$ & $2$\\
$3$ & $29$ & $13$ & $32$ & $4$ & $3$ & $2$\\
$3$ & $117$ & $57$ & $120$ & $4$ & $5$ & $1$\\
$3$ & $133$ & $65$ & $136$ & $4$ & $3$ & $2$\\
$3$ & $361$ & $179$ & $364$ & $6$ & $5$ & $1$\\
$5$ & $45$ & $21$ & $96$ & $8$ & $7$ & $2$\\
\bottomrule
\end{tabular}
\captionof{table}{Digit sums and the first step covered by the reduction criterion.}
\label{tab:infinite-digit-sums}
\end{center}
For example, the two rows starting at $n_0=117$ and $n_0=361$
have $\ds(A-1)=5$, so $\ds(A-1)+2r>5$ for every $r\geq1$.
The other three rows at $p=3$ give $3+2r>5$ precisely when
$r\geq2$. The row at $p=5$ gives $7+4r>11$ precisely when
$r\geq2$. The first digit-sum inequality holds in all six rows.

The omitted first steps cannot be supplied by the same hypothesis.
In fact,
\begin{equation}\label{eq:first-shift-mu}
 \mu_3(21)=\mu_3(45)=\mu_3(201)=3,\qquad \mu_5(117)=7.
\end{equation}
These values follow directly from~\eqref{eq:minami-digit-criterion}:
the pairs of digit sums at the claimed value $b$ and at $b-1$ are
\[
\begin{array}{c|c|c|c}
(p,m)&b&\ds((p-1)m+b)&\ds((p-1)m+b-1)\\\hline
(3,21)&3&\ds(45)=3&\ds(44)=6\\
(3,45)&3&\ds(93)=3&\ds(92)=4\\
(3,201)&3&\ds(405)=3&\ds(404)=10\\
(5,117)&7&\ds(475)=7&\ds(474)=14
\end{array}
\]
In each row the first digit sum is at most $b$ and the second
exceeds $b-1$, proving the asserted minimum. Thus the four
degrees $(p,n)=(3,45),(3,93),(3,405),(5,237)$ require the
following independent calculations.

Use the notation for diagonal fixed blocks introduced in the
proof of Proposition~\ref{prop:finite-invariant-bases} and the
three defects in~\eqref{eq:rankthree-three-defects}. Define
\begin{align*}
 g_{3,21}={}&x^3yz^{17}-x^3y^{13}z^5-x^3y^{17}z
          +x^5y^{11}z^5-x^{11}y^5z^5\\
 &\quad+x^{15}yz^5-x^{15}y^5z+x^{17}y^3z,\\
 g_{3,45}={}&x^5y^{23}z^{17}+x^{11}y^{17}z^{17}
          -x^{17}y^{11}z^{17}+x^{17}y^{23}z^5\\
 &\quad-x^{23}y^5z^{17}-x^{23}y^{17}z^5,\\
 g_{3,201}={}&x^{23}y^{17}z^{161}-x^{23}y^{125}z^{53}
          -x^{23}y^{161}z^{17}+x^{77}y^{71}z^{53}\\
 &\quad-x^{131}y^{17}z^{53}+x^{131}y^{53}z^{17}
          +x^{161}y^{23}z^{17},\\
 g_{5,117}={}&x^{15}y^3z^{99}-x^{15}y^{83}z^{19}
          -x^{15}y^{99}z^3+x^{19}y^{79}z^{19}\\
 &\quad+2x^{39}y^{59}z^{19}-2x^{59}y^{39}z^{19}
          -x^{79}y^{19}z^{19}\\
 &\quad+x^{95}y^3z^{19}-x^{95}y^{19}z^3
          +x^{99}y^{15}z^3.
\end{align*}

\begin{prop}\label{prop:first-shift-invariants}
Let $(p,m)$ belong to
$\{(3,21),(3,45),(3,201),(5,117)\}$, and let $n=2m+3$.
Then
\[
 (\QH^{n}(3))^{G_3}=\Fp\,[g_{p,m}U],
 \qquad \dim D_{3,n}=1.
\]
\end{prop}
\begin{proof}
Corollary~\ref{cor:central-top} reduces the assertion to the
invariants of $(\QP_3)_m\otimes\det^{-1}$.
The Cartan image relations of Proposition~\ref{prop:hit-matrix},
with the same lexicographic order and reduction $R$ as above, give
the following dimensions.  Here $r_T$ and $t$ retain their preceding
meaning, and $\ell=(m-3(p-2))/(p-1)$.
\begin{center}
\begin{tabular}{@{}rrrrrrrr@{}}
\toprule
$p$ & $m$ & $\binom{m+2}{2}$ & $\rank M_{3,p,m}$ &
$\dim(\QP_3)_m$ & $\binom{\ell+2}{2}$ & $r_T$ & $t$\\
\midrule
$3$ & $21$ & $253$ & $227$ & $26$ & $55$ & $47$ & $8$\\
$3$ & $45$ & $1081$ & $1057$ & $24$ & $253$ & $247$ & $6$\\
$3$ & $201$ & $20503$ & $20490$ & $13$ & $5050$ & $5043$ & $7$\\
$5$ & $117$ & $7021$ & $6894$ & $127$ & $406$ & $396$ & $10$\\
\bottomrule
\end{tabular}
\captionof{table}{Polynomial quotient dimensions and diagonal fixed
blocks in the four additional degrees.}
\label{tab:four-shifted-ranks}
\end{center}
The corresponding ordered nonpivot tuples are
\begin{align*}
 \mathcal B_{3,21}={}&\bigl((3,1,17),(3,13,5),(3,17,1),(5,11,5),\\
 &\qquad (11,5,5),(15,1,5),(15,5,1),(17,3,1)\bigr),\\
 \mathcal B_{3,45}={}&\bigl((5,23,17),(11,17,17),(17,11,17),\\
 &\qquad (17,23,5),(23,5,17),(23,17,5)\bigr),\\
 \mathcal B_{3,201}={}&\bigl((23,17,161),(23,125,53),(23,161,17),\\
 &\qquad (77,71,53),(131,17,53),(131,53,17),(161,23,17)\bigr),\\
 \mathcal B_{5,117}={}&\bigl((15,3,99),(15,83,19),(15,99,3),
                            (19,79,19),\\
 &\qquad (39,59,19),(59,39,19),(79,19,19),\\
 &\qquad (95,3,19),(95,19,3),(99,15,3)\bigr).
\end{align*}
Thus Proposition~\ref{prop:torus-block} identifies each diagonal
fixed subspace with the span of the associated classes $[v_jU]$.
We retain the definition of $\mathcal E_{12}$, $\mathcal E_{23}$
and $\mathcal E_\tau$ as the normalized, distinct, nonzero
coordinate forms of the three defects.  In particular, the
displayed sets below retain forms that are linearly dependent on
other forms in the same set.

For $(p,m)=(3,21)$, reduction of the defects gives
\begin{align*}
 \mathcal E_{12}&=\{c_2+c_6,\ c_3+c_8,\ c_4+c_5\},\\
 \mathcal E_{23}&=\{c_1+c_3,\ c_4+c_5,\ c_6+c_7\},\\
 \mathcal E_\tau&=\{c_2-c_4+c_5,\ c_4+c_5,\ c_2+c_4,
                           \ c_3-c_7\}.
\end{align*}
Put $a=c_1$.  The first form in $\mathcal E_{23}$ gives
$c_3=-a$, and the second form in $\mathcal E_{12}$ gives
$c_8=a$.  The last form in $\mathcal E_\tau$ now gives
$c_7=-a$.  The last form in $\mathcal E_{23}$ gives $c_6=a$,
and the first form in $\mathcal E_{12}$ gives $c_2=-a$.
The third form in $\mathcal E_\tau$ gives $c_4=a$, after
which the last form in $\mathcal E_{12}$ gives $c_5=-a$.
Consequently every solution has the form
\[
 (c_1,\ldots,c_8)=a(1,-1,-1,1,-1,1,-1,1).
\]
On this vector the forms in $\mathcal E_{12}$ and
$\mathcal E_{23}$ all vanish.  In their displayed order, the
forms in $\mathcal E_\tau$ have values $(-3a,0,0,0)$, so
they also vanish in $\mathbb F_3$.  Hence $a$ is unrestricted.
The vector with $a=1$ is the coefficient vector of $g_{3,21}$.

For $(p,m)=(3,45)$, the three sets are
\begin{align*}
 \mathcal E_{12}&=\{c_1+c_5,\ c_2+c_3,\ c_4+c_6\},\\
 \mathcal E_{23}&=\{c_1-c_2,\ c_3+c_4,\ c_4-c_5-c_6,
                           \ c_3+c_5+c_6\},\\
 \mathcal E_\tau&=\{c_1-c_2,\ c_1+c_2+c_3+c_5,\ c_4+c_6\}.
\end{align*}
Put $a=c_1$.  The first form in $\mathcal E_{23}$ yields
$c_2=a$.  The first two forms in $\mathcal E_{12}$ then yield
$c_5=-a$ and $c_3=-a$.  The second form in $\mathcal E_{23}$
gives $c_4=a$, and the last form in $\mathcal E_{12}$ gives
$c_6=-a$.  Thus
\[
 (c_1,\ldots,c_6)=a(1,1,-1,1,-1,-1).
\]
All forms in $\mathcal E_{12}$ and $\mathcal E_\tau$ vanish
on this vector.  The forms in $\mathcal E_{23}$ have values
$(0,0,3a,-3a)$, which vanish in $\mathbb F_3$.  The value
$a=1$ gives $g_{3,45}$.

For $(p,m)=(3,201)$, the sets are
\begin{align*}
 \mathcal E_{12}&=\{c_2-c_5,\ c_3+c_7\},\\
 \mathcal E_{23}&=\{c_1+c_3,\ c_5+c_6\},\\
 \mathcal E_\tau&=\{c_2+c_4,\ c_3+c_6\}.
\end{align*}
Put $a=c_1$.  The first form in $\mathcal E_{23}$ gives
$c_3=-a$.  The second form in $\mathcal E_{12}$ gives $c_7=a$,
and the second form in $\mathcal E_\tau$ gives $c_6=a$.
The second form in $\mathcal E_{23}$ now gives $c_5=-a$.
The first forms in $\mathcal E_{12}$ and $\mathcal E_\tau$
give $c_2=-a$ and $c_4=a$, respectively.  All six forms vanish
on the resulting vector
\[
 (c_1,\ldots,c_7)=a(1,-1,-1,1,-1,1,1).
\]
Its value at $a=1$ is the coefficient vector of $g_{3,201}$.

For $(p,m)=(5,117)$, the normalized coordinate forms are
\begin{align*}
 \mathcal E_{12}&=\{c_2+c_8,\ c_3+c_{10},\ c_4+c_7,
                           \ c_5+c_6\},\\
 \mathcal E_{23}&=\{c_1+c_3,\ c_4+2c_5,
                           \ c_4-2c_5+2c_6-2c_7,\ c_8+c_9\},\\
 \mathcal E_\tau&=\{c_2-c_4+c_5,\ c_4+2c_5,\ c_5+c_6,\\
 &\qquad c_2-2c_4-2c_5-c_6,\ c_4-c_5+2c_6,
                           \ c_6-2c_7,\\
 &\qquad c_5+2c_6-2c_7,\ c_2+2c_4+c_5+c_6+c_7,\\
 &\qquad c_4+c_5+c_6+c_7,\ c_2+c_4,\ c_3-c_9\}.
\end{align*}
Put $a=c_1$.  The first form in $\mathcal E_{23}$ gives
$c_3=-a$, and the second form in $\mathcal E_{12}$ gives
$c_{10}=a$.  The last form in $\mathcal E_\tau$ then gives
$c_9=-a$.  The last form in $\mathcal E_{23}$ gives $c_8=a$,
and the first form in $\mathcal E_{12}$ gives $c_2=-a$.
The form $c_2+c_4$ gives $c_4=a$, after which the third form
in $\mathcal E_{12}$ gives $c_7=-a$.  The form $c_6-2c_7$
gives $c_6=-2a$, and the last form in $\mathcal E_{12}$ gives
$c_5=2a$.  Therefore
\[
 (c_1,\ldots,c_{10})=a(1,-1,-1,1,2,-2,-1,1,-1,1).
\]
On this vector, the forms in $\mathcal E_{12}$ all vanish.
In their respective displayed orders, the forms in
$\mathcal E_{23}$ and $\mathcal E_\tau$ take the values
\[
 (0,5a,-5a,0),\qquad
 (0,5a,0,-5a,-5a,0,0,0,0,0,0).
\]
Every entry is zero in $\mathbb F_5$, so $a$ is unrestricted.
The value $a=1$ gives $g_{5,117}$.

In each case the first coordinate of the displayed vector is $1$
when $a=1$.  Since these are coordinates in a quotient basis,
the associated class is nonzero.  The three defects vanish
precisely on the stated line, and diagonal invariance was built
into its basis.  Lemma~\ref{lem:gl-generators} therefore proves
that this line is the full $G_3$-invariant space.
Finally, \eqref{eq:coinvariant-duality} gives $\dim D_{3,n}=1$.
\end{proof}

The additional calculations determine all starting dimensions
needed in Table~\ref{tab:infinite-digit-sums}. The subsequent
steps are isomorphisms of the polynomial quotients as
$G_3$-modules, so taking the determinant twist and then fixed
vectors preserves their dimensions.

\begin{thm}\label{thm:infinite-isomorphisms}
Let $r\geq0$ be an integer. If $p=3$ and
\[
 n\in\{16\cdot3^r-3,\ 32\cdot3^r-3,\ 120\cdot3^r-3,
        \ 136\cdot3^r-3,\ 364\cdot3^r-3\},
\]
or if $p=5$ and $n=48\cdot5^r-3$, then
\[
 \Tr_{3,n}:D_{3,n}\longrightarrow\Ext_{\A}^{3,n+3}(\Fp,\Fp)
\]
is an isomorphism. Its domain and codomain have dimension two
when $p=3$ and $n=120\cdot3^r-3$, and dimension one in the
other five families.
\end{thm}
\begin{proof}
Use the initial pairs and degree notation
of~\eqref{eq:transfer-shift-degrees}. Every row has
$p-1\mid m_0+3$, and hence $n_r+3$ is divisible by $2(p-1)$
for every $r$. Corollary~\ref{cor:central-top} and
\eqref{eq:coinvariant-duality} therefore identify $D_{3,n_r}^*$
with the determinant-twisted polynomial invariant space at
degree $m_r$.

For $(p,n_0)=(3,117)$ and $(3,361)$,
Table~\ref{tab:infinite-digit-sums} and
Proposition~\ref{prop:digit-stability} give an equivariant
extraction isomorphism from degree $m_r$ to degree $m_{r-1}$
for every $r\geq1$. Applying these isomorphisms successively
identifies the invariant space at $m_r$ with the invariant
space at $m_0$. Proposition~\ref{prop:finite-invariant-bases}
gives dimensions two and one, respectively. Duality gives
the same dimensions for $D_{3,n_r}$.

For the remaining four initial pairs, the same proposition
gives dimension one at $r=0$. Proposition~\ref{prop:first-shift-invariants}
gives dimension one at $r=1$, at the four polynomial degrees listed
in~\eqref{eq:first-shift-mu}. For every $r\geq2$,
Table~\ref{tab:infinite-digit-sums} verifies the hypotheses
of Proposition~\ref{prop:digit-stability}. Successive
extractions therefore identify the invariant space at $m_r$
with the space at $m_1$. It follows that $\dim D_{3,n_r}=1$
for every $r\geq0$ in these four cases. This argument does
not assert that the first extraction is an isomorphism of
the entire polynomial quotients.

It remains to compare these dimensions with the targets
and verify surjectivity. The parameters in
\eqref{eq:family1}--\eqref{eq:family4} are
\begin{center}
\begin{tabular}{@{}ccclc@{}}
\toprule
$p$ & $n_0$ & Family & Parameters & Target dimension\\
\midrule
$3$ & $13$ & $2$ & $(i,j)=(r,r)$ & $1$\\
$3$ & $29$ & $3$ & $(i,j)=(r,r+1)$ & $1$\\
$3$ & $117$ & $2$ & $(i,j)=(r,r+3)$ & $2$\\
$3$ & $133$ & $4$ & $(i,j)=(r,r+3)$ & $1$\\
$3$ & $361$ & $1$ & $(k,j,i)=(r,r+2,r+4)$ & $1$\\
$5$ & $45$ & $2$ & $(i,j)=(r,r)$ & $1$\\
\bottomrule
\end{tabular}
\end{center}
For family~$1$, the two gap inequalities are equalities.
For family~$3$ in~\eqref{eq:family3}, the difference $j-i=1$ avoids all three
excluded values, and for family~$4$ the difference $j-i=3$
avoids all four excluded values. In family~$2$, the
difference $j-i$ is either $0$ or $3$, so $j\ne i+2$.
Thus every row satisfies its full index restrictions.
For the row with $n_0=117$, one has $j\geq1$ and
$|j-i-1|=2$, giving target dimension two. In the other
family~$2$ rows, $|j-i-1|=1$, so their targets have dimension
one. The remaining target dimensions follow from
Theorem~\ref{thm:target-dimensions}, which uses the complete
basis in~\cite[Theorem~8.1 and Table~8.1]{Aikawa}.
Theorem~\ref{thm:surjectivity} now gives a surjective transfer
in every row. Since the domain and codomain have equal finite
dimensions, the kernel has dimension zero. This proves every
asserted isomorphism.
\end{proof}

The extraction isomorphisms also give explicit invariant
representatives throughout the six families. For a polynomial
$f\in P_3$ and $a\geq0$, put
\begin{equation}\label{eq:infinite-lift}
 \Phi_a(f)=(xyz)^{p^a-1}f^{p^a}.
\end{equation}
This formula is used only in the degrees where the preceding
extraction is an isomorphism on indecomposables.

\begin{corl}\label{cor:infinite-invariant-bases}
Let $(p,n_0)$ be a row of Table~\ref{tab:infinite-digit-sums},
let $n_r$ be as in~\eqref{eq:transfer-shift-degrees}, and
let $(s,\mathcal F)$ be given by
\begin{center}
\begin{tabular}{@{}ccl@{}}
\toprule
$(p,n_0)$ & $s$ & $\mathcal F$\\
\midrule
$(3,13)$ & $1$ & $\{g_{3,21}\}$\\
$(3,29)$ & $1$ & $\{g_{3,45}\}$\\
$(3,117)$ & $0$ & $\{F_{3,57},G_{3,57}\}$\\
$(3,133)$ & $1$ & $\{g_{3,201}\}$\\
$(3,361)$ & $0$ & $\{f_{3,179}\}$\\
$(5,45)$ & $1$ & $\{g_{5,117}\}$\\
\bottomrule
\end{tabular}
\end{center}
If $r\geq s$, then
\[
 \bigl\{[\Phi_{r-s}(f)U]:f\in\mathcal F\bigr\}
\]
is a basis of $(\QH^{n_r}(3))^{G_3}$.
\end{corl}
\begin{proof}
In every row the stated polynomials give an invariant basis
at degree $n_s$, by Propositions~\ref{prop:finite-invariant-bases}
and~\ref{prop:first-shift-invariants}. For each $j>s$ the
map $\overline\kappa_3:(\QP_3)_{m_j}\to(\QP_3)_{m_{j-1}}$
is an equivariant isomorphism by the proof of
Theorem~\ref{thm:infinite-isomorphisms}. The monomial definition
gives $\kappa_3(\Phi_1(f))=f$. If $f$ is hit in degree
$m_{j-1}$, then the class of $\Phi_1(f)$ is in the kernel
of $\overline\kappa_3$, and is therefore zero. Thus
$\Phi_1$ descends in these degrees and induces the inverse
of $\overline\kappa_3$. This inverse is equivariant because
$\overline\kappa_3$ is equivariant and bijective. Tensoring
with $\langle U\rangle$ gives the inverse on the top exterior
quotients, and it consequently takes an invariant basis
to an invariant basis.

The maps $\Phi_a$ are $\Fp$-linear, since Frobenius is
$\Fp$-linear. Their defining formula gives
$\Phi_1(\Phi_a(f))=\Phi_{a+1}(f)$: the exponent of $xyz$
is $(p-1)+p(p^a-1)=p^{a+1}-1$. Induction therefore identifies
the composite of $r-s$ inverse extractions with
$\Phi_{r-s}$. Its polynomial degree is
$p^{r-s}(m_s+3)-3=m_r$, so its total degree after multiplication
by $U$ is $n_r$. Corollary~\ref{cor:central-top} completes the
identification with the full cohomology invariant space.
\end{proof}

\begin{nx}\label{rem:stripping-domain}
Meyer's stripping operation is defined on the reduced-power
sub-Hopf algebra $\mathcal B\subset\A$, rather than on the
annihilated homology of $BV_h$. If $\mathcal B^{\vee}$ is its
graded dual, $\xi\in\mathcal B^{\vee}$, and
$\Delta(\theta)=\sum\theta_{(1)}\otimes\theta_{(2)}$, then
\[
 \xi\cap\theta
 =\sum\langle\xi,\theta_{(2)}\rangle\theta_{(1)}
 \qquad(\theta\in\mathcal B).
\]
The identities in~\cite[\S4.1, equation~(4.1) and
Lemma~4.3, pp.~339--340]{Meyer2001} are identities between
Steenrod operations. In particular,
$\Delta(\Pp^1)=\Pp^1\otimes1+1\otimes\Pp^1$ and the
Milnor pairing give $\xi_1\cap\Pp^1=1$. Thus stripping
does not preserve the augmentation ideal $\mathcal B^+$.
It cannot be transported to a Steenrod module merely by
stripping the operation in an expression $\theta f$:
$\Pp^1(1)=0$, whereas $(\xi_1\cap\Pp^1)(1)=1$.
An operation on $D_{h,n}$ therefore requires a separate
construction preserving annihilated homology and the
$G_h$-coinvariant relations.
\end{nx}

\subsection{Kernels, coefficient functionals, and degree comparisons}
\label{subsec:comparison-obstructions}

The common-shift reduction leaves invariant dimensions in the
normalized degrees and their first shifts to be determined.
We examine three constructions that could be used in such a
comparison. The extraction kernel can have nonzero general linear
invariants, coefficient functionals associated with rank-one
non-hit monomials need not detect the fixed-vector equations, and
equivariant multiplication by a Dickson invariant need not
preserve hits. Compatibility of the transfer with the Frobenius
operation gives an exact comparison of successive transfer
kernels. The results below prove this uniform comparison,
calculate a particular extraction kernel and its first cohomology,
and establish the stated obstructions in degrees belonging to
the four principal families.

\paragraph{The invariant kernel of extraction.}
For $d\geq0$, write $m=pd+3(p-1)$ and set
\[
 M_d=(\QP_3)_m\otimes\langle U\rangle,\qquad
 N_d=(\QP_3)_d\otimes\langle U\rangle,\qquad
 K_d=\ker(\overline\kappa_3\otimes1).
\]
The equivariance in Lemma~\ref{lem:extraction-equivariant} makes
these spaces $G_3$-modules. The following statement identifies what
can be inferred from the diagonal weights without imposing the
remaining general linear invariance equations.

\begin{prop}\label{prop:extraction-kernel-weights}
Let $p$ be an odd prime and let $d\geq0$. Then $K_d$ is spanned
by the classes $[x^ay^bz^cU]$ with $a+b+c=pd+3(p-1)$ for which
at least one of $a,b,c$ is not congruent to $p-1$ modulo $p$.
Moreover, the diagonal subgroup $T_3$ satisfies
\[
 \dim K_d^{T_3}=\dim M_d^{T_3}-\dim N_d^{T_3}.
\]
If $p-1$ divides $d+3$, then every scalar matrix in $G_3$ acts
as the identity on $K_d$.
\end{prop}
\begin{proof}
The spanning assertion is the quotient form of Minami's
kernel description~\cite[Proposition~1.7(ii)]{Minami},
tensored with $\langle U\rangle$.
Put $m=pd+3(p-1)$.

The surjection $M_d\longrightarrow N_d$ is $T_3$-equivariant.
Since $|T_3|=(p-1)^3$ is invertible in $\Fp$, averaging a lift over
$T_3$ lifts every $T_3$-fixed vector to a $T_3$-fixed vector. Thus
\[
 0\longrightarrow K_d^{T_3}\longrightarrow M_d^{T_3}
 \longrightarrow N_d^{T_3}\longrightarrow0
\]
is exact, and taking dimensions proves the formula.
Finally, $\lambda I$ acts on $M_d$ by
$\lambda^{-(m+3)}=\lambda^{-p(d+3)}$.
If $p-1$ divides $d+3$, this scalar is $1$ for every
$\lambda\in\Fp^\times$. The same holds on its submodule $K_d$.
\end{proof}

The short exact sequence of $G_3$-modules
\[
 0\longrightarrow K_d\longrightarrow M_d
 \xrightarrow{\overline\kappa_3\otimes1}N_d\longrightarrow0
\]
gives the usual exact sequence in group cohomology
\cite[Chapter~III]{Brown1982}:
\begin{equation}\label{eq:kernel-invariant-exact-sequence}
 0\longrightarrow K_d^{G_3}\longrightarrow M_d^{G_3}
 \longrightarrow N_d^{G_3}
 \xrightarrow{\delta_d}H^1(G_3,K_d).
\end{equation}
Consequently,
\begin{equation}\label{eq:kernel-invariant-dimension}
 \dim M_d^{G_3}
 =\dim K_d^{G_3}+\dim N_d^{G_3}-\rank\delta_d.
\end{equation}
In particular, $K_d^{G_3}=0$ would imply
$\dim M_d^{G_3}\leq\dim N_d^{G_3}$, without any hypothesis
on $H^1(G_3,K_d)$. Under this hypothesis, equality holds
if and only if $\delta_d=0$.
The scalar matrices act trivially in all four principal families,
so the central-character argument does not establish the vanishing
of this cohomology group.

The connecting homomorphism can be determined during an induction
in which injectivity of the transfer is already known in the lower
degree.  The additional ingredient is the compatibility of the
transfer with the Frobenius operation.  We write
\[
 P^0:\Ext_{\A}^{s,t}(\Fp,\Fp)
       \longrightarrow\Ext_{\A}^{s,pt}(\Fp,\Fp)
\]
for the operation induced by Frobenius on the dual Steenrod algebra.
This is a multiplicative operation: the Frobenius is a morphism of
Hopf algebras, and its map on the cobar complex preserves the
concatenation product.  Minami's homological operation and its
compatibility with the transfer are given in
\cite[end of \S1 and Theorem~4.4]{Minami}.

\begin{lem}\label{lem:frobenius-injective-principal}
Let $p$ be an odd prime, and let $n$ belong to one of the four
families \eqref{eq:family1}--\eqref{eq:family4}.  Then
\[
 P^0:\Ext_{\A}^{3,n+3}(\Fp,\Fp)
       \longrightarrow\Ext_{\A}^{3,p(n+3)}(\Fp,\Fp)
\]
is injective.
\end{lem}
\begin{proof}
We first determine the effect of $P^0$ on the rank-one and rank-two
factors in Theorem~\ref{thm:target-dimensions}.  For $s\geq0$, put
\[
 B_s=x^{(p-1)p^s-1}u.
\]
Corollary~\ref{cor:full-rank-one} shows that $[B_s]$ is nonzero.
The action of $\lambda\in G_1=\Fp^\times$ on $B_s$ is
multiplication by $\lambda^{-(p-1)p^s}=1$, so this class spans
$\QH^{qp^s-1}(1)^{G_1}$.  The exponent identity
\[
 (p-1)p^{s+1}-1
   =p\bigl((p-1)p^s-1\bigr)+(p-1)
\]
gives
\[
 (\overline\kappa_1\otimes1)[B_{s+1}]=[B_s].
\]
The dual map between the corresponding rank-one coinvariant
spaces is therefore an isomorphism.

The classes $[\Lambda_s]$, $[M_s]$, and $[N_s]$ in
\eqref{eq:ranktwo-lambda}--\eqref{eq:ranktwo-nu} span nonzero
rank-two invariant lines by
Theorem~\ref{thm:ranktwo-explicit-invariants}.
Every exponent in the representatives with index $s+1$ is of
the form $pa+p-1$, where $a$ is the corresponding exponent in
the representative with index $s$.  The coefficient $2$ in
$M_s$ is unchanged.  Applying the defining monomial formula for
$\kappa_2$ consequently gives
\begin{equation}\label{eq:ranktwo-frobenius-lines}
 \begin{split}
 (\overline\kappa_2\otimes1)[\Lambda_{s+1}]&=[\Lambda_s],\\
 (\overline\kappa_2\otimes1)[M_{s+1}]&=[M_s],\\
 (\overline\kappa_2\otimes1)[N_{s+1}]&=[N_s].
 \end{split}
\end{equation}
The maps on these invariant lines, and hence their dual maps
on coinvariants, are isomorphisms.

To identify these dual maps with Minami's homological operation,
let $b_i$ be the homology element dual to the one-variable
cohomology monomial of total degree $i$.  On the tensor product
of $h$ odd-degree elements, that operation is
\[
 b_{2a_1+1}\otimes\cdots\otimes b_{2a_h+1}
 \longmapsto
 b_{2(pa_1+p-1)+1}\otimes\cdots\otimes
 b_{2(pa_h+p-1)+1};
\]
it is zero on a tensor product containing an even-degree factor
\cite[end of \S1]{Minami}.
Pairing this formula with polynomial monomials times
$u_1\cdots u_h$ shows that its dual on the top exterior
component is $\overline\kappa_h\otimes1$.
Lemma~\ref{lem:extraction-equivariant} supplies equivariance
for the polynomial factor, and the identity on the exterior
factor preserves its determinant character.  Thus the same
identification holds after taking coinvariants and invariants.

The transfers in ranks one and two are isomorphisms
\cite[Theorem~1]{Crossley1999Transfer}.  Their commutation with
the two operations in \cite[Theorem~4.4]{Minami}, together with
the preceding nonzero maps on coinvariants, implies that there
are scalars $a_s,b_s,c_s,e_s\in\Fp^\times$ such that
\begin{equation}\label{eq:frobenius-factor-shifts}
 \begin{aligned}
 P^0(h_s)&=a_s h_{s+1},&
 P^0(\widetilde\lambda_s)&=b_s\widetilde\lambda_{s+1},\\
 P^0(\mu_s)&=c_s\mu_{s+1},&
 P^0(\nu_s)&=e_s\nu_{s+1}.
 \end{aligned}
\end{equation}
Here the names of the degree-two classes agree, up to nonzero
scalars, with those in Aikawa's basis, as fixed above
\cite[Theorem~3]{Crossley1999Transfer}.

Simultaneously increasing every index by one preserves all
restrictions in \eqref{eq:family1}--\eqref{eq:family4} and
multiplies the internal degree $n+3$ by $p$.
Multiplicativity and \eqref{eq:frobenius-factor-shifts} send
the basis vector in each of families $1$, $3$, and $4$ to a
nonzero scalar multiple of the corresponding basis vector
with all indices increased by one.  In the second family the
first basis vector is sent to a nonzero scalar multiple of
$\widetilde\lambda_{i+1}h_{j+1}$.  If the second basis vector
$\widetilde\lambda_{j-1}h_{i+1}$ is present, its image is a
nonzero scalar multiple of
$\widetilde\lambda_jh_{i+2}$.  These are distinct basis entries
at the new degree by Theorem~\ref{thm:target-dimensions}.
If the lower degree has only one basis vector, its image is
still a nonzero basis vector, whether the new target has
dimension one or two.  The exceptional relations at $p=3$
do not change the independence of these allowed entries
\cite[Theorem~8.1 and Tables~8.1--8.2]{Aikawa}.
Thus the images of a basis are linearly independent in every
case, proving injectivity.
\end{proof}

The preceding lemma gives an exact comparison of the transfer
kernels in successive degrees, uniformly in the prime and the
independent indices. Its two nonnegative contributions distinguish
a kernel inherited from the lower degree from an excess in the
invariant extraction kernel. Lower-degree injectivity forces the
connecting homomorphism to vanish without any hypothesis on
$H^1(G_3,K_d)$.

\begin{prop}\label{prop:transfer-forces-connecting-zero}
Let $p$ be an odd prime, and let $n$ belong to one of the four
families \eqref{eq:family1}--\eqref{eq:family4}.  Put
\[
 d=\frac{n-3}{2},\qquad n^+=p(n+3)-3,
 \qquad m=pd+3(p-1).
\]
Let $M_d,N_d,K_d$ and $\delta_d$ be as in
\eqref{eq:kernel-invariant-exact-sequence}, and let
$J_d:D_{3,n}\longrightarrow D_{3,n^+}$ be induced by Minami's
homological operation.  Write
\[
 E_t=\Ext_{\A}^{3,t+3}(\Fp,\Fp),\qquad
 Z_t=\ker\Tr_{3,t},\qquad \epsilon_t=\dim Z_t,
 \qquad \Delta_n=\dim E_{n^+}-\dim E_n.
\]
Then there is an exact sequence
\begin{equation}\label{eq:transfer-kernel-comparison-exact}
 0\longrightarrow\ker J_d\longrightarrow Z_n
 \xrightarrow{J_d} Z_{n^+}\longrightarrow(K_d^{G_3})^*
 \longrightarrow E_{n^+}/P^0E_n\longrightarrow0,
\end{equation}
and
\begin{equation}\label{eq:connecting-kernel-bound}
 \rank\delta_d=\dim\ker J_d\leq\epsilon_n,
 \qquad \dim K_d^{G_3}\geq\Delta_n.
\end{equation}
Moreover,
\begin{equation}\label{eq:transfer-kernel-dimension-comparison}
 \epsilon_{n^+}
 =\bigl(\epsilon_n-\rank\delta_d\bigr)
   +\bigl(\dim K_d^{G_3}-\Delta_n\bigr),
\end{equation}
where both summands are nonnegative.
If $\Tr_{3,n}$ is injective, then $\delta_d=0$ and
\begin{equation}\label{eq:conditional-kernel-increment}
 \dim D_{3,n^+}=\dim D_{3,n}+\dim K_d^{G_3}.
\end{equation}
Under the same hypothesis, $\Tr_{3,n^+}$ is an isomorphism if
and only if
\begin{equation}\label{eq:conditional-kernel-criterion}
 \dim K_d^{G_3}=\Delta_n.
\end{equation}
\end{prop}
\begin{proof}
On the top exterior component, the pairing calculation in the
proof of Lemma~\ref{lem:frobenius-injective-principal} identifies
the dual of $J_d$ with $\overline\kappa_3\otimes1$.
Both $n$ and $n^+$ have the form occurring in
Corollary~\ref{cor:central-top}, which identifies the dual
coinvariant spaces with
\[
 D_{3,n}^*\cong N_d^{G_3},\qquad
 D_{3,n^+}^*\cong M_d^{G_3}.
\]
Thus $J_d^*$ is the map $M_d^{G_3}\longrightarrow N_d^{G_3}$
in \eqref{eq:kernel-invariant-exact-sequence}.
Exactness and finite-dimensional duality give
\[
 \dim\ker J_d=\dim\operatorname{coker}J_d^*
 =\rank\delta_d,
 \qquad
 \bigl(\operatorname{coker}J_d\bigr)^*\cong K_d^{G_3}.
\]

Minami's compatibility theorem
\cite[Theorem~4.4]{Minami} gives
\begin{equation}\label{eq:transfer-frobenius-compatibility}
 \Tr_{3,n^+}\circ J_d=P^0\circ\Tr_{3,n}.
\end{equation}
If $J_d(v)=0$, this equality implies
$P^0\Tr_{3,n}(v)=0$.
Lemma~\ref{lem:frobenius-injective-principal} then gives
$\Tr_{3,n}(v)=0$.  Consequently $\ker J_d\subseteq Z_n$,
which proves the first inequality in
\eqref{eq:connecting-kernel-bound}.
Compatibility also gives $J_d(Z_n)\subseteq Z_{n^+}$.

Set $C_d=D_{3,n^+}/J_dD_{3,n}$, and denote its quotient map
by $\pi:D_{3,n^+}\longrightarrow C_d$.  Define
\[
 \rho:C_d\longrightarrow E_{n^+}/P^0E_n,
 \qquad \rho(\pi(v))=\Tr_{3,n^+}(v)+P^0E_n.
\]
Changing $v$ by $J_d(w)$ changes its transferred image by
$P^0\Tr_{3,n}(w)$, so $\rho$ is well defined.
Theorem~\ref{thm:surjectivity} implies that $\rho$ is surjective.

Let $v\in Z_{n^+}$ satisfy $\pi(v)=0$.  Write $v=J_d(w)$.
Then \eqref{eq:transfer-frobenius-compatibility} gives
$P^0\Tr_{3,n}(w)=0$, and the injectivity of $P^0$ gives
$w\in Z_n$.  Conversely, every element of $J_d(Z_n)$ lies
in $Z_{n^+}$ and maps to zero under $\pi$.
It follows that
$\ker(\pi|_{Z_{n^+}})=J_d(Z_n)$.
If $\pi(v)\in\ker\rho$, then
$\Tr_{3,n^+}(v)=P^0(e)$ for some $e\in E_n$.
Surjectivity of $\Tr_{3,n}$ gives $w\in D_{3,n}$ with
$\Tr_{3,n}(w)=e$.
Then $v-J_d(w)\in Z_{n^+}$ and
$\pi(v-J_d(w))=\pi(v)$.
Thus $\ker\rho=\pi(Z_{n^+})$.
Together with $\ker J_d\subseteq Z_n$, these identities prove
the asserted exact sequence with $C_d$ as its fourth nonzero
term.  The already established dual identification gives
$C_d\cong(K_d^{G_3})^*$ and hence
\eqref{eq:transfer-kernel-comparison-exact}.

Since $P^0$ is injective,
$\dim(E_{n^+}/P^0E_n)=\Delta_n$.
Surjectivity of $\rho$ therefore gives
$\dim K_d^{G_3}\geq\Delta_n$.
The exact sequence now gives
\[
 \epsilon_{n^+}
 =\dim J_d(Z_n)+\dim\ker\rho
 =\epsilon_n-\rank\delta_d+\dim K_d^{G_3}-\Delta_n.
\]
The two summands in
\eqref{eq:transfer-kernel-dimension-comparison} are respectively
$\dim J_d(Z_n)$ and $\dim\ker\rho$, so both are nonnegative.

If $\Tr_{3,n}$ is injective, then $Z_n=0$, and
\eqref{eq:connecting-kernel-bound} gives $\delta_d=0$.
Equation~\eqref{eq:kernel-invariant-dimension} gives
\eqref{eq:conditional-kernel-increment}.
Equation~\eqref{eq:transfer-kernel-dimension-comparison} reduces
to $\epsilon_{n^+}=\dim K_d^{G_3}-\Delta_n$.
By Theorem~\ref{thm:surjectivity}, the upper transfer is
surjective, and it is an isomorphism precisely when this last
quantity is zero.  This proves
\eqref{eq:conditional-kernel-criterion}.
\end{proof}

Equation~\eqref{eq:transfer-kernel-dimension-comparison} also
describes the limitation of a comparison of dimension increments.
If
\[
 \dim K_d^{G_3}-\rank\delta_d=\Delta_n,
\]
then $\epsilon_{n^+}=\epsilon_n$; vanishing of this common
dimension still requires an initial injectivity theorem.
If $H^1(G_3,K_d)=0$, then $\delta_d=0$ and
$\epsilon_{n^+}\geq\epsilon_n$.
More generally, the upper transfer is an isomorphism if and
only if both
$\rank\delta_d=\epsilon_n$ and
$\dim K_d^{G_3}=\Delta_n$.

In families $1$, $3$, and $4$, $\Delta_n=0$.
In family $2$, $\Delta_n=0$ except when $j=0$ and $i\geq1$,
in which case $\Delta_n=1$ by
Theorem~\ref{thm:target-dimensions}.
At every such boundary degree, dualizing $\rho$ supplies an
injection
\[
 \bigl(E_{n^+}/P^0E_n\bigr)^*\lhook\joinrel\longrightarrow
 K_d^{G_3}.
\]
Thus $K_d^{G_3}$ contains a one-dimensional subspace for every
odd prime and every pair $(i,0)$ with $i\geq1$, independently
of any injectivity assumption on the lower transfer.
Once lower-degree injectivity has been proved, the remaining
requirement at the next step is the exact value of
$\dim K_d^{G_3}$; no calculation of $H^1(G_3,K_d)$ is needed
for that induction.  The proposition does not determine the
kernel of the transfer at the normalized initial degrees,
where the relative index differences remain arbitrary. The rank
of the connecting homomorphism depends on the entire extension,
not only on the isomorphism class of $K_d$. The augmentation
example in Subsection~\ref{subsec:composition-factors} exhibits
this distinction for $G_3$-modules with trivial scalar action.

The next calculations determine nonzero diagonal and general
linear fixed subspaces of $K_d$. They also distinguish an
increase in invariant dimension from failure of injectivity
of the transfer.

The two congruence conditions alone do not force the diagonal
fixed subspace of the kernel to vanish. Indeed, if
$m\geq3(p-2)$ and $m\equiv-3\pmod{p-1}$, the exponent triple
\[
 (p-2,p-2,m-2p+4)
\]
has nonnegative entries, satisfies the diagonal weight condition
of Proposition~\ref{prop:torus-block}, and is excluded by
$\kappa_3$. Whether its class is nonzero requires the hit relations.
The following calculations show that the intersection of these
conditions also contains nonzero classes after imposing the hit relations.

\begin{prop}\label{prop:nonzero-extraction-kernels}
Let $K_d$ be defined as above. If $(p,d)=(3,5)$ or $(5,21)$,
then $\dim K_d^{T_3}=6$. If $p=3$ and $d=17$, then
\[
 0\ne[G_{3,57}U]\in K_{17}^{G_3}.
\]
\end{prop}
\begin{proof}
For $(p,d)=(3,5)$, the polynomial degrees of $N_d$ and $M_d$ are $5$
and $21$. Tables~\ref{tab:rankthree-ranks}
and~\ref{tab:four-shifted-ranks} give dimensions $2$ and $8$
for their diagonal fixed subspaces. Proposition~\ref{prop:extraction-kernel-weights}
therefore gives $\dim K_5^{T_3}=8-2=6$. In fact, a basis is
\[
 \bigl\{[x^3yz^{17}U],\ [x^3y^{13}z^5U],\
 [x^3y^{17}zU],\ [x^{15}yz^5U],\
 [x^{15}y^5zU],\ [x^{17}y^3zU]\bigr\}.
\]
These six classes are among the nonpivot classes in
$\mathcal B_{3,21}$, so they are independent, and every one is
annihilated by $\kappa_3$.
For $(p,d)=(5,21)$, the same tables give dimensions $4$ and
$10$ in polynomial degrees $21$ and $117$, respectively.
Thus $\dim K_{21}^{T_3}=10-4=6$.

Now let $p=3$ and $d=17$. Proposition~\ref{prop:finite-invariant-bases}
proves that $[G_{3,57}U]$ is a nonzero $G_3$-fixed class.
Applying $\kappa_3$ to the representative in
\eqref{eq:rankthree-g57} leaves only the monomials whose
exponents are all congruent to $2$ modulo $3$ and gives
\[
 \kappa_3G_{3,57}=x^{13}y^3z-x^{15}yz.
\]
The Cartan formula gives the explicit identity
\[
 x^{13}y^3z-x^{15}yz
 =\Pp^1\bigl(x^{11}y^3z-x^{11}yz^3-x^{13}yz\bigr).
\]
Indeed, the three terms on the right expand to
$2x^{13}y^3z+x^{11}y^3z^3$,
$-2x^{13}yz^3-x^{11}y^3z^3$, and
$-x^{15}yz-x^{13}y^3z-x^{13}yz^3$.
Their sum is the left side because $3=0$ in $\mathbb F_3$.
The extracted polynomial is therefore hit, proving
$[G_{3,57}U]\in K_{17}^{G_3}$.
\end{proof}

For $p=3$, the preceding step connects the degree $n=37$
in family~\eqref{eq:family2}, with $(i,j)=(1,0)$, to degree
$n=117$, with $(i,j)=(2,1)$. The nonzero invariant kernel
therefore occurs within the four principal families. 
The nonzero invariant kernel can be determined exactly. The additional
calculation concerns polynomial degree $17$ at $p=3$ and uses the
same quotient reduction and general linear generators as the
preceding finite calculations.

\begin{prop}\label{prop:boundary-kernel-dimension}
Let $p=3$. Then
\[
 (\QH^{37}(3))^{G_3}
   =\mathbb F_3\,[H_{3,17}U],\qquad
 K_{17}^{G_3}=\mathbb F_3\,[G_{3,57}U],
\]
where
\[
 H_{3,17}=-xy^{11}z^5+x^5y^7z^5-x^7y^5z^5
          +x^{11}yz^5-x^{11}y^5z+x^{15}yz.
\]
Moreover, $\Tr_{3,37}$ is an isomorphism and the restriction
$M_{17}^{G_3}\longrightarrow N_{17}^{G_3}$ of
$\overline\kappa_3\otimes1$ is surjective with one-dimensional kernel.
\end{prop}
\begin{proof}
Corollary~\ref{cor:central-top} identifies
$(\QH^{37}(3))^{G_3}$ with the invariants of
$(\QP_3)_{17}\otimes\langle U\rangle$.
The matrix of Proposition~\ref{prop:hit-matrix} has $171$ rows
and rank $141$ in this degree. In the diagonal fixed block,
the $36$ odd exponent triples have hit rank $30$.
The six nonpivot triples, in the order used below, are
\[
 \mathcal B_{3,17}=
 \bigl((1,11,5),(5,7,5),(7,5,5),(11,1,5),
       (11,5,1),(15,1,1)\bigr).
\]
Thus the corresponding classes $[v_jU]$, $1\leq j\leq6$,
form a basis of the diagonal fixed subspace.
Write $f=\sum_{j=1}^6c_jv_j$. With the defect notation of
\eqref{eq:rankthree-three-defects}, the normalized distinct
nonzero coordinate forms are
\begin{align*}
 \mathcal E_{12}&=\{c_1+c_4,\ c_2+c_3,\ c_5+c_6\},\\
 \mathcal E_{23}&=\{c_2+c_3,\ c_4+c_5,\ c_1+c_6\},\\
 \mathcal E_\tau&=\{c_2+c_3,\ c_1-c_3,\ c_1-c_2+c_3,
                    \ c_1+c_2,\ c_1+c_2+c_3+c_4+c_5+c_6\}.
\end{align*}
Put $a=c_1$. The equations $c_1+c_4=0$ and $c_4+c_5=0$
give $c_4=-a$ and $c_5=a$. The equation $c_5+c_6=0$
gives $c_6=-a$. The equations $c_1-c_3=0$ and
$c_2+c_3=0$ give $c_3=a$ and $c_2=-a$. Therefore every
invariant vector has the form
\[
 (c_1,c_2,c_3,c_4,c_5,c_6)=a(1,-1,1,-1,1,-1).
\]
Every displayed equation vanishes on this vector: the only
sum not cancelling in pairs is $c_1-c_2+c_3=3a=0$.
Hence $a$ is unrestricted and the fixed subspace has dimension one.
The choice $a=-1$ gives $H_{3,17}$. Its last coordinate is $1$,
so its class is nonzero.

Formula~\eqref{eq:extraction-map-general} and
\eqref{eq:rankthree-f57} give the polynomial identity
$\kappa_3F_{3,57}=H_{3,17}$. The preceding proof of
Proposition~\ref{prop:nonzero-extraction-kernels} gives
$[\kappa_3G_{3,57}]=0$. Proposition~\ref{prop:finite-invariant-bases}
identifies $[F_{3,57}U]$ and $[G_{3,57}U]$ as a basis of
$M_{17}^{G_3}$. Thus the restriction of
$\overline\kappa_3\otimes1$ maps this basis to
$[H_{3,17}U]$ and $0$, respectively. Its image is all of
$N_{17}^{G_3}$, and its kernel is exactly
$\mathbb F_3[G_{3,57}U]$.
Finally, $n=37$ is the second family with $(i,j)=(1,0)$.
Theorem~\ref{thm:surjectivity} makes $\Tr_{3,37}$ surjective,
and Theorem~\ref{thm:target-dimensions} gives a one-dimensional
target. The fixed-space calculation and
\eqref{eq:coinvariant-duality} give a one-dimensional domain.
Consequently the transfer is an isomorphism.
\end{proof}

In~\eqref{eq:kernel-invariant-exact-sequence}, the connecting
map $\delta_{17}$ is therefore zero at $p=3$, although
$K_{17}^{G_3}$ is one-dimensional. In particular, the
vanishing of the connecting map alone does not imply equality
of the two invariant dimensions.

The preceding invariant calculation determines the connecting map at the
first step of the family $40\cdot3^r-3$.  In this case the coefficient
module admits a more precise cohomological calculation.  Restriction to
a Borel subgroup reduces the calculation to two transvections and the
diagonal weights.

\begin{prop}\label{prop:boundary-kernel-first-cohomology}
Let $p=3$, and let $K_{17}$ be the kernel of
\[
 (\QP_3)_{57}\otimes\langle U\rangle
 \xrightarrow{\overline\kappa_3\otimes1}
 (\QP_3)_{17}\otimes\langle U\rangle.
\]
If $B$ is the subgroup generated by $T_3$ and the substitutions
$X(y)=y-x$ and $Y(z)=z-y$, with all other variables fixed, then
\[
 \dim K_{17}=13,\qquad
 \dim K_{17}^{T_3}=7,\qquad
 \dim K_{17}^{B}=3,
 \qquad H^1(B,K_{17})=H^1(G_3,K_{17})=0.
\]
\end{prop}
\begin{proof}
Write $K=K_{17}$.  Table~\ref{tab:rankthree-ranks} and
Proposition~\ref{prop:boundary-kernel-dimension} give dimensions $43$
and $30$ for the domain and codomain of extraction.  Extraction is
surjective, so $\dim K=43-30=13$.

For the following calculation, put $e_i=[f_iU]$, where
\begin{align*}
 f_1&=x^3yz^{53},& f_2&=x^3y^{49}z^5,& f_3&=x^3y^{52}z^2,\\
 f_4&=x^3y^{53}z,&
 f_5&=x^{47}y^5z^5-x^{41}y^{11}z^5,&
 f_6&=x^{48}y^4z^5,\\
 f_7&=x^{48}y^7z^2,& f_8&=x^{51}yz^5,& f_9&=x^{51}y^4z^2,\\
 f_{10}&=x^{51}y^5z,& f_{11}&=x^{52}y^3z^2,&
 f_{12}&=x^{52}y^4z,\\
 f_{13}&=x^{53}y^3z.&&&&
\end{align*}
Every $f_i$ other than $f_5$ is annihilated by $\kappa_3$ at the
polynomial level.  For $f_5$, the Cartan identity in the proof of
Proposition~\ref{prop:nonzero-extraction-kernels} gives
\[
 \kappa_3 f_5=x^{15}yz-x^{13}y^3z
 =-\Pp^1\bigl(x^{11}y^3z-x^{11}yz^3-x^{13}yz\bigr).
\]
Thus $e_i\in K$ for all $i$.  In the quotient reduction used for
Table~\ref{tab:rankthree-ranks}, the twelve monomials with $i\ne5$,
together with the two monomials in $f_5$, occupy distinct nonpivot
positions.  Comparison of their coordinates shows that the thirteen
classes are independent.  They therefore form a basis of $K$.

Let $t_j$ change the sign of the $j$th polynomial and exterior
generators, and let $D_j$ denote its action on $K$.  Each $e_i$ is a
simultaneous eigenvector for $D_1,D_2,D_3$.  In the second column below,
the triple records the corresponding eigenvalues.  Polynomial
substitution followed by the Cartan quotient reduction of
Proposition~\ref{prop:hit-matrix} gives the remaining two columns:
\[
\begin{array}{c|c|c|c}
 i&(D_1,D_2,D_3)\text{ on }e_i&(X-1)e_i&(Y-1)e_i\\ \hline
 1 &(1,1,1)&0&e_3-e_4\\
 2 &(1,1,1)&e_6-e_8&e_3-e_4\\
 3 &(1,-1,-1)&e_{11}&e_4\\
 4 &(1,1,1)&-e_{12}-e_{13}&0\\
 5 &(1,1,1)&-e_6+e_8&e_9-e_{10}\\
 6 &(-1,-1,1)&e_8&e_7+e_{12}\\
 7 &(-1,1,-1)&-e_9-e_{11}&-e_{12}\\
 8 &(1,1,1)&0&e_9-e_{10}\\
 9 &(1,-1,-1)&-e_{11}&e_{10}\\
 10&(1,1,1)&e_{12}+e_{13}&0\\
 11&(-1,1,-1)&0&e_{12}\\
 12&(-1,-1,1)&-e_{13}&0\\
 13&(1,1,1)&0&0
\end{array}
\]
The diagonal weights identify
\[
 K^{T_3}=\langle e_1,e_2,e_4,e_5,e_8,e_{10},e_{13}\rangle.
\]
For a vector in this subspace, the two last columns give the four
independent fixed-vector equations
\[
 a_2-a_5=0,\quad -a_4+a_{10}=0,\quad
 a_1+a_2=0,\quad a_5+a_8=0.
\]
It follows that
\[
 K^B=\langle -e_1+e_2+e_5-e_8,\ e_4+e_{10},\ e_{13}\rangle,
\]
which proves the asserted dimensions of the two fixed subspaces.

Let $c:B\to K$ be a $1$-cocycle, with convention
$c(gh)=c(g)+gc(h)$.  Since $|T_3|=8$ is invertible in
$\mathbb F_3$, subtracting a coboundary makes $c$ vanish on $T_3$.
Explicitly, if $S=\sum_{t\in T_3}c(t)$, then
$c(t)=t(-S/8)-(-S/8)$ for every $t\in T_3$.
We henceforth assume this normalization.

Put $b=c(X)$ and $d=c(Y)$.  The cocycle identity now gives
$c(tst^{-1})=tc(s)$ for $t\in T_3$ and $s\in B$.
The elements $t_1,t_2$ conjugate $X$ to $X^{-1}$, whereas
$t_3$ commutes with $X$.  Since $X$ has order three,
$c(X^{-1})=-X^2b$.  Consequently
\[
 (D_1+X^2)b=(D_2+X^2)b=(D_3-1)b=0.
\]
Likewise,
\[
 (D_2+Y^2)d=(D_3+Y^2)d=(D_1-1)d=0.
\]
The displayed action table makes the first three equations equivalent
to
\[
 b_j=0\quad(j\notin\{6,8,12,13\}),\qquad
 b_6+b_8=0,\qquad b_{12}-b_{13}=0,
\]
where $b=\sum b_je_j$.  The second three equations are equivalent to
\[
 d_j=0\quad(j\notin\{3,4,9,10\}),\qquad
 d_3+d_4=0,\qquad d_9+d_{10}=0.
\]
Thus there exist $a,b',c',d'\in\mathbb F_3$ such that
\[
 c(X)=a(e_6-e_8)+b'(e_{12}+e_{13}),\qquad
 c(Y)=c'(e_3-e_4)+d'(e_9-e_{10}).
\]
Set
\[
 k=a(e_2-e_1)+b'e_{10}+c'e_1+d'e_8\in K^{T_3}.
\]
The table gives $(X-1)k=c(X)$ and $(Y-1)k=c(Y)$.
The cocycle $g\mapsto c(g)-(g-1)k$ therefore vanishes on
$T_3$, $X$, and $Y$.  These elements generate $B$, so repeated
application of the cocycle identity shows that it vanishes on every
element of $B$.  Every cocycle is consequently a coboundary, proving
$H^1(B,K)=0$.

Finally, $B$ has order $3^3(3-1)^3=216$, while
$|G_3|=(27-1)(27-3)(27-9)=11232$.  Thus $[G_3:B]=52$ is
invertible in $\mathbb F_3$.  The restriction--corestriction identity
in group cohomology~\cite[Chapter~III]{Brown1982} makes
\[
 H^1(G_3,K)\longrightarrow H^1(B,K)
\]
injective.  Its target is zero, which proves $H^1(G_3,K)=0$.
\end{proof}

The vanishing of the entire cohomology group is stronger than
the vanishing of $\delta_{17}$ already implied by
Proposition~\ref{prop:transfer-forces-connecting-zero}.
It is a statement about this particular coefficient module.
In the uniform inductive criterion, the remaining quantity is
$\dim K_d^{G_3}$, not the whole first cohomology group.

The calculation also extends the third family in
Theorem~\ref{thm:infinite-isomorphisms} by its preceding degree.
The dimensions change at the first step, while the transfer
remains an isomorphism on both sides.

\begin{corl}\label{cor:extended-dimension-jump-family}
Let $p=3$ and $r\geq0$, and put $n_r=40\cdot3^r-3$.
Then $\Tr_{3,n_r}$ is an isomorphism and
\[
 \dim D_{3,n_r}=\dim\Ext^{3,n_r+3}_{\A}
 =\begin{cases}
 1,&r=0,\\
 2,&r\geq1.
 \end{cases}
\]
\end{corl}
\begin{proof}
For $r=0$, Proposition~\ref{prop:boundary-kernel-dimension}
gives the isomorphism and the one-dimensional spaces in degree
$37$. For $r\geq1$, the identity
$40\cdot3^r-3=120\cdot3^{r-1}-3$ identifies the degree with
the two-dimensional family in
Theorem~\ref{thm:infinite-isomorphisms}.
\end{proof}

\paragraph{Spike coefficient functionals.}
The coefficient functionals of Proposition~\ref{prop:spike-independence}
provide independent classes, but they need not detect all the equations
for general linear invariance. The digit sums in the fourth family give
a uniform obstruction to using only those functionals.

\begin{prop}\label{prop:fourth-family-no-spikes}
Let $p$ be an odd prime, let $i,j\geq0$ satisfy
$j\notin\{i-1,i,i+1,i+2\}$, and put
\[
 m=(p-1)(2p^{i+1}+p^i+p^j)-3.
\]
Then $\mathcal S_m=\varnothing$ for the rank-three set
$\mathcal S_m$ of Proposition~\ref{prop:spike-independence}.
\end{prop}
\begin{proof}
The four positions $i$, $i+1$, $i+2$ and $j$ are distinct, and
\[
 m+3=(p-1)p^i+(p-2)p^{i+1}+p^{i+2}+(p-1)p^j.
\]
Every displayed coefficient is a nonzero base-$p$ digit. Hence
$\ds(m+3)=3(p-1)$.
Suppose that $(a_1,a_2,a_3)\in\mathcal S_m$. By the definition of
$\mathcal S_m$, there exist integers $1\leq c_\ell\leq p-1$ and
$r_\ell\geq0$ such that $a_\ell+1=c_\ell p^{r_\ell}$ for
$1\leq\ell\leq3$. Therefore
\[
 3(p-1)=\ds\!\left(\sum_{\ell=1}^3c_\ell p^{r_\ell}\right)
 \leq\sum_{\ell=1}^3c_\ell\leq3(p-1).
\]
Both inequalities are equalities. The second gives
$c_1=c_2=c_3=p-1$. Each carry in base $p$ decreases the sum of
the digits being added by $p-1$, so equality in the first inequality
implies that no carry occurs. Since $2(p-1)\geq p$, the three
positions $r_1,r_2,r_3$ must then be distinct. It follows that
$m+3$ has exactly three nonzero digits, all equal to $p-1$.
This contradicts its displayed expansion, which has four nonzero
digits. Thus $\mathcal S_m$ is empty.
\end{proof}

For the degrees of Proposition~\ref{prop:fourth-family-no-spikes},
there is no nonzero coefficient functional of the specified type.
In particular, projecting the fixed-vector equations onto these
functionals gives the zero map. Taking general linear translates of
this empty collection does not change that conclusion.

\begin{prop}\label{prop:spike-projection-small-degree}
Let $p=3$, let $W=(\QP_3)_5\otimes\langle U\rangle$, and let
$T_3$ be the diagonal subgroup of $G_3$. For $A\in\mathcal S_5$, let
$\lambda_A$ be coefficient extraction on the polynomial factor.
For $s\in\{\sigma_{12},\sigma_{23},\tau\}$, define
\[
 E_{s,A}:W^{T_3}\longrightarrow\mathbb F_3,
 \qquad E_{s,A}(v)=\lambda_A(sv-v),
\]
where the generators act as in~\eqref{eq:rankthree-three-defects}.
Then every $E_{s,A}$ is zero, whereas the joint fixed-vector map
$W^{T_3}\to W^{\oplus3}$ has rank one.
\end{prop}
\begin{proof}
Put $A_0=x^3yz$, $B_0=xy^3z$ and $C_0=xyz^3$.
Proposition~\ref{prop:degree-thirteen} gives
\[
 W^{T_3}=\langle[A_0U],[B_0U],[C_0U]\rangle,
 \qquad [A_0U]+[B_0U]+[C_0U]=0,
\]
with this as the only relation. Each exponent triple of these three
monomials contains the exponent $3$, for which $3+1=4$ has two
nonzero base-$3$ digits. Thus every $\lambda_A$ with
$A\in\mathcal S_5$ vanishes on their span.
The two adjacent interchanges preserve this span, including the
factor $-1$ from their action on $U$. Their defects are therefore
annihilated by every $\lambda_A$.
The three polynomial transvection defects are
\[
 \tau(A_0)-A_0=-x^4z,\qquad
 \tau(B_0)-B_0=-x^4z,\qquad
 \tau(C_0)-C_0=-x^2z^3.
\]
The first two contain the exponent $4$, and the last contains the
exponent $3$. Neither $4+1=5$ nor $3+1=4$ has one nonzero
base-$3$ digit. All their coefficients at triples in
$\mathcal S_5$ consequently vanish. This proves $E_{s,A}=0$.
By Proposition~\ref{prop:degree-thirteen}, $\dim W^{T_3}=2$ and
$\dim W^{G_3}=1$. The kernel of the joint fixed-vector map is
$W^{G_3}$ by Lemma~\ref{lem:gl-generators}. The rank-nullity theorem
therefore gives rank $2-1=1$.
\end{proof}

\paragraph{Multiplication by Dickson invariants.}
The Dickson invariant algebra gives equivariant multiplication maps on
the polynomial algebra, but a map on hit quotients requires preservation
of the hit subspace in the relevant degree.  These requirements are
independent.  To make the distinction explicit, normalize the Dickson
generators by
\begin{equation}\label{eq:dickson-characteristic-polynomial}
 \prod_{(a,b,c)\in\Fp^3}(T+ax+by+cz)
 =T^{p^3}-D_2T^{p^2}+D_1T^p-D_0T.
\end{equation}
The classical Dickson theorem identifies
$P_3^{G_3}=\Fp[D_0,D_1,D_2]$ and gives
$\deg D_s=p^3-p^s$ for $0\le s\le2$
\cite{Dickson1911}.  The action being contragredient does not change
this invariant algebra, since the inverse-transpose automorphism
permutes $G_3$.

\begin{prop}\label{prop:dickson-degree-obstruction}
Let $p$ be an odd prime, and let $j\ge0$.  Set
$a_j=(p-1)^2p^j$.  If $j=0$ or $j=1$, then
$(P_3^{G_3})_{a_j}=0$.  If $j=2$, then
\[
 (P_3^{G_3})_{a_2}=\Fp D_2^{p-1}.
\]
If a positive power of one Dickson generator has degree $a_j$,
then $j\ge2$ and that power is $D_2^{(p-1)p^{j-2}}$.
\end{prop}
\begin{proof}
The smallest positive degree of a Dickson generator is
$p^3-p^2=p^2(p-1)$.  For $j\le1$ one has
$a_j\le p(p-1)^2<p^2(p-1)$, proving the first assertion.

A Dickson monomial $D_0^{e_0}D_1^{e_1}D_2^{e_2}$ has degree
$a_2$ precisely when
\[
 (p^2+p+1)e_0+p(p+1)e_1+p^2e_2=p^2(p-1).
\]
Reduction modulo $p$ gives $e_0\equiv0\pmod p$.
On the other hand, the displayed equality implies
$e_0\le p^2(p-1)/(p^2+p+1)<p$, so $e_0=0$.
After division by $p$, reduction modulo $p$ gives
$e_1\equiv0\pmod p$, whereas
$e_1\le p(p-1)/(p+1)<p$.  Thus $e_1=0$ and $e_2=p-1$.
Since the Dickson generators are algebraically independent,
their monomials form a basis of the invariant algebra, proving
the asserted one-dimensionality.

For a positive power of $D_0$, the required equality would make
$p^2+p+1$ divide $(p-1)p^j$.  This is impossible because
$p^2+p+1$ is relatively prime to $p$ and is larger than $p-1$.
For $D_1$, it would make $p+1$ divide $(p-1)p^{j-1}$;
when $j\ge1$ this is impossible because $p+1$ is relatively
prime to $p$ and exceeds $p-1$.  The case $j=0$ was already
excluded.  Finally, $e(p^3-p^2)=a_j$ is equivalent to
$e=(p-1)p^{j-2}$, which is a positive integer exactly when
$j\ge2$.
\end{proof}

Increasing one index $j$ to $j+1$ in any of the displayed
degree families changes polynomial degree by $a_j$.
The degree-matching power from Proposition~\ref{prop:dickson-degree-obstruction}
does not, however, induce a map on the corresponding hit quotients.
The failure occurs between admissible degrees of
family~\eqref{eq:family1} for every odd prime.

\begin{prop}\label{prop:dickson-hit-obstruction}
Let $p$ be an odd prime and $j\ge2$.  Put
\[
 \begin{split}
 F_j&=D_2^{(p-1)p^{j-2}},\\
 m_j&=(p-1)(1+p^j+p^{j+3})-3,\\
 m'_j&=(p-1)(1+p^{j+1}+p^{j+3})-3.
 \end{split}
\]
Then multiplication by $F_j$ does not carry
$(\Hit_3)_{m_j}$ into $(\Hit_3)_{m'_j}$.
The same failure holds on the polynomial weight space
whose three exponents are congruent to $p-2$ modulo $p-1$.
Consequently multiplication by $F_j$ does not induce a map
$(\QP_3)_{m_j}\to(\QP_3)_{m'_j}$ or a map on their
determinant-twisted $T_3$-fixed quotients.
\end{prop}
\begin{proof}
Specializing~\eqref{eq:dickson-characteristic-polynomial} at $y=z=0$
gives
\[
 \prod_{a\in\Fp}(T+ax)^{p^2}
 =(T^p-x^{p-1}T)^{p^2}
 =T^{p^3}-x^{p^2(p-1)}T^{p^2}.
\]
Therefore $D_2(x,0,0)=x^{p^2(p-1)}$, and the coefficient
of $x^{a_j}$ in $F_j$ is $1$.

Set
\[
 \begin{split}
 A_1&=(p-1)p^{j+3}-1,\\
 A_2&=p-2,\\
 A_3&=(p-1)p^{j+1}-1,\\
 b&=(p-1)(p^{j+3}-p^{j+1})-1.
 \end{split}
\]
The integer $b$ is nonnegative.  Define
\[
 h_j=\Pp^{p^j}(x^b y^{A_2}z^{A_3}).
\]
Its degree is
\[
 \begin{split}
 b+A_2+A_3+(p-1)p^j
 &=(p-1)(p^{j+3}+1+p^j)-3\\
 &=m_j.
 \end{split}
\]
Thus $h_j\in(\Hit_3)_{m_j}$.  Every monomial in $h_j$ has
each exponent congruent to $p-2$ modulo $p-1$, because the
three exponents of its input have those residues and every
term in the Cartan formula changes an exponent by a multiple
of $p-1$.

Consider the coefficient of $x^{A_1}y^{A_2}z^{A_3}$ in $F_jh_j$.
The input of the reduced power already has $y$- and $z$-exponents
$A_2$ and $A_3$.  Steenrod reduced powers never decrease either
exponent, and multiplication by a polynomial also never decreases
them.  Consequently, a contribution to the specified coefficient
must use the monomial $x^{a_j}$ from $F_j$ and must apply the
entire operation $\Pp^{p^j}$ to $x^b$ in the Cartan expansion.
The equality
\[
 b+(p-1)p^j+a_j=A_1
\]
shows that this contribution has coefficient
$\binom{b}{p^j}$.  Since $b\equiv-1\pmod{p^{j+1}}$, its
$j$th base-$p$ digit is $p-1$.  Lucas's congruence therefore gives
\[
 \binom{b}{p^j}=p-1=-1\qquad\text{in }\Fp.
\]

Each $A_\ell+1$ has exactly one nonzero base-$p$ digit.
The coefficient functional constructed in the proof of
Proposition~\ref{prop:spike-independence} therefore vanishes on
$(\Hit_3)_{m'_j}$ but takes the value $-1$ on $F_jh_j$.
It follows that $F_jh_j$ is non-hit.  Both the input and output
have the required polynomial weight.  Tensoring with
$\det^{-1}$ makes that weight $T_3$-fixed by
Proposition~\ref{prop:torus-block}, proving the statement about
the twisted quotients.
\end{proof}

The two triples of indices in Proposition~\ref{prop:dickson-hit-obstruction}
are $(0,j,j+3)$ and $(0,j+1,j+3)$.  They satisfy the
restrictions in~\eqref{eq:family1}: in the first triple
$0\le j-2\le j-1$, and in the second
$0\le j-1\le j-1$.  Thus the failure does not depend on
leaving the four families under consideration.  At $j=2$,
Proposition~\ref{prop:dickson-degree-obstruction} shows that every
nonzero invariant of the required degree is a scalar multiple
of $F_2$.  No multiplication map by a homogeneous Dickson
invariant of that degree can therefore provide the proposed
map on these hit quotients.

Even when a different construction gives an injection from
one invariant space to another, its dimension inequality must
have the required direction.  An injection associated with
$j\mapsto j+1$ proves only that the dimension at the smaller
index is at most the dimension at the larger index.  It does
not bound the latter dimension from above.  Moreover, the
common index shift induced by extraction preserves all index
differences.  Hence a reduction to bounded differences requires
an additional operation with an independently proved effect
on those differences and on the direction of the dimension
inequality.

\subsection{Rank-two products and the transfer kernel}
\label{subsec:ranktwo-products}
Restrictions to two variables and coefficient extraction have different
effects on the determinant twist. We first describe that distinction and
then construct coefficient functionals from the full rank-two quotient.
The invariant representatives in Appendix~\ref{app:ranktwo} are classes in
an invariant subspace. They do not constitute the dual of the full
rank-two quotient, which is the space required for this construction.
Nor does the coefficient of a non-hit rank-two monomial necessarily
annihilate hits. For example, at $p=3$ the monomial $x^7y^5$
is not hit, while
$\Pp^1(x^5y^5)=2x^7y^5+2x^5y^7$ has a nonzero coefficient
at that monomial. A functional on the quotient must annihilate
all hit relations, including such relations between nonzero classes.
Crossley's bases of that quotient are given in
\cite[Theorem~2 and Table~1]{Crossley1999Transfer} and
\cite{Crossley1999Bases}.

For $h\geq1$ and $d\geq0$, put
\begin{equation}\label{eq:ranktwo-dual-notation}
 A_h(d)=\operatorname{Hom}_{\Fp}((\QP_h)_d,\Fp),\qquad
 W_m=(\QP_3)_m\otimes\langle U\rangle\quad(m\geq0).
\end{equation}
We regard $A_h(d)$ as the annihilator of $(\Hit_h)_d$ in the
dual of $(P_h)_d$. If $a\geq0$, let
\[
 C_{a,z}\Bigl(\sum_{b\geq0}f_b(x,y)z^b\Bigr)=f_a(x,y).
\]
Define $C_{a,x}$ and $C_{a,y}$ by the corresponding permutations of
the variables. The order of the two remaining variables is their
order in $(x,y,z)$.

\begin{prop}\label{prop:ranktwo-coefficient-extraction}
Let $p$ be an odd prime and let $a\geq0$ satisfy
$a+1=cp^r$, where $1\leq c\leq p-1$ and $r\geq0$.
Then
\[
 C_{a,z}\Pp^s=\Pp^s C_{a,z}\qquad(s\geq0),
\]
and $C_{a,z}$ induces a linear map
\[
 (\QP_3)_m\longrightarrow(\QP_2)_{m-a}
 \qquad(m\geq a).
\]
Consequently $\alpha C_{a,z}\in A_3(m)$ for every
$\alpha\in A_2(m-a)$.
\end{prop}
\begin{proof}
Write $f=\sum_b f_bz^b$ and put $f_b=0$ for $b<0$.
By the Cartan formula, the coefficient of $z^a$ in $\Pp^s f$ is
\[
 \sum_{t=0}^s
 \binom{a-(p-1)t}{t}\Pp^{s-t}f_{a-(p-1)t},
\]
where a summand with $a-(p-1)t<0$ is zero. If $t>0$ and
$b=a-(p-1)t\geq0$, the coefficient
$\binom bt$ is the coefficient of $z^a$ in $\Pp^t z^b$.
Theorem~\ref{thm:rank-one} implies that $z^a$ is not hit.
It follows that $\binom bt=0$ in $\Fp$: a nonzero coefficient
would express $z^a$ as a nonzero scalar multiple of
$\Pp^t z^b$. Thus only $t=0$ remains, and its value is
$\Pp^s f_a$. This proves the displayed identity. In particular,
coefficient extraction sends each positive reduced-power image to
a positive reduced-power image. It therefore descends to the
quotients. Composing the induced map with a functional on
$(\QP_2)_{m-a}$ gives the last assertion.
\end{proof}

The case $a=0$ is restriction to $z=0$. Its usefulness depends on
whether it is applied to a class or to an equation for invariance.

\begin{prop}\label{prop:ranktwo-restrictions-invariants}
Let $p$ be an odd prime and let $m\geq0$.
If $[fU]\in W_m^{T_3}$, then $[f]$ has a representative divisible
by $xyz$. If $[fU]\in W_m^{G_3}$, then the image of $[f]$ under
every linear restriction from three variables to two variables
of rank two is zero in $(\QP_2)_m$.
\end{prop}
\begin{proof}
By Proposition~\ref{prop:torus-block}, every monomial in a
representative of a $T_3$-fixed class has each exponent congruent
to $p-2$ modulo $p-1$. Since $p$ is odd, each exponent is at
least $p-2>0$. This proves the divisibility assertion and shows
that restriction to $z=0$ annihilates this representative.

Suppose now that $[fU]$ is $G_3$-fixed. Since
$gU=\det(g)^{-1}U$, the equality $g[fU]=[fU]$ gives
$[gf]=\det(g)[f]$ in $(\QP_3)_m$. Restriction to $z=0$ is
Steenrod-linear and hence induces a map on the quotient.
Applying it to this equality gives zero. Every rank-two
restriction is obtained from restriction to $z=0$ by an
invertible change of coordinates in three variables and an
invertible change of coordinates in the two remaining
variables. Both changes preserve the hit subspace. Therefore
every such restriction sends $[f]$ to zero.
\end{proof}

The restriction of the full cohomology class $fU$ to a rank-two
subspace is also zero, because the image of $U$ lies in the third
exterior power of a two-dimensional vector space. The preceding
proposition concerns the polynomial coefficient, where the
vanishing requires the determinant character. It does not imply
that restrictions of all expressions $gf-f$ vanish for an
arbitrary $T_3$-fixed class. Such expressions can still impose
necessary equations for invariance.

For a fixed decomposition with variables $(x,y)$ and $z$, define
\begin{equation}\label{eq:ranktwo-product-orbit-space}
 B_{3,m}=\sum_{\substack{0\leq a\leq m\\
                         a+1=cp^r,\;1\leq c\leq p-1}}
      \{\alpha C_{a,z}:\alpha\in A_2(m-a)\},
 \qquad
 E_{3,m}=\Fp[G_3]B_{3,m}\subseteq A_3(m).
\end{equation}
Permutations of the coordinates belong to $G_3$, so the definition
of $E_{3,m}$ already includes coefficient extraction in each
variable. On $W_m$, these functionals are evaluated on the
polynomial coefficient of $U$. The dual determinant twist
multiplies each translate by a nonzero scalar and hence does not
change the span in~\eqref{eq:ranktwo-product-orbit-space}.

If $E_{3,m}=A_3(m)$, a class $w\in W_m^{T_3}$ is $G_3$-fixed
if and only if
\begin{equation}\label{eq:ranktwo-orbit-invariance-equations}
 \ell((g-1)w)=0
 \quad\text{for every $\ell\in E_{3,m}$ and every
 generator $g$ in Lemma~\ref{lem:gl-generators}}.
\end{equation}
Indeed, the forward implication follows from $(g-1)w=0$.
For the converse, the full dual separates vectors in $W_m$,
so each defect is zero. The changes of coordinates in
$E_{3,m}$ cannot in general be omitted.

\begin{prop}\label{prop:ranktwo-coordinate-tests-insufficient}
Let $p=3$, let $m=65$, and put
\[
 v=[x^{25}y^{23}z^{17}U]\in W_{65}.
\]
Then $v\in W_{65}^{T_3}$ and every functional of the form
$\alpha C_{a,x}$, $\alpha C_{a,y}$ or $\alpha C_{a,z}$,
with $a+1=cp^r$ and $\alpha\in A_2(65-a)$, vanishes on
$(g-1)v$ for every permutation of the variables and every
elementary transvection $x_j\mapsto x_j-cx_i$, where
$i\ne j$ and $c\in\mathbb F_3$. Nevertheless
$v\notin W_{65}^{G_3}$.
\end{prop}
\begin{proof}
The three exponents of $v$ are odd, so
Proposition~\ref{prop:torus-block} gives $T_3$-invariance.
The allowable values of $a$ are
\[
 0,1,2,5,8,17,26,53.
\]
Crossley's quotient basis in
\cite[Theorem~2 and Table~1]{Crossley1999Transfer}, restricted
to the top exterior component, gives
\[
 (\QP_2)_d=0\quad
 \text{for }d\in\{39,48,57,60,63,64,65\},
 \qquad \dim(\QP_2)_{12}=1.
\]
Thus only $a=53$ can give a nonzero functional. Define
\[
 \eta=\lambda_{(7,5)}-\lambda_{(5,7)}\in A_2(12).
\]
Here $\lambda_{(a,b)}$ denotes extraction of the coefficient
of $x^ay^b$. To check the displayed membership directly, the
coefficient of each of $x^7y^5$ and $x^5y^7$ in
$\Pp^1 f$, for $f\in(P_2)_{10}$, is twice the coefficient
of $x^5y^5$ in $f$. Their difference is zero. In an image
of $\Pp^3$ from degree $6$, a contribution to $x^7y^5$
has reduced-power indices $(t,3-t)$.
For $t=0$ the second exponent before applying the operation
would be negative. For $t=1,2,3$, the corresponding products
of binomial coefficients contain, respectively,
$\binom12$, $\binom32$ and $\binom13$, all zero in
$\mathbb F_3$. The same verification with the variables
interchanged treats $x^5y^7$. The generators $\Pp^1$ and
$\Pp^3$ give all hit relations in degree $12$, since the
next reduced-power generator increases polynomial degree
by $18$. Therefore $\eta$ annihilates hits. Moreover,
$\eta(x^7y^5)=1$, so it is nonzero and spans $A_2(12)$.

Put
\[
 \phi_x=\eta C_{53,x},\qquad
 \phi_y=\eta C_{53,y},\qquad
 \phi_z=\eta C_{53,z}.
\]
Every exponent of $x^{25}y^{23}z^{17}$ is less than $53$.
After one elementary transvection, an exponent is bounded
by one of $25+23=48$, $25+17=42$ and $23+17=40$.
These numbers are also less than $53$. A permutation
does not change the exponents. Hence all three functionals
vanish on $v$ and on every indicated $gv$. This proves the
assertion about their defects.

Let $t_{xy}$ send $y$ to $y-x$ and let $t_{xz}$ send $z$
to $z-x$, with the remaining variables fixed. Both have
determinant one. Proposition~\ref{prop:ranktwo-coefficient-extraction}
shows that $\phi_x t_{xz}$ is a well-defined functional on
$W_{65}$. Its value on $(t_{xy}-1)v$ is
\begin{align*}
 \phi_x t_{xz}(t_{xy}-1)v
 &=\binom{23}{7}\binom{17}{5}
      -\binom{23}{5}\binom{17}{7}\\
 &=0-1\cdot2=1\quad\text{in }\mathbb F_3.
\end{align*}
The binomial coefficients follow from Lucas's congruence;
the signs from the two substitutions are positive for
both monomials being compared. The term $t_{xz}v$ has
$y$-exponent $23$ and contributes to neither coefficient.
Thus $(t_{xy}-1)v$ is nonzero in the quotient, proving that
$v$ is not $G_3$-fixed.
\end{proof}

The same changes of coordinates that detect this defect give a
small certificate for the full dual in degree $65$. We record
the coefficient data explicitly, so that the distinction between
coordinate extraction and its $G_3$-span does not depend on a
large invariant matrix.

\begin{prop}\label{prop:ranktwo-orbit-spanning-degree65}
Let $p=3$ and let $m=65$. Then
\[
 E_{3,65}=A_3(65),\qquad \dim A_3(65)=13.
\]
\end{prop}
\begin{proof}
The quotient-dimension calculation in
Proposition~\ref{prop:finite-invariant-bases} gives
$\dim(\QP_3)_{65}=13$. In the order used below, consider
the monomials with exponent triples
\begin{align*}
 &(7,5,53),\ (7,41,17),\ (7,50,8),\ (7,53,5),\
   (25,23,17),\ (34,14,17),\ (34,23,8),\\
 &(43,5,17),\ (43,14,8),\ (43,17,5),\
   (52,5,8),\ (52,8,5),\ (53,7,5).
\end{align*}
Denote the corresponding classes by $q_0,\ldots,q_{12}$.
Use $\phi_x,\phi_y,\phi_z$ from the preceding proof.
Let $d_x$ and $d_y$ multiply $x$ and $y$, respectively,
by $-1$, and include the determinant twist on $U$.
For distinct variables $a,b$, let $t_{ab}$ send $b$ to
$b-a$ and fix the other variables. A product of operators
on a functional is read by precomposition: $\phi gh$
means $w\mapsto\phi(g(hw))$.

The thirteen functionals
\begin{align*}
 &\phi_x,\ \phi_y,\ \phi_z,\
 \phi_x t_{xy},\ \phi_x t_{xz},\ \phi_y t_{yz},\
 \phi_x t_{xy}d_x,\\
 &\phi_x t_{xy}t_{xz},\ \phi_x t_{xy}t_{yz},\
 \phi_x t_{xz}d_x,\ \phi_x t_{xz}t_{zy},\
 \phi_y t_{yz}d_y,\ \phi_x t_{xy}t_{xz}d_x
\end{align*}
have, on $(q_0U,\ldots,q_{12}U)$, the following matrix over
$\mathbb F_3$:
\begingroup\small\setcounter{MaxMatrixCols}{13}\setlength{\arraycolsep}{3pt}
\[
\begin{pmatrix}
0&0&0&0&0&0&0&0&0&0&0&0&1\\
0&0&0&1&0&0&0&0&0&0&0&0&0\\
1&0&0&0&0&0&0&0&0&0&0&0&0\\
0&0&0&2&0&0&0&0&0&2&0&1&1\\
1&0&0&0&0&0&0&1&0&0&2&0&1\\
2&2&1&1&0&0&0&0&0&0&0&0&0\\
0&0&0&2&0&0&0&0&0&2&0&2&1\\
1&1&1&2&1&2&2&1&1&2&2&1&1\\
1&1&2&2&1&2&1&1&2&2&1&1&1\\
1&0&0&0&0&0&0&1&0&0&1&0&1\\
1&1&2&2&1&1&2&1&2&2&2&2&1\\
2&2&2&1&0&0&0&0&0&0&0&0&0\\
1&1&1&2&1&1&1&1&1&2&1&2&1
\end{pmatrix}
\]
\endgroup
Every entry is a coefficient in the substitution specified
by its row. More explicitly, if $L_1,L_2,L_3$ are the linear
forms obtained by the substitution and $(a,b,c)$ is the
column exponent triple, an entry for $\phi_x$ is the
coefficient of $x^{53}y^7z^5$ minus the coefficient of
$x^{53}y^5z^7$ in $L_1^aL_2^bL_3^c$, multiplied by
the determinant character of the substitution. The formulas
for $\phi_y$ and $\phi_z$ are obtained by permuting the
coordinates. Expanding the three factors and applying
Lucas's congruence gives the displayed matrix.

Its determinant is $2$ in $\mathbb F_3$. Thus the thirteen
classes and the thirteen functionals are linearly independent.
The established dimension $13$ shows that both are bases.
Every functional in the display belongs to $E_{3,65}$ by
definition, which proves the required equality.
\end{proof}

\paragraph{An additional family detected by rank-two functionals.}
The rank-two construction also determines an additional infinite
family. In its initial polynomial degree, the diagonal fixed
subspace has only ten monomials before taking the quotient.
Four explicit hit relations and five coefficient equations
therefore suffice to determine the invariant line. Minami's
degree reduction applies at every subsequent step.

Define
\begin{equation}\label{eq:rankthree-degree9-representative}
 F_{3,9}=xy^3z^5-xy^7z-x^3yz^5-x^3y^5z+x^5y^3z+x^7yz.
\end{equation}

\begin{prop}\label{prop:ranktwo-detection-degree9}
Let $p=3$. Then
\[
 \bigl((\QP_3)_9\otimes\langle U\rangle\bigr)^{G_3}
   =\mathbb F_3[F_{3,9}U],
\]
and the displayed generator is nonzero.
\end{prop}
\begin{proof}
By Proposition~\ref{prop:torus-block}, a representative in
the diagonal fixed subspace is a sum of monomials with
all three exponents odd. There are ten such exponent
triples in degree $9$. In the same weight block, a hit
relation is generated by $\Pp^1$ on an odd exponent
triple of degree $7$ or by $\Pp^3$ on an odd exponent
triple of degree $3$. Indeed, reduced powers preserve
each exponent modulo $2$, and the description in
Proposition~\ref{prop:hit-matrix} uses precisely the
generators $\Pp^{3^j}$. The next generator, $\Pp^9$,
increases polynomial degree by $18$ and cannot contribute.

The six odd triples of degree $7$ are
\[
 (1,1,5),\ (1,3,3),\ (1,5,1),\
 (3,1,3),\ (3,3,1),\ (5,1,1).
\]
Their images give the four relations
\begin{align*}
 \Pp^1(xyz^5)&=x^3yz^5+xy^3z^5+2xyz^7,\\
 \Pp^1(xy^3z^3)&=x^3y^3z^3,\\
 \Pp^1(xy^5z)&=x^3y^5z+2xy^7z+xy^5z^3,\\
 \Pp^1(x^5yz)&=2x^7yz+x^5y^3z+x^5yz^3.
\end{align*}
The images of $x^3yz^3$ and $x^3y^3z$ are also
$x^3y^3z^3$. The unique odd triple of degree $3$ is
$(1,1,1)$, and $\Pp^3(xyz)=x^3y^3z^3$.
Thus the displayed four relations span the entire
hit subspace in this weight block. They are independent:
the monomials $xyz^7$, $x^3y^3z^3$, $xy^5z^3$ and
$x^5yz^3$ occur with nonzero coefficients in the first,
second, third and fourth relations, respectively, and
in none of the other displayed relations. Consequently
the six classes represented by
\begin{equation}\label{eq:degree9-torus-basis}
 (v_1,v_2,v_3,v_4,v_5,v_6)
   =(xy^3z^5,\ xy^7z,\ x^3yz^5,\
     x^3y^5z,\ x^5y^3z,\ x^7yz)
\end{equation}
give a basis after tensoring with $U$.

We now impose five necessary equations for $G_3$-invariance.
For two variables, let $\lambda_{(a,b)}$ denote the
coefficient of the monomial with exponent pair $(a,b)$.
Define functionals on $P_2$ by
\[
 \alpha_0=\lambda_{(1,8)},\quad
 \alpha_1=\lambda_{(4,5)}-\lambda_{(2,7)},\quad
 \alpha_2=\lambda_{(7,2)}-\lambda_{(5,4)},\quad
 \alpha_3=\lambda_{(8,1)},\qquad
 \beta=\lambda_{(0,8)}.
\]
The functionals $\alpha_0$, $\alpha_3$ and $\beta$
annihilate hits by Proposition~\ref{prop:spike-independence}.
To verify this for $\alpha_1$, the coefficient of each
of $x^4y^5$ and $x^2y^7$ in $\Pp^1 f$, with
$f\in(P_2)_7$, is twice the coefficient of $x^2y^5$
in $f$. Their difference is therefore zero. Neither
monomial has a nonzero coefficient in a $\Pp^3$ image
from degree $3$. For $x^4y^5$, the possible pairs of
reduced-power indices are $(1,2)$ and $(2,1)$; their
coefficients contain $\binom12$ and $\binom02$,
respectively. For $x^2y^7$, the possible pairs are
$(0,3)$ and $(1,2)$; their coefficients contain
$\binom13$ and $\binom01$, respectively. These
coefficients vanish. No higher generator contributes
to degree $9$. This proves that $\alpha_1$ annihilates
all hit relations. Interchanging the variables sends
$\alpha_1$ to $-\alpha_2$, so $\alpha_2$ also
annihilates hits. Hence
\[
 \alpha_0,\alpha_1,\alpha_2,\alpha_3\in A_2(9),
 \qquad \beta\in A_2(8).
\]

Evaluate these functionals on the remaining variables
$(y,z)$ after extracting the coefficient in $x$.
Proposition~\ref{prop:ranktwo-coefficient-extraction}
shows that
\[
 \ell_i=\alpha_i C_{0,x}\quad(0\leq i\leq3),
 \qquad \ell_4=\beta C_{1,x}
\]
are well-defined on $(\QP_3)_9$. The exponents $0$ and
$1$ satisfy the required rank-one condition at $p=3$.
We evaluate them on $W_9$ by extracting the polynomial
coefficient of $U$.

For distinct variables $a,b$, let $t_{ab}$ send $b$ to
$b-a$ and fix the other variables. These transvections
have determinant one. Apply the five functionals on defects
\[
 \ell_0(t_{zx}-1),\quad
 \ell_1(t_{yx}-1),\quad
 \ell_2(t_{zx}-1),\quad
 \ell_3(t_{yx}-1),\quad
 \ell_4(t_{zy}-1).
\]
On the basis~\eqref{eq:degree9-torus-basis}, their matrix is
\begin{equation}\label{eq:degree9-ranktwo-constraints}
 \begin{pmatrix}
 0&0&2&0&0&2\\
 2&0&2&0&0&0\\
 0&2&0&1&0&0\\
 0&2&0&2&2&2\\
 2&2&0&0&0&0
 \end{pmatrix}.
\end{equation}
Here the original values of all five functionals on
the torus basis are zero. For the first row, setting
$x=0$ after $x\mapsto x-z$ gives coefficient $-1=2$
precisely on $v_3$ and $v_6$. For the second row,
setting $x=0$ after $x\mapsto x-y$ gives coefficient
$2$ at $y^4z^5$ precisely on $v_1$ and $v_3$;
no basis monomial contributes to $y^2z^7$.
For the third row, $v_2$ contributes $2$ at $y^7z^2$,
whereas $v_4$ contributes $2$ at $y^5z^4$ and hence
$-2=1$ after applying $\alpha_2$.
For the fourth row, the pairs with $z$-exponent $1$
and $x$- plus $y$-exponent $8$ are exactly
$v_2,v_4,v_5,v_6$, each with coefficient $2$.
For the last row, setting $y=0$ after
$y\mapsto y-z$ gives coefficient $2$ at $xz^8$
on $v_1$ and $v_2$. These five comparisons give
every nonzero entry in~\eqref{eq:degree9-ranktwo-constraints}.

Write a fixed vector as $[\sum_{j=1}^6 c_jv_jU]$ and
put $a=c_1$. The last equation gives $c_2=-a$.
The second gives $c_3=-a$, and the first then gives
$c_6=a$. The third gives $c_4=-a$. Substitution in
the fourth gives $c_5=a$. Thus every fixed vector lies
in the line with coefficient vector
\[
 a(1,-1,-1,-1,1,1).
\]
Equivalently, the matrix has rank $5$; the minor obtained
by deleting its last column has determinant $1$ in
$\mathbb F_3$.

The corresponding cohomological degree is $21$.
It belongs to family~\eqref{eq:family2} with $(i,j)=(0,1)$.
Theorem~\ref{thm:target-dimensions} gives a one-dimensional
target, and Theorem~\ref{thm:surjectivity} gives a
surjective transfer. Hence $\dim D_{3,21}\geq1$.
By Corollary~\ref{cor:central-top} and the duality in
\eqref{eq:coinvariant-duality}, this is a lower bound
for the fixed space just considered. The upper bound
is one, so the fixed space is exactly the displayed
line. Its vector with $a=1$ is $[F_{3,9}U]$.
It is nonzero because its first coordinate in the
quotient basis is $1$.
\end{proof}

The initial degree has no exceptional first extraction step.
This permits the preceding calculation to be extended without
any further matrix calculation.

\begin{thm}\label{thm:additional-infinite-transfer-family}
Let $p=3$, let $r\geq0$, and put
\[
 n_r=24\cdot3^r-3,\qquad
 m_r=12\cdot3^r-3,\qquad
 F_r=\Phi_r(F_{3,9}),
\]
with $\Phi_r$ as in~\eqref{eq:infinite-lift}. Then
\[
 (\QH^{n_r}(3))^{G_3}=\mathbb F_3[F_rU],
\]
the displayed generator is nonzero, and
\[
 \Tr_{3,n_r}:D_{3,n_r}\longrightarrow
          \Ext_{\A}^{3,n_r+3}(\mathbb F_3,\mathbb F_3)
\]
is an isomorphism between one-dimensional vector spaces.
\end{thm}
\begin{proof}
For $m_0=9$, the notation of
Proposition~\ref{prop:digit-stability} gives
\[
 R=6,\qquad A=2(9+3)=24=(220)_3,
 \qquad A-1=23=(212)_3.
\]
Thus $\ds(A)=4\leq6$, while
$\ds(A-1)+2r=5+2r>5$ for every $r\geq1$.
The proposition therefore gives an equivariant isomorphism
\[
 (\QP_3)_{m_r}\otimes\langle U\rangle
   \xrightarrow{\ \overline\kappa_3\otimes1\ }
 (\QP_3)_{m_{r-1}}\otimes\langle U\rangle
\]
at every positive step. Since $2$ divides $m_0+3=12$,
the same proposition and Corollary~\ref{cor:central-top}
identify their fixed spaces with the duals of
$D_{3,n_r}$ and $D_{3,n_{r-1}}$.

The inverse-extraction identity for $\Phi_r$ was proved in
Corollary~\ref{cor:infinite-invariant-bases}. Since every
extraction in the present sequence is an equivariant
isomorphism, it applies starting at $m_0=9$.
Thus the inverse of the $r$-fold extraction sends
$[F_{3,9}U]$ to $[F_rU]$. Proposition~\ref{prop:ranktwo-detection-degree9}
identifies the initial invariant line, so these inverse
isomorphisms give precisely the stated nonzero invariant
line in every degree $n_r$.

Finally,
\[
 n_r=4(3^{r+1}+3^{r+1})-3
\]
is family~\eqref{eq:family2} with $(i,j)=(r,r+1)$.
The admissibility condition $j\ne i+2$ holds, and
$|j-i-1|=0$. Theorem~\ref{thm:target-dimensions}
therefore gives target dimension one for every $r$.
Theorem~\ref{thm:surjectivity} gives surjectivity,
and the invariant-space calculation gives domain
dimension one. The transfer is consequently an
isomorphism.
\end{proof}

\paragraph{Extraction and the remaining coinvariants.}
We write
\[
 E_{3,m}^{\perp}
 =\{v\in(\QP_3)_m:\alpha(v)=0
       \text{ for every }\alpha\in E_{3,m}\}.
\]
Since $E_{3,m}$ is a $G_3$-submodule, its annihilator is
also a $G_3$-submodule. The dual extraction map carries a tensor
product of rank-two and rank-one annihilated classes to a tensor
product of the same type. When extraction is an isomorphism, every
such product in the higher degree arises in this way. Consequently
both the quotient by these products and its dual annihilator are
preserved as $G_3$-modules.

\begin{prop}\label{prop:block-span-extraction}
Let $p$ be an odd prime, let $d\geq0$, and put
$m=pd+3(p-1)$. If
\[
 \overline\kappa_3:(\QP_3)_m\longrightarrow(\QP_3)_d
\]
is an isomorphism, then
\[
 \overline\kappa_3^*(B_{3,d})=B_{3,m},\qquad
 \overline\kappa_3^*(E_{3,d})=E_{3,m}.
\]
Consequently, $\overline\kappa_3^*$ induces a $G_3$-equivariant
isomorphism
\[
 A_3(d)/E_{3,d}\xrightarrow{\ \cong\ }A_3(m)/E_{3,m},
\]
and $\overline\kappa_3$ restricts to a $G_3$-equivariant
isomorphism
\[
 E_{3,m}^{\perp}\xrightarrow{\ \cong\ }E_{3,d}^{\perp}.
\]
In particular, $E_{3,m}=A_3(m)$ if and only if
$E_{3,d}=A_3(d)$.
\end{prop}
\begin{proof}
Let $b_{a_1,\ldots,a_h}$ denote the functional on the polynomial
algebra that extracts the coefficient of
$x_1^{a_1}\cdots x_h^{a_h}$. Dualizing
\eqref{eq:extraction-map-general} gives
\begin{equation}\label{eq:extraction-supported-dual}
 \kappa_h^*(b_{a_1,\ldots,a_h})
 =b_{pa_1+p-1,\ldots,pa_h+p-1}.
\end{equation}
By~\eqref{eq:extraction-on-cohits}, this map restricts to an
injection from $A_h(e)$ to $A_h(pe+h(p-1))$ for every
$e\geq0$. In view of Theorem~\ref{thm:rank-one}, the fixed
coordinate space in~\eqref{eq:ranktwo-product-orbit-space}
can also be written as
\[
 B_{3,t}=\sum_{a+b=t}A_2(a)\otimes A_1(b),
\]
where the tensor product acts on the variables $(x,y)$ and $z$.
Formula~\eqref{eq:extraction-supported-dual} shows that
\[
 \overline\kappa_3^*(\alpha\otimes\beta)
 =\overline\kappa_2^*(\alpha)
       \otimes\overline\kappa_1^*(\beta).
\]
The degrees on the right are $pa+2(p-1)$ and $pb+p-1$
when $\alpha\in A_2(a)$ and $\beta\in A_1(b)$.
Their sum is $m$ when $a+b=d$. Hence
$\overline\kappa_3^*(B_{3,d})\subseteq B_{3,m}$.

For the reverse inclusion, consider a nonzero pure tensor
$\alpha\otimes\beta$ with
$\alpha\in A_2(a)$, $\beta\in A_1(b)$, and $a+b=m$.
It belongs to $A_3(m)$ by
Proposition~\ref{prop:ranktwo-coefficient-extraction}.
The assumed isomorphism gives
$A_3(m)=\overline\kappa_3^*(A_3(d))$. Therefore
\eqref{eq:extraction-supported-dual} implies that
$\alpha\otimes\beta$ vanishes on every monomial for which
at least one exponent is not congruent to $p-1$ modulo $p$.

Choose a monomial on which $\beta$ is nonzero. If either
exponent of a monomial in $(x,y)$ fails that congruence,
evaluation of the tensor product on its product with the
chosen monomial is zero. Since the value of $\beta$ is
nonzero, the value of $\alpha$ on that monomial is zero.
Thus every monomial on which $\alpha$ is nonzero has both
exponents congruent to $p-1$ modulo $p$. Choosing instead
a monomial on which $\alpha$ is nonzero proves the same
support assertion for $\beta$ in its single variable.
In particular,
\[
 a=pa'+2(p-1),\qquad b=pb'+p-1
\]
for nonnegative integers $a',b'$ with $a'+b'=d$.
Formula~\eqref{eq:extraction-supported-dual} consequently
gives unique polynomial functionals $\alpha'$ and $\beta'$
of degrees $a'$ and $b'$ such that
$\alpha=\kappa_2^*\alpha'$ and
$\beta=\kappa_1^*\beta'$.

We verify that these functionals annihilate hits. More
generally, let $\lambda$ be a polynomial functional with
$\kappa_h^*\lambda\in A_h(pe+h(p-1))$, where
$\lambda$ has degree $e$. Let $s>0$ and let $f$ be a
homogeneous polynomial of degree $e-(p-1)s\geq0$.
Set
\[
 \widetilde f=(x_1\cdots x_h)^{p-1}f^p.
\]
Then $\kappa_h(\widetilde f)=f$ by
\eqref{eq:extraction-map-general}. Equation
\eqref{eq:extraction-powers-general} gives
\[
 \lambda(\Pp^s f)
 =\lambda\bigl(\kappa_h\Pp^{ps}\widetilde f\bigr)
 =(\kappa_h^*\lambda)(\Pp^{ps}\widetilde f)=0.
\]
The final equality holds because $ps>0$ and the functional
$\kappa_h^*\lambda$ annihilates hits. This proves
$\lambda\in A_h(e)$, without asserting that
$f\mapsto\widetilde f$ preserves hits. Applying the
argument to $\alpha'$ and $\beta'$ gives
$\alpha'\in A_2(a')$ and $\beta'\in A_1(b')$.
It follows that
\[
 \alpha\otimes\beta
 =\overline\kappa_3^*(\alpha'\otimes\beta')
 \in\overline\kappa_3^*(B_{3,d}).
\]
Such pure tensors span $B_{3,m}$, which proves the first
equality in the statement.

By Lemma~\ref{lem:extraction-equivariant},
$\overline\kappa_3^*$ is $G_3$-equivariant. Taking
$G_3$-spans in the equality just proved yields
$\overline\kappa_3^*(E_{3,d})=E_{3,m}$.
The isomorphism on the quotient modules follows because
$\overline\kappa_3^*:A_3(d)\to A_3(m)$ is an
isomorphism. Finally, for $v\in(\QP_3)_m$ and
$\ell\in E_{3,d}$, duality gives
\[
 (\overline\kappa_3^*\ell)(v)
 =\ell(\overline\kappa_3v).
\]
Since $\overline\kappa_3^*(E_{3,d})=E_{3,m}$, this
identity shows that $v\in E_{3,m}^{\perp}$ if and only
if $\overline\kappa_3v\in E_{3,d}^{\perp}$. The
isomorphism on the annihilators and the final equivalence
follow.
\end{proof}

The degree-$65$ certificate gives a spanning theorem throughout
one of the fourth-family sequences. Its first extraction
isomorphism follows from the quotient dimensions, even though
the sufficient digit criterion does not apply at that step.

\begin{corl}\label{cor:ranktwo-span-infinite-family}
Let $p=3$, let $r\geq0$, and put $m_r=68\cdot3^r-3$.
Then $E_{3,m_r}=A_3(m_r)$.
\end{corl}
\begin{proof}
Proposition~\ref{prop:ranktwo-orbit-spanning-degree65} gives
the assertion for $m_0=65$. Tables~\ref{tab:rankthree-ranks}
and~\ref{tab:four-shifted-ranks} give
\[
 \dim(\QP_3)_{65}=\dim(\QP_3)_{201}=13.
\]
The extraction map from degree $201=3\cdot65+6$ to degree
$65$ is surjective by~\eqref{eq:extraction-on-cohits}.
Equal dimensions make it an isomorphism. Every subsequent
extraction in this sequence is an isomorphism by
Theorem~\ref{thm:uniform-common-shift}, applied to
family~\eqref{eq:family4} with initial indices $(i,j)=(0,3)$.
Induction using Proposition~\ref{prop:block-span-extraction}
therefore proves the assertion for all $r$.
\end{proof}

\paragraph{Coinvariants generated by rank-two products.}
The spanning condition has a direct implication for coinvariant
dimensions in the first and fourth families. The contribution of
the quotient not generated by the rank-two products can be
identified exactly with the kernel of the transfer. Its dual is the invariant subspace of the annihilator already
defined above, after tensoring with the top exterior line.

\begin{prop}\label{prop:block-product-transfer}
Let $p$ be an odd prime, let $n$ belong to the first family
\eqref{eq:family1} or the fourth family \eqref{eq:family4},
and put $m=(n-3)/2$. Let
$Q_m=A_3(m)/E_{3,m}$ and
$L_3=\Lambda^3(V_3)$.
Then
\[
 \ker\Tr_{3,n}\cong(Q_m\otimes L_3)_{G_3},
 \qquad
 \dim D_{3,n}
 =1+\dim\bigl(E_{3,m}^{\perp}\otimes\langle U\rangle\bigr)^{G_3}.
\]
Consequently, $\Tr_{3,n}$ is an isomorphism if and only if
$\bigl(E_{3,m}^{\perp}\otimes\langle U\rangle\bigr)^{G_3}=0$.
\end{prop}
\begin{proof}
Let $L_h$ be the top exterior homology line, on which $G_h$
acts by the determinant character. Under the block-diagonal
inclusion $G_2\times G_1\subseteq G_3$ one has
$L_3=L_2\otimes L_1$. By
Corollary~\ref{cor:central-top} and
\eqref{eq:coinvariant-duality},
\[
 D_{3,n}\cong\bigl(A_3(m)\otimes L_3\bigr)_{G_3}.
\]
Let $T\subseteq D_{3,n}$ be the image of
$(E_{3,m}\otimes L_3)_{G_3}$.
The right exactness of coinvariants gives an exact sequence
\begin{equation}\label{eq:block-residual-coinvariants}
 0\longrightarrow T\longrightarrow D_{3,n}
 \xrightarrow{\pi}(Q_m\otimes L_3)_{G_3}
 \longrightarrow0.
\end{equation}
By the definition of $E_{3,m}$, the subspace $T$ is spanned
by $G_3$-translates of the tensor products
\[
 (A_2(a)\otimes L_2)\otimes
 (A_1(b)\otimes L_1),\qquad a+b=m.
\]
The determinant twist only multiplies a translate by a
nonzero scalar, and the class of a translate is equal to
the original class after passage to $G_3$-coinvariants.
Relations from $G_2\times G_1$ are among the
$G_3$-coinvariant relations, so each product factors through
the coinvariants of its two factors.

In rank one the indicated tensor factor is the entire
annihilated homology in degree $2b+1$. In rank two, positive
even cohomological degree has only exterior degrees zero
and two. Every positive-degree polynomial monomial is a
Bockstein image: if $x_i$ divides $f$, then
$f=\beta((f/x_i)u_i)$. Hence the exterior-degree-zero
summand of the cohit quotient vanishes. Dualizing identifies
the indicated top exterior homology factor with the entire
annihilated homology in degree $2a+2$.
Consequently there is a surjection
\begin{equation}\label{eq:block-product-surjection}
 \bigoplus_{a+b=m}
 D_{2,2a+2}\otimes D_{1,2b+1}
 \longrightarrow T.
\end{equation}

Put $q=2(p-1)$ and $N=(n+3)/q$. By
\cite[Corollary~1]{Crossley1999Transfer}, a nonzero
rank-one factor in this sum has degree
$2b+1=qp^t-1$ with $t\geq0$ and has dimension one.
The nonzero positive even-degree rank-two factors have
dimension one and degree $2a+2=qC-2$, where
\begin{equation}\label{eq:block-ranktwo-degree-patterns}
 C=p^{s+1},\qquad
 C=p^{s+1}+2p^s,\qquad
 C=p^r+p^s\ (r\geq s+2),\qquad
 C=2p^{s+1}+p^s
\end{equation}
for nonnegative indices. Thus every nonzero tensor factor
in~\eqref{eq:block-product-surjection} must satisfy
$C+p^t=N$.

Suppose first that $n$ belongs to the first family, so
$N=p^k+p^j+p^i$ with the stipulated separations.
Modulo $p-1$, the right-hand side is three. The first
possibility in~\eqref{eq:block-ranktwo-degree-patterns}
gives $C+p^t\equiv2$, and the second and fourth give
$C+p^t\equiv4$. Neither congruence can hold, since the
difference from three is one in absolute value and
$p-1\geq2$. Only $C=p^r+p^s$ remains.
Here $r\ne s$, so a repeated value of $t$ would give a
digit two, which is strictly less than every odd prime.
There is no carry in either case. Uniqueness of
base-$p$ expansion therefore implies
$\{r,s,t\}=\{i,j,k\}$ with three distinct indices.

For each $t\geq0$, choose the rank-one class
$\eta_t\in D_{1,qp^t-1}$ whose transfer is $h_t$; its
existence and uniqueness follow from
\cite[Theorem~1]{Crossley1999Transfer}.
For $r\geq s+2$, the nonzero class $h_sh_r$ is listed in
\cite[Theorem~3]{Crossley1999Transfer}.
Multiplicativity, proved in
\cite[Proposition~2]{Crossley1999Transfer}, gives
\[
 \Tr_2(\eta_s\eta_r)=h_sh_r\ne0.
\]
Since $D_{2,q(p^r+p^s)-2}$ has dimension one,
$\eta_s\eta_r$ spans it. Every nonzero tensor factor in
\eqref{eq:block-product-surjection} therefore maps to a
scalar multiple of a product of $\eta_i,\eta_j,\eta_k$.
Associativity identifies the parenthesizations. A
permutation of the factors is induced by a coordinate
permutation in $G_3$, with the homological sign for
permuting the factors. All three classes have odd
homological degree, so that sign is the sign of the
permutation and is nonzero in $\Fp$. Hence all the
images span at most the line generated by
$\eta_i\eta_j\eta_k$.

Suppose next that $n$ belongs to the fourth family. Then
\[
 N=2p^{i+1}+p^i+p^j,\qquad
 j\notin\{i-1,i,i+1,i+2\},
\]
so its base-$p$ digit sum is four. The expression $C+p^t$
has total coefficient sum two for $C=p^{s+1}$ and three
for $C=p^r+p^s$. Replacing a coefficient at least $p$ by
a carry decreases the coefficient sum by $p-1$.
Consequently neither of these possibilities can produce
$N$. In the remaining two possibilities the initial
coefficient sum is four. Equality with $N$ therefore
requires that no carry occur, including when $p=3$.

If $C=p^{s+1}+2p^s$, the cases $t=s$ and $t=s+1$
produce, respectively, a coefficient three or two
coefficients equal to two, and cannot give the three
digits of $N$. Thus $t$ is distinct from $s$ and $s+1$.
The unique digit two then forces $s=i+1$, and the other
digit at $s+1=i+2$ forces $j=i+2$, which is excluded.
If $C=2p^{s+1}+p^s$, the cases $t=s+1$ and $t=s$
are excluded by the same digit patterns.
The unique digit two now gives $s=i$, and comparison
of the other two digit positions gives $t=j$.
Thus exactly one tensor factor in
\eqref{eq:block-product-surjection} can be nonzero,
and both its factors have dimension one.

Both cases prove $\dim T\leq1$. In the first family,
the product $\eta_i\eta_j\eta_k$ lies in $T$ and transfers
to the nonzero class $h_ih_jh_k$. In the fourth family,
choose the rank-two preimage $\xi_i$ of $\nu_i$ using
\cite[Theorem~1]{Crossley1999Transfer}. Its degree is the
positive even degree $q(2p^{i+1}+p^i)-2$, so its product
with $\eta_j$ lies in $T$. Multiplicativity gives
$\Tr_3(\xi_i\eta_j)=\nu_ih_j\ne0$.
The nonvanishing and one-dimensionality of the targets
in both cases are given by
Theorem~\ref{thm:target-dimensions}.
It follows that $\dim T=1$ and that
$\Tr_{3,n}|_T$ is an isomorphism onto the target.

The map $\pi$ in~\eqref{eq:block-residual-coinvariants}
therefore restricts to an injective map on
$\ker\Tr_{3,n}$: its kernel is
$T\cap\ker\Tr_{3,n}=0$.
For any $w\in(Q_m\otimes L_3)_{G_3}$, choose
$v\in D_{3,n}$ with $\pi(v)=w$.
There is a unique $t\in T$ satisfying
$\Tr_{3,n}(t)=\Tr_{3,n}(v)$.
Then $v-t\in\ker\Tr_{3,n}$ and $\pi(v-t)=w$.
This proves the asserted isomorphism.

Finally, the dual of $Q_m$ is $E_{3,m}^{\perp}$ under
the perfect pairing with $(\QP_3)_m$. Since
$L_3^*=\langle U\rangle$, coinvariant duality gives
\[
 \bigl((Q_m\otimes L_3)_{G_3}\bigr)^*
 \cong
 \bigl(E_{3,m}^{\perp}\otimes\langle U\rangle\bigr)^{G_3}.
\]
Taking dimensions in~\eqref{eq:block-residual-coinvariants}
and using $\dim T=1$ proves the dimension formula.
The isomorphism criterion follows from the identified
kernel.
\end{proof}

The equality $E_{3,m}=A_3(m)$ makes the obstruction in
Proposition~\ref{prop:block-product-transfer} zero, but
that equality is not required for injectivity. It would suffice to prove that
the annihilator of the rank-two products has no invariant
after tensoring with the top exterior line. Corollary~\ref{cor:ranktwo-span-infinite-family}
establishes this condition throughout the fourth-family
sequence $n=136\cdot3^r-3$ at $p=3$.

For the fourth family, subtracting the smaller of $i$ and $j$
from both indices leaves exactly the pairs $(0,s)$ with $s\geq3$
and $(s,0)$ with $s\geq2$. The exclusions in
\eqref{eq:family4} give these two ranges directly.
Theorem~\ref{thm:uniform-common-shift} and
Proposition~\ref{prop:block-span-extraction} therefore reduce the
remaining invariant vanishing assertion to the polynomial degrees
\[
 \begin{aligned}
 m&=p^r(p-1)(2p+1+p^s)-3,
       &&s\geq3,\quad r\in\{0,1\},\\
 m&=p^r(p-1)(2p^{s+1}+p^s+1)-3,
       &&s\geq2,\quad r\in\{0,1\}.
 \end{aligned}
\]
The parameter $s$ and the odd prime $p$ remain unrestricted.
The extraction isomorphisms preserve the quotient
$A_3(m)/E_{3,m}$ and hence cannot, by themselves, prove the
vanishing of its determinant-twisted coinvariants.

\subsection{Composition factors and invariant subspaces}
\label{subsec:composition-factors}
The remaining invariant dimensions can also be approached through
the finite-group module structure. We state the precise
composition-multiplicity bound that would imply injectivity,
and distinguish it from an unsupported equality between
composition multiplicities and invariant dimensions.
The distinction between a composition factor and a submodule is
essential in a comparison with the rank-two calculation. In
\cite[\S2, Table~2 and Comment~2]{Crossley1999Transfer}, the
composition factors are supplemented by a description of those
which occur as submodules and of the exceptional direct-sum
decompositions. It is the trivial submodules which determine
the invariant dimensions. For a finite-dimensional $\Fp G_3$-module
$W$, write $[W:\Fp]$ for the multiplicity of the trivial module
in a composition series, and write $\operatorname{soc}W$ for the
sum of its simple submodules. The standard identities
\[
 \dim W^{G_3}
 =\dim\operatorname{Hom}_{\Fp G_3}(\Fp,W)
 =[\operatorname{soc}W:\Fp]
 \leq[W:\Fp]
\]
follow from the description of the socle and of the largest
trivial submodule; see~\cite[\S6.3]{Webb2016}.
The fixed-space dimension is not determined by the class of a module
in the Grothendieck group of finite-dimensional modules, even when
the scalar matrices act trivially. Put $k=\Fp$, $G=G_3$,
$\overline G=G/Z(G)$, and $R=k[\overline G]$, regarded as a
$kG$-module through the quotient homomorphism. Let
$I=\ker(\varepsilon:R\to k)$, where
$\varepsilon(\sum_{\bar g}a_{\bar g}\bar g)=\sum_{\bar g}a_{\bar g}$
is the augmentation. The exact sequence
$0\to I\to R\to k\to0$ gives
\[
 [R]=[I\oplus k]\qquad\text{in }G_0(kG),
\]
with the Grothendieck-group convention of~\cite[\S9.5]{Webb2016}.
An element of $R$ is fixed by $G$ if and only if its coefficients
in the basis $\overline G$ are all equal. Thus
$R^G=k\sum_{\bar g\in\overline G}\bar g$ is one-dimensional.
The scalar subgroup has order $p-1$, whereas $p$ divides $|G|$,
so $p$ divides $|\overline G|$. Consequently the displayed
sum belongs to $I$, and $I^G=R^G$. Therefore
\[
 \dim R^G=1,\qquad\dim(I\oplus k)^G=2.
\]
These modules have the same composition multiplicities and
different invariant dimensions, and both have trivial scalar action.

The same example shows that the isomorphism classes of the kernel
and quotient modules do not determine the connecting homomorphism.
For $0\to I\to R\xrightarrow{\varepsilon}k\to0$, the map
$R^G\to k$ is zero because the augmentation of the displayed sum
is $|\overline G|=0$ in $k$. Exactness therefore makes the connecting
homomorphism $k\to H^1(G,I)$ injective. For the split sequence
$0\to I\to I\oplus k\to k\to0$, the map on fixed subspaces
onto $k$ is surjective, so the connecting homomorphism is zero.
The kernel, quotient, and cohomology group are the same in the two
sequences, while the connecting homomorphisms have ranks one and zero.

In particular, composition multiplicities provide an upper bound
even when the extensions between the factors have not been
determined. Together with the surjectivity already established,
this gives a sufficient criterion for injectivity.

\begin{prop}\label{prop:composition-multiplicity-criterion}
Let $p$ be an odd prime, let $n$ belong to one of the four families
\eqref{eq:family1}--\eqref{eq:family4}, and put
$m=(n-3)/2$.
If
\[
 [W_m:\Fp]\leq
 \dim\Ext_{\A}^{3,n+3}(\Fp,\Fp),
\]
then $\Tr_{3,n}$ is an isomorphism.
\end{prop}
\begin{proof}
By Corollary~\ref{cor:central-top} and
\eqref{eq:coinvariant-duality},
$\dim D_{3,n}=\dim W_m^{G_3}$.
Theorem~\ref{thm:surjectivity} implies
\[
 \dim\Ext_{\A}^{3,n+3}(\Fp,\Fp)
 \leq\dim W_m^{G_3}.
\]
The fixed subspace is a direct sum of $\dim W_m^{G_3}$
copies of the trivial module and is a submodule of $W_m$.
Consequently its multiplicity cannot exceed $[W_m:\Fp]$.
The hypothesis therefore makes all inequalities in
\[
 \dim\Ext_{\A}^{3,n+3}(\Fp,\Fp)
 \leq\dim W_m^{G_3}
 \leq[W_m:\Fp]
 \leq\dim\Ext_{\A}^{3,n+3}(\Fp,\Fp)
\]
equalities. The surjective transfer has equal-dimensional
domain and target, so its kernel is zero.
\end{proof}

In the first and fourth families,
Proposition~\ref{prop:block-product-transfer} gives the more specific
criterion
\[
 \Tr_{3,n}\text{ is an isomorphism}
 \quad\Longleftrightarrow\quad
 \operatorname{Hom}_{\Fp G_3}
 \bigl(\det,E_{3,m}^{\perp}\bigr)=0,
\]
because $\langle U\rangle$ is the inverse determinant representation.
In particular, the composition-multiplicity condition
$[E_{3,m}^{\perp}:\det]=0$ is sufficient. Determining these
multiplicities would give a sufficient bound, while determining
whether the determinant representation occurs in the socle gives
the exact obstruction.

An application using Weyl modules would require a filtration
whose successive quotients are identified as representations of
the finite group $G_3$. The usual polynomial representation of
the algebraic group $\mathrm{GL}_3$ does not descend to the
Steenrod quotient in every degree. The following calculation
specifies this obstruction.

\begin{prop}\label{prop:hit-not-algebraic-group-stable}
Let $p$ be an odd prime, let $k=\overline{\Fp}$, and let
$P_3=\Fp[x,y,z]$ with polynomial generators of degree one.
Then the subspace
\[
 k\otimes_{\Fp}(\Hit_3)_{p+1}
 \subseteq k[x,y,z]_{p+1}
\]
is not stable under all substitutions from
$\mathrm{GL}_3(k)$.
\end{prop}
\begin{proof}
The proof of Theorem~\ref{thm:first-power-kernel}
identifies $(\Hit_3)_{p+1}$ as the span of
\[
 2x^{p+1},\quad 2y^{p+1},\quad 2z^{p+1},\quad
 x^py+xy^p,\quad x^pz+xz^p,\quad y^pz+yz^p.
\]
The two monomials
$x^py$ and $xy^p$ occur only in the polynomial
$\Pp^1(xy)$ among the six displayed generators.
Thus their coefficients are equal in every element of
$k\otimes_{\Fp}(\Hit_3)_{p+1}$.

Choose $\lambda\in k^\times$ such that $\lambda^p\ne\lambda$,
and apply the substitution $x\mapsto\lambda x$, $y\mapsto y$,
$z\mapsto z$. It sends the hit polynomial
$\Pp^1(xy)=x^py+xy^p$ to
\[
 \lambda^p x^py+\lambda xy^p.
\]
Its two indicated coefficients are different, so the resulting
polynomial does not belong to the scalar extension of the hit
subspace. This substitution is induced by an element of
$\mathrm{GL}_3(k)$, also under the contragredient convention
by taking the inverse diagonal matrix. The asserted failure
of stability follows.
\end{proof}

Proposition~\ref{prop:hit-not-algebraic-group-stable} concerns
the action induced by polynomial substitutions. It does not
exclude a finite-group filtration obtained by another
construction. To apply
Proposition~\ref{prop:composition-multiplicity-criterion} in all
normalized degrees, one still has to establish the required
composition-multiplicity bound uniformly in the prime and the
independent index differences.

Proposition~\ref{prop:transfer-degree-zero},
Theorems~\ref{thm:infinite-isomorphisms}
and~\ref{thm:additional-infinite-transfer-family}, and
Corollary~\ref{cor:extended-dimension-jump-family}
establish the degree-zero case and the stated infinite
families in the injectivity assertion of
Corollary~2.9 of~\cite[p.~7]{PhucOriginal}.
Theorem~\ref{thm:uniform-common-shift} reduces the common
index shifts for every odd prime.
Proposition~\ref{prop:transfer-forces-connecting-zero}
separates the two contributions to the transfer kernel at
an inductive step, while
Proposition~\ref{prop:block-product-transfer} identifies
the exact obstruction from the quotient by rank-two
products in the first and fourth families.
Neither result determines that obstruction in all
normalized degrees. The relative index differences and
the odd prime remain unrestricted in
Corollary~\ref{cor:uniform-boundary-reduction}, and the
corresponding invariant dimensions have not all been
determined here. The counterexamples above concern
particular comparison maps and coefficient tests; they
do not disprove injectivity of the transfer in the
remaining degrees.

\begin{nx}\label{rem:comments-context}
The comments in~\cite{Sum2025} require distinctions between
mathematical assertions, their hypotheses, and the evidence offered
for them. The observations on the classical rank-one classification,
the missing citation for Crossley's example
\cite[p.~175]{Crossley1999Bases}, the defective
kernel argument, and the insufficiency of finite calculations
for an infinite-family assertion are valid and are addressed at the
corresponding places above. The following points concern the
scope of the remaining comments.

In the quotations of~\cite[\S2, items~5--6]{Sum2025},
$\Pp^k$ replaces the $\Pp^{p^k}$ printed in
\cite[Theorems~2.6--2.7, p.~6]{PhucOriginal}.
The original $k=0$ term is therefore $\Pp^1$.
This discrepancy does not validate the multiplier assertion:
Proposition~\ref{prop:limited-counterexample} satisfies its
actual restricted premise and disproves both that assertion
and the unconditional preservation statement proposed
in~\cite[\S2, item~6]{Sum2025}.
The replacement is Theorem~\ref{thm:bounded-multiplier},
with its two explicit degree restrictions.
By contrast, the map in~\cite[Theorem~2.6]{PhucOriginal}
is well defined on indecomposables and is zero, as
\cite[\S2, item~5]{Sum2025} itself observes and
Theorem~\ref{thm:zero-frobenius-map} records. Its redundant
hypothesis must not be confused with a failure of the
quotient-map conclusion.

The specialization $q_1=q_2=0$ discussed
in~\cite[\S2, item~6]{Sum2025} is outside the positive-$q_j$
hypotheses of~\cite[Theorem~2.7]{PhucOriginal}; it is a
comparison with a related multiplier, rather than a case
of the printed statement. The introduction to
\cite[Theorems~2.6--2.7, p.~6]{PhucOriginal} explicitly
cites the results of Tanay and \"Oner~\cite{TanayOner2013}.
Thus there is a textual
basis for discussing that dependence. However, the perceived
difficulty of a related proof does not establish the validity
or invalidity of the asserted theorem, and it does not imply
a conclusion about an author's ability to prove it. The
mathematical issue is whether the cited identities preserve
the specified hit subspace. The counterexamples and the
bounded replacement above settle that issue directly.

The lack of a previously published quotient basis, mentioned
in~\cite[\S1, item~7]{Sum2025}, is not an obstruction to a
calculation in a fixed degree. Proposition~\ref{prop:hit-matrix}
constructs such a basis from the complete image relations,
and Theorem~\ref{thm:gl-invariants} gives its fixed-vector
equations. These constructions alone do not turn finitely many
calculations into a proof for all primes and degrees.
The additional argument in Subsection~\ref{subsec:infinite-transfer}
supplies an equivariant reduction with a verified digit hypothesis
and treats its four exceptional first steps separately.
Theorem~\ref{thm:infinite-isomorphisms} therefore proves
isomorphisms in six infinite families.
Theorem~\ref{thm:additional-infinite-transfer-family}
adds the family $p=3$, $n=24\cdot3^r-3$, and
Corollary~\ref{cor:extended-dimension-jump-family}
includes the initial degree of the dimension-changing family.
These results supplement the degree-zero statement of
Proposition~\ref{prop:transfer-degree-zero}.
They do not assert the unrestricted conclusion of
\cite[Theorem~2.8]{PhucOriginal} for all relative index gaps
and all odd primes. The failure of the one-dimensional assertion
in that theorem also does not, by itself, disprove that the
transfer is an isomorphism: throughout the family
$p=3$, $n=120\cdot3^r-3$, both spaces have dimension two
and the transfer is an isomorphism.

Finally, the joint assessment of
\cite[Theorem~2.8 and Corollary~2.11]{PhucOriginal}
in~\cite[\S1, item~7]{Sum2025} must distinguish the
argument concerning $\dim D_{3,n}$ from the independent product
argument for surjectivity. Theorem~\ref{thm:remaining-surjectivity}
verifies all target-basis entries and recovers the latter
conclusion using Crossley's multiplicative low-rank transfer
\cite[Theorem~1 and Proposition~2]{Crossley1999Transfer}.
It makes no claim of injectivity in those six families.
\end{nx}

\appendix
\section{Explicit invariant bases in rank two}\label{app:ranktwo}

The purpose of this appendix is to determine explicit bases of
$\QH^n(2)^{G_2}$ in every degree in which this invariant space is
nonzero. Crossley~\cite[Corollary~1]{Crossley1999Transfer}
determined these degrees and proved that each corresponding
invariant space is one-dimensional. That corollary records the
dimensions without listing representatives for the invariant
classes in all the degree families. The bases of the full quotient $\QH^n(2)$ were already given
in~\cite[Theorem~2 and Table~1]{Crossley1999Transfer}. Theorem~\ref{thm:ranktwo-explicit-invariants}
supplements the invariant-dimension statement by specifying a
nonzero polynomial--exterior representative in each degree family.

Crossley's determination of the invariant dimensions uses the
$G_2$-module structure described in
\cite[\S2, Table~2 and the comments following it]{Crossley1999Transfer}.
The composition factors are accompanied by information identifying
which factors occur as submodules, including the exceptional
direct-sum decompositions. The invariant dimensions follow by
locating the trivial submodules. In positive even degrees, the
module description is obtained from the dual of the annihilated
polynomial homology computed in~\cite{Crossley1996}, tensored
with the dual determinant character carried by $uv$; the odd-degree
components are treated separately in that description.

The calculations below work directly with coordinates relative to
Crossley's quotient bases. We first select the vectors fixed by the diagonal
subgroup, taking account of the exterior factor, and then impose
coefficient equations arising from an interchange of the variables
or a transvection. The Bockstein relations are included in exterior
degree one, and coefficient functionals and Lucas's congruence
justify the comparisons in the quotient. The equivariant extraction map in
Lemma~\ref{lem:ranktwo-extraction} reduces the calculation for the
family $M_s$ to its initial degree. These steps identify a one-dimensional subspace containing
every $G_2$-invariant vector in the degree under consideration.
Crossley's dimension theorem then proves that the identified
subspace is the whole invariant space. Thus the explicit
representatives are obtained from equations for invariance written
with respect to the known quotient bases, while the degrees, dimensions, and
existence of the invariant subspaces are supplied by Crossley's
module-theoretic calculation.

We use the rank-two notation of
Subsection~\ref{subsec:transfer-preliminaries}.
The symbols $[z]$ denote classes in $\QH^n(2)$, where $n$ is the
cohomological degree of $z$.  For $n>0$ even, multiplication by $uv$
identifies $\QH^n(2)$ with
$(\QP_2)_{n/2-1}\otimes\langle uv\rangle$: positive-degree polynomial
classes are Bockstein images, and no Bockstein enters exterior degree
two. To verify the first assertion on a monomial, write a
positive-degree polynomial monomial as $x_i f$, where
$x_i$ is either $x$ or $y$. Then $x_i f=\beta(fu_i)$,
with $u_i$ the corresponding exterior generator.  In odd degrees the quotient is in exterior degree one.
In degree zero, $\QH^0(2)=\Fp$, whereas its exterior-degree-two
component is zero.  These assertions are also instances of
Proposition~\ref{prop:exterior-slices}.

For $s\geq0$, define
\begin{align}
 R&=x^{p-2}y^{p-1}u,\label{eq:ranktwo-rho}\\
 \Lambda_s&=
 x^{(p-1)p^s-1}y^{(p-1)^2p^s-1}uv,\label{eq:ranktwo-lambda}\\
 M_s&=\bigl(
 x^{(p-1)p^s-1}y^{(p^2-1)p^s-1}
 +2x^{p(p-1)p^s-1}y^{2(p-1)p^s-1}
 \bigr)uv,\label{eq:ranktwo-mu}\\
 N_s&=x^{p(p-1)p^s-1}y^{(p^2-1)p^s-1}uv.
 \label{eq:ranktwo-nu}
\end{align}
If $r\geq s+2$ and $s\geq0$, put
\[
 H=(p-1)p^r,\qquad L=(p-1)p^s,\qquad C=(p-1)p^{s+1},
\]
and define
\begin{equation}\label{eq:ranktwo-separated}
 \Omega_{r,s}=
 \bigl(x^{L-1}y^{H-1}-x^{H-1}y^{L-1}
       +x^{C-1}y^{H+L-C-1}\bigr)uv.
\end{equation}
If $s\geq1$, put $D=(p-1)p^s$ and define
\begin{equation}\label{eq:ranktwo-odd}
 Z_s=x^{D-1}u+y^{D-1}v-x^{D-p+1}y^{p-2}v.
\end{equation}
All exponents in these formulas are nonnegative.

\begin{samepage}
\begin{thm}\label{thm:ranktwo-explicit-invariants}
Let $p$ be an odd prime and let $n\geq0$.  If $n$ occurs in the
following table, then the displayed singleton is a basis of
$\QH^n(2)^{G_2}$.  If $n$ does not occur in the table, then
$\QH^n(2)^{G_2}=0$.
\begin{center}
\renewcommand{\arraystretch}{1.35}
\begin{tabular}{@{}ll@{}}
\hline
Cohomological degree $n$ & Basis \\
\hline
$0$ & $\{[1]\}$ \\
$4(p-1)-1$ & $\{[R]\}$ \\
$2(p^2-p)p^s-2$, $s\geq0$ & $\{[\Lambda_s]\}$ \\
$2(p^2+p-2)p^s-2$, $s\geq0$ & $\{[M_s]\}$ \\
$2(p-1)(p^r+p^s)-2$, $r\geq s+2$, $s\geq0$
  & $\{[\Omega_{r,s}]\}$ \\
$2(p-1)p^s-1$, $s\geq1$ & $\{[Z_s]\}$ \\
$2(2p^2-p-1)p^s-2$, $s\geq0$ & $\{[N_s]\}$ \\
\hline
\end{tabular}
\end{center}
\end{thm}
\end{samepage}

We prepare two elementary facts needed to identify these representatives.
The first records coefficient functionals that survive passage to
indecomposables.  The second provides the particular reduction in
degree used in the proof of the $M_s$ formula.

\begin{lem}\label{lem:ranktwo-spike-functionals}
Let $a,b,d\geq0$ and suppose that each of $a+1,b+1,d+1$ is of the
form $cp^j$, with $1\leq c\leq p-1$ and $j\geq0$.
Then the coefficient of $x^ay^buv$ defines a linear functional on
$\QH^{2(a+b)+2}(2)$.  The coefficients of $x^du$ and $y^dv$
define linear functionals on $\QH^{2d+1}(2)$.
\end{lem}

\begin{proof}
The coefficient-functional argument in the proof of
Proposition~\ref{prop:spike-independence} shows that the coefficient
of $x^ay^b$ vanishes on polynomial hits.  No Bockstein enters the
top exterior component, so Proposition~\ref{prop:exterior-slices}
proves the assertion about $x^ay^buv$.

Only monomials of the form $x^e u$, with $e\ge0$, can contribute to the
coefficient of $x^du$ in an image of a positive reduced power: reduced
powers preserve the exterior factor and cannot decrease the $y$-exponent.
The contribution of each such monomial is therefore a rank-one reduced-power
coefficient and
vanishes by Theorem~\ref{thm:rank-one}.  The same reasoning applies
to $y^dv$.  Finally,
\[
 \beta(fuv)=xfv-yfu
\]
has neither a pure $x^du$ term nor a pure $y^dv$ term.
The exterior-degree-one formula in
Proposition~\ref{prop:exterior-slices} therefore gives the required
factorization through the quotient.
\end{proof}

Put $\kappa=\kappa_2$, with $\kappa_h$ defined
in~\eqref{eq:extraction-map-general}, and set
$\kappa(fuv)=\kappa(f)uv$. Its compatibility with reduced powers
is~\eqref{eq:extraction-powers-general}, and its equivariance
is given by Lemma~\ref{lem:extraction-equivariant}. The specified
rank-two degrees can also be compared directly using Crossley's
quotient bases.

\begin{lem}\label{lem:ranktwo-extraction}
Let $m_j=(p^2+p-2)p^j-2$ for $j\geq0$, and let $s\geq1$.
Then $\kappa$ induces a $G_2$-equivariant isomorphism
\[
 (\QP_2)_{m_s}\otimes\langle uv\rangle
 \longrightarrow
 (\QP_2)_{m_{s-1}}\otimes\langle uv\rangle.
\]
\end{lem}
\begin{proof}
The identity $m_s=pm_{s-1}+2(p-1)$ gives the degrees of
the induced map in~\eqref{eq:extraction-on-cohits}.
For $m_s=(p^2+p-2)p^s-2$, the specialization
$(i,j)=(p-1,p-3)$ of Crossley's basis
\cite[Theorem~2, Table~1, row~4]{Crossley1999Transfer} consists of
\begin{align*}
 &\left[x^{(k+1)p^s-1}
 y^{(p^2+p-3-k)p^s-1}uv\right],
 &&p-3\leq k\leq p-1,\\
 &\left[x^{(k+1)p^{s+1}-1}
 y^{((p-k)p-2)p^s-1}uv\right],
 &&1\leq k\leq p-1.
\end{align*}
For $s\geq1$, extraction sends each displayed basis vector to the
vector with the same index $k$ and with $s$ replaced by $s-1$.
It therefore sends a basis bijectively to a basis.  The induced
map is an isomorphism. Its equivariance is supplied by
Lemma~\ref{lem:extraction-equivariant}.
\end{proof}

\begin{proof}[Proof of Theorem~\ref{thm:ranktwo-explicit-invariants}]
Crossley~\cite[Corollary~1]{Crossley1999Transfer} states that the
invariant space has dimension one in exactly the degrees in the
table and is zero in all other degrees.  We must identify the
nonzero vector on each of these lines.  Let $T$ be the diagonal
subgroup of $G_2$.  By Proposition~\ref{prop:torus-block}, a
monomial $x^ay^buv$ is $T$-fixed precisely when
\begin{equation}\label{eq:ranktwo-torus-weight}
 a+1\equiv b+1\equiv0\pmod{p-1}.
\end{equation}
Every monomial in the top exterior bases of
\cite[Theorem~2 and Table~1]{Crossley1999Transfer} is a
$T$-eigenvector.
Consequently, the basis vectors satisfying
\eqref{eq:ranktwo-torus-weight} form a basis of the $T$-fixed
subspace.  The assertion in degree zero follows from
$\QH^0(2)=\Fp$.

In degree $4(p-1)-1$, take $t=2p-3$ in
\cite[Table~1, row~3]{Crossley1999Transfer}. Its index runs
from $t+2-p=p-1$ to $p-1$, so there is one basis vector,
\[
 [x^{p-2}y^{p-1}u+x^{p-1}y^{p-2}v].
\]
The Bockstein identity
\[
 \beta(x^{p-2}y^{p-2}uv)
 =x^{p-1}y^{p-2}v-x^{p-2}y^{p-1}u
\]
identifies this vector with $2[R]$.
Since $p$ is odd, $[R]$ is nonzero.  The full quotient is
one-dimensional and its invariant subspace has dimension one;
therefore $[R]$ is invariant.

For $n=2(p^2-p)p^s-2$, take $(i,j)=(p-3,p-1)$ in
\cite[Table~1, row~4]{Crossley1999Transfer}. The basis consists of
\begin{align*}
 &\left[x^{(k+1)p^s-1}
     y^{(p^2-p-k-1)p^s-1}uv\right],
 &&k=p-2,p-1,\\
 &\left[x^{(k+1)p^{s+1}-1}
     y^{p(p-k-2)p^s-1}uv\right],
 &&1\le k\le p-3.
\end{align*}
The second range is empty when $p=3$.
Since $p^s\equiv1\pmod{p-1}$, in either group the first
exponent plus one is congruent to $k+1$ modulo $p-1$.
The sum of the two exponents plus two is divisible by
$p-1$, so condition~\eqref{eq:ranktwo-torus-weight}
reduces to $k+1\equiv0\pmod{p-1}$.
The first range contains exactly one such index, $k=p-2$,
and the second contains none. At $k=p-2$ the exponents are
$(p-1)p^s-1$ and $(p-1)^2p^s-1$, giving $[\Lambda_s]$.
Thus the diagonal fixed subspace is the line spanned by
this nonzero basis vector. The general linear invariant
space has dimension one by
\cite[Corollary~1]{Crossley1999Transfer} and is contained
in that line, so the two spaces coincide.

For $n=2(p^2+p-2)p^s-2$, write
\[
 A_s=x^{(p-1)p^s-1}y^{(p^2-1)p^s-1}uv,\qquad
 B_s=x^{p(p-1)p^s-1}y^{2(p-1)p^s-1}uv.
\]
The specialization $(i,j)=(p-1,p-3)$ of
\cite[Table~1, row~4]{Crossley1999Transfer}, displayed in
Lemma~\ref{lem:ranktwo-extraction}, has exactly two
$T$-fixed vectors, namely $[A_s]$ and $[B_s]$.
Indeed, in each of its two groups condition
\eqref{eq:ranktwo-torus-weight} requires $k=p-2$.
Let $\sigma$ interchange $x,y$ and $u,v$; in particular,
$\sigma(uv)=-uv$.  In the case $s=0$, the Cartan formula gives
\[
 \Pp^1(x^{p^2-p-1}y^{p-2})
 =-x^{p^2-2}y^{p-2}
  -2x^{p^2-p-1}y^{2p-3}.
\]
Multiplication of this displayed identity by $uv$ is a reduced-power
image, because every positive reduced power vanishes on $uv$.
It follows that
\[
 \sigma[A_0]=2[B_0],\qquad
 \sigma[B_0]=\tfrac12[A_0],
\]
where the second equality follows by applying $\sigma$ to the
first.  For $s\geq1$, iterate the equivariant isomorphism of
Lemma~\ref{lem:ranktwo-extraction}.  Its $s$-fold iterate sends
$[A_s],[B_s]$ to $[A_0],[B_0]$, respectively; injectivity therefore
gives the same two identities with subscript $s$.
For a vector $a[A_s]+b[B_s]$, the interchange gives
$\tfrac12 b[A_s]+2a[B_s]$. Equality with the original
vector is equivalent to $b=2a$. The $\sigma$-fixed line
in the $T$-fixed space is therefore
\[
 \Fp\bigl([A_s]+2[B_s]\bigr)=\Fp[M_s].
\]
The nonzero $G_2$-fixed space, whose dimension is given by
\cite[Corollary~1]{Crossley1999Transfer}, is contained in
this line, so the two spaces are equal.  Independence of $[A_s]$ and $[B_s]$ also proves
$[M_s]\ne0$.

Next suppose $r\geq s+2$ and retain $H,L,C$ from
\eqref{eq:ranktwo-separated}.  Set
\[
 E=x^{L-1}y^{H-1}uv,\qquad
 F=x^{H-1}y^{L-1}uv,\qquad
 B=x^{C-1}y^{H+L-C-1}uv.
\]
The specialization $i=j=p-2$ in
\cite[Table~1, row~5]{Crossley1999Transfer} has the two
endpoint basis vectors $[E]$, $[F]$, together with
\[
 \left[x^{(k+1)p^{s+1}-1}
 y^{H+L-(k+1)p^{s+1}-1}uv\right],
 \qquad1\le k\le p-1.
\]
Both endpoint vectors satisfy
\eqref{eq:ranktwo-torus-weight}, since $H$ and $L$ are
multiples of $p-1$. For an indexed vector the first
exponent plus one is congruent to $k+1$ modulo $p-1$,
so diagonal invariance requires $k=p-2$ in this range.
The second exponent then has the required residue as
well, and the selected vector is $[B]$.
Consequently $[E],[F],[B]$ form a basis of the diagonal
fixed subspace.
By Lemma~\ref{lem:ranktwo-spike-functionals}, extraction of the
coefficients at $E$ and $F$ defines functionals on the quotient.
Write a $G_2$-invariant vector as
\[
 w=a[E]+b[F]+c[B].
\]
Since $\sigma E=-F$ and $\sigma F=-E$, comparison of the
$E$-coefficients in $\sigma w=w$ gives $b=-a$.
The monomial $\sigma B$ contributes to neither endpoint:
$C$ lies strictly between $L$ and $H$.

Consider the transvection
\[
 \tau(x)=x,\quad\tau(y)=y-x,\quad
 \tau(u)=u,\quad\tau(v)=v-u.
\]
It fixes $uv$.  The coefficients of $F$ in $\tau E$, $\tau F$,
and $\tau B$, respectively, are
\[
 (-1)^{H-L}\binom{H-1}{L-1},\qquad
 1,\qquad
 (-1)^{H-C}\binom{H+L-C-1}{L-1}.
\]
Both signs are positive, since $H,L,C$ are multiples of the even
integer $p-1$.  The base-$p$ expansion of $L-1$ has digits
$p-1$ in positions $0,\ldots,s-1$, digit $p-2$ in position $s$,
and zero higher digits.  The corresponding digits of $H-1$ are
$p-1$ through position $r-1$.  The binomial coefficient congruence
therefore gives
\[
 \binom{H-1}{L-1}=\binom{p-1}{p-2}=-1
 \quad\text{in }\Fp.
\]
Write $H+L-C=Kp^s$, where
\[
 K=(p-1)p^{r-s}-(p-1)^2>0,
 \qquad K\equiv p-1\pmod p.
\]
The digits of $Kp^s-1$ are $p-1$ below position $s$ and
$p-2$ in position $s$.  Since $L-1$ has no higher nonzero
digits, the second binomial coefficient is
\[
 \binom{H+L-C-1}{L-1}=1
 \quad\text{in }\Fp.
\]
Comparison of the $F$-coefficients in $\tau w=w$ now yields
$-a+b+c=b$, hence $c=a$.
Every $G_2$-invariant vector consequently lies on the line
\[
 \Fp\bigl([E]-[F]+[B]\bigr)=\Fp[\Omega_{r,s}].
\]
The vector displayed is nonzero, because $[E],[F],[B]$ belong
to a basis.  The dimension-one assertion in
\cite[Corollary~1]{Crossley1999Transfer} shows that this
line is precisely the invariant space.

For the odd family, let $D=(p-1)p^s$ with $s\geq1$ and put
\[
 X=x^{D-1}u,\qquad Y=y^{D-1}v,\qquad
 W=x^{D-p+1}y^{p-2}v.
\]
In~\cite[Table~1, row~6]{Crossley1999Transfer}, take $i=p-2$.
The two endpoint vectors are $[X]$ and $[Y]$.
An indexed vector
\[
 x^{D-1-k}y^ku+x^{D-k}y^{k-1}v,\qquad1\le k\le p-1,
\]
is multiplied by $\alpha^{D-k}\gamma^k$ under the substitution
$(x,u)\mapsto(\alpha x,\alpha u)$,
$(y,v)\mapsto(\gamma y,\gamma v)$.
It is therefore torus-fixed precisely when $k=p-1$.  For this value
the Bockstein identity
\[
 \beta(x^{D-p}y^{p-2}uv)
 =x^{D-p+1}y^{p-2}v-x^{D-p}y^{p-1}u
\]
identifies its class with $2[W]$.  Thus the $T$-fixed space has
basis $[X],[Y],[W]$.
The pure-endpoint coefficient functionals in
Lemma~\ref{lem:ranktwo-spike-functionals} apply, since
$D-1$ is a rank-one non-hit exponent.
If
$w=a[X]+b[Y]+c[W]$ is $G_2$-fixed, comparison of pure endpoints
under $\sigma$ gives $a=b$; neither $W$ nor $\sigma W$ has a
pure-endpoint term.
Under the same transvection $\tau$, the coefficient of $x^{D-1}u$
in $\tau X$ is $1$.
In $\tau Y=(y-x)^{D-1}(v-u)$ this coefficient is
$-(-1)^{D-1}=1$, because $D$ is even.
In
\[
 \tau W=x^{D-p+1}(y-x)^{p-2}(v-u)
\]
it is $-(-1)^{p-2}=1$, because $p$ is odd.
The equality $\tau w=w$ therefore implies $a+b+c=a$, so
$c=-b$.  The invariant space is contained in
\[
 \Fp\bigl([X]+[Y]-[W]\bigr)=\Fp[Z_s].
\]
This is a nonzero line by independence of $[X],[Y],[W]$.
The dimension-one result of
\cite[Corollary~1]{Crossley1999Transfer} proves equality.

Finally, for $n=2(2p^2-p-1)p^s-2$, take $i=j=p-2$ in
\cite[Table~1, row~7]{Crossley1999Transfer}.
The index range $i\le k\le j$ contains only $k=p-2$.
The first exponent of its basis vector is
$(k+1)p^{s+1}-1=p(p-1)p^s-1$, and the second is
\[
 (p^2+(i-k-1)p+j+1)p^s-1=(p^2-1)p^s-1.
\]
Thus the sole basis vector is $[N_s]$.
The full quotient is one-dimensional, and its invariant
subspace has dimension one by
\cite[Corollary~1]{Crossley1999Transfer}. The two spaces
therefore coincide, proving the last row.
All degree and exterior-degree assertions follow directly from
the exponents in \eqref{eq:ranktwo-rho}--\eqref{eq:ranktwo-odd}.
\end{proof}

\subsection{Two explicit degrees}

Two specializations illustrate how the dimension of the full quotient
differs from the dimension of its invariant subspace.  The first
uses the separated family in Theorem~\ref{thm:ranktwo-explicit-invariants};
the second lies below the first possible polynomial hit and is
decided by the diagonal action alone.

\begin{corl}\label{cor:ranktwo-p3-degree38}
Let $p=3$.  Then $\QH^{38}(2)^{G_2}$ is one-dimensional with basis
\[
 \bigl\{[(xy^{17}-x^{17}y+x^5y^{13})uv]\bigr\}.
\]
\end{corl}
\begin{proof}
Take $r=2$ and $s=0$ in
Theorem~\ref{thm:ranktwo-explicit-invariants}.  Then
$H=18$, $L=2$, and $C=6$, so
\eqref{eq:ranktwo-separated} is the displayed representative.
Its polynomial degree is $18$ and its exterior degree is $2$,
giving cohomological degree $38$.
\end{proof}

\begin{prop}\label{prop:ranktwo-p13-degree22}
Let $p=13$.  Then $\QH^{22}(2)$ has the basis
\[
 \bigl\{[x^ay^{10-a}uv]:0\leq a\leq10\bigr\}
\]
and dimension $11$, whereas $\QH^{22}(2)^{G_2}=0$.
\end{prop}
\begin{proof}
Proposition~\ref{prop:exterior-slices} identifies the positive
even-degree quotient with
$(\QP_2)_{10}\otimes\langle uv\rangle$.
For every $s\ge1$, the homogeneous component $(P_2)_{10-12s}$ is zero, since
$10-12s<0$. Thus no image of a positive reduced power can contribute in
polynomial degree $10$. Lemma~\ref{lem:hit-images} gives $(\Hit_2)_{10}=0$,
and consequently $(\QP_2)_{10}=(P_2)_{10}$.
The eleven displayed monomials therefore form a basis.
By Proposition~\ref{prop:torus-block}, a torus-fixed basis vector
would require both $a$ and $10-a$ to be congruent to $11$
modulo $12$.  Neither exponent, which lies between $0$ and $10$,
has that residue.  The torus-fixed subspace is zero, and so is its
$G_2$-fixed subspace.
\end{proof}

\medskip
\noindent
%The finite computations in Subsections~\ref{subsec:polynomial-preliminaries}
%and~\ref{subsec:rankthree-calculations}
%were performed in C++17; the output data are available at Zenodo,
%DOI: \ZenodoDOI.

\end{document}